\documentclass[review,onefignum,onetabnum]{siamart171218}

\usepackage{enumerate}
\usepackage{amsmath}
\usepackage{mathrsfs}
\usepackage{xcolor}	
\usepackage{amssymb}
\usepackage{bm}
\usepackage{graphicx}
\usepackage{subcaption}
\usepackage{adjustbox}
\usepackage{soul}
\usepackage{listings}
\usepackage{array}  
\usepackage{mathtools}
\usepackage{abraces}
\usepackage{algorithmic,float}
\usepackage{tikz}
\usepackage{pgfplots}
\usepackage{subcaption}

\pgfplotsset{compat=1.17}
\usepackage{mdframed}
\newmdenv[linecolor=green!50!black, fontcolor=green!50!black, backgroundcolor=green!20, linewidth=2pt, roundcorner=10pt]{gnote}

\definecolor{amber}{rgb}{1.0, 0.49, 0.0}
\definecolor{darkpastelgreen}{rgb}{0.01, 0.75, 0.24}
\definecolor{darkred}{rgb}{0.64, 0.0, 0.0}
\definecolor{amberhl}{rgb}{1.0, 0.75, 0.0}

\newmdenv[linecolor=blue!50!black, fontcolor=blue!50!black, backgroundcolor=blue!20, linewidth=2pt, roundcorner=10pt]{anote}

\newtheorem{assumption}{Assumption}

\newcolumntype{L}{>{$}l<{$}}

\newcommand{\fl}{\text{flexp}}
\newcommand{\sint}{\texttt{short int}}
\newcommand{\qpr}{\texttt{quad}}
\newcommand{\dpr}{\texttt{double}}
\newcommand{\spr}{\texttt{single}}
\newcommand{\hpr}{\texttt{half}}
\newcommand{\ppr}{\texttt{perforate}}
\newcommand{\bsize}{M_k}
\newcommand{\qdot}{\textbf{qdot}}

\newcommand{\secref}[1]{Section~\ref{#1}}

\headers{Qdot: Quantized Dot Product for Approximate HPC}{J. Diffenderfer, D. Osei-Kuffuor, H. Menon}

\title{Qdot: Quantized Dot Product Kernel for Approximate High Performance Computing\thanks{Submitted to the editors \today.\\ This work was performed under the auspices of the U.S. Department of Energy by Lawrence Livermore National Laboratory under Contract DE-AC52-07NA27344. Work at LLNL was funded by the Laboratory Directed Research and Development Program under project tracking code 20-ERD-043, LLNL-JRNL-820357-DRAFT.}}
\author{James Diffenderfer\thanks{Lawrence Livermore National Laboratory,
    Livermore, CA. email: (diffenderfer2, oseikuffuor1, gopalakrishn1)@llnl.gov}
  \and Daniel Osei-Kuffuor\footnotemark[2]
  \and Harshitha Menon\footnotemark[2]
}

\ifpdf
\hypersetup{
  pdftitle={Qdot: Quantized Dot Product Kernel for Approximate High Performance Computing},
  pdfauthor={J. Diffenderfer, D. Osei-Kuffuor, H. Menon}
}
\fi

\ifpdf
\hypersetup{
  pdftitle={Qdot: Quantized Dot Product Kernel for Approximate High Performance Computing},
  pdfauthor={J. Diffenderfer, D. Osei-Kuffuor, H. Menon}
}
\fi
\RequirePackage[normalem]{ulem} %DIF PREAMBLE
\RequirePackage{color}\definecolor{RED}{rgb}{1,0,0}\definecolor{BLUE}{rgb}{0,0,1}\definecolor{purple}{rgb}{0.5,0,0.5} %DIF PREAMBLE

\begin{document}

\maketitle

% REQUIRED
\begin{abstract}
Approximate computing techniques have been successful in reducing computation and power costs in several domains. 
However, error sensitive applications in high-performance computing are unable to benefit from %many approximate computing techniques as 
existing approximate computing strategies %\red{(e.g. loop perforation)} 
that are not developed with guaranteed error bounds. While approximate computing techniques can be developed for individual high-performance computing applications by domain specialists, %\red{(do we have any examples of this? I imagine it is true so I put it in here for now but we can remove it)}, 
this often requires additional theoretical analysis and potentially extensive software modification. Hence, the development of low-level error-bounded approximate computing strategies that can be introduced into any high-performance computing application without requiring additional analysis or significant software alterations is desirable. In this paper, we provide a contribution in this direction by proposing a general framework for designing error-bounded approximate computing strategies and apply it to the dot product kernel to develop \textbf{qdot} --- an error-bounded approximate dot product kernel. Following the introduction of \textbf{qdot}, we perform a theoretical analysis that yields a deterministic bound on the relative approximation error introduced by \textbf{qdot}. Empirical tests are performed to illustrate the tightness of the derived error bound and to demonstrate the effectiveness of \textbf{qdot} on a synthetic dataset, as well as two scientific benchmarks -- Conjugate Gradient (CG) and the Power method. In particular, using \textbf{qdot} for the dot products in CG  can result in a majority of components being perforated or quantized to half precision without increasing the iteration count required for convergence to the same solution as CG using a double precision dot product.
\end{abstract}

% REQUIRED

% REQUIRED

% Body
\section {Introduction}
\label{sec:introduction}
Scientific application development is rife with approximation techniques. From spatial and temporal discretization schemes to numerical solution strategies, approximation techniques are quite ubiquitous in scientific computing. In the context of high performance computing (HPC), approximate computing strategies have been typically employed to (a) reduce computational complexity - for example polynomial approximations of functional forms \cite{trefethen97} - and (b) make better use of computing resources - for example iterative refinement on Tensor Core Units \cite{haidar2020mixed}. For all of these strategies in HPC, the overarching goal is to safeguard the accuracy of the desired quantity of interest, although in some cases, the additional benefit of preserving the convergence properties of the underlying algorithm may be realized \cite{dembo1982inexact}.

Outside these traditional approximation strategies, other methods have been developed to reduce computational complexity by economizing computation, that is, computing only when necessary. Strategies such as loop perforation \cite{mittal2016survey} are used to skip specific iterations of a loop in a computational kernel in order to reduce the cost. The choice of which iterations to skip may be random or prescribed by some predetermined stride. Another common strategy known as function memoization \cite{michie1968memo} employs a look-up table to store computed entries of computationally expensive kernels. Subsequent calls to the kernel may then be substituted by retrieving an appropriate entry in the look-up table. While these approximate strategies are more efficient than their exact counterparts, they are not designed with closed form bounds for error introduced by the approximation. This lack of error control precludes them from being considered as is in HPC applications. Significance-aware strategies \cite{dinos2015model} attempt to mitigate this by accepting an approximation within some prescribed tolerance. 

Another class of approximate computing strategies that has gained a lot of interest in recent years is precision-based approximate computing \cite{lindstrom2014zfp,menon2018,gonzalez2013precimonious}. Here, the notion of economizing computation is based on computing only on important information, that is, making use of only        the important bits. The approach typically follows identifying regions of code or portions of an algorithm where low precision arithmetic or storage may be utilized, subject to some error tolerance. These strategies can lead to efficient mixed-precision algorithms that can efficiently utilize current and emerging HPC hardware. 

Existing approximate computing strategies, including the tolerance-based strategies, lack a priori error analysis on their impact on the convergence and the quality of the quantity of interest in the kernels or algorithms they approximate. For most of these methods, this means they need to be deployed as a preprocessing tool to first identify the parameters whose approximation satisfy the required error tolerance, before running the algorithm with these parameters. In this paper, we present a general framework for developing error-bounded approximate computing techniques. We define key terminology that may be used to evaluate other approximate computing strategies, in terms of their effectiveness as an approximation technique and their computational efficiency. We demonstrate this framework by designing an error-bounded quantized dot product algorithm (\qdot) as an approximation technique for a dot product kernel. This new approximation strategy can thus be deployed at run time since its error bound is known and can be manipulated at run time. It is worth noting that hardware-based approximate computing strategies have also been developed for specialized hardware such as FPGA's to exploit efficiency by computing in the presence of hardware faults \cite{mittal2016survey,sinha2016low}. These techniques are beyond the scope of this paper and, as such, are not addressed here.

The paper is organized as follows. In~\secref{sec:error-bounded-ac}, we discuss a general framework for designing error-bounded approximate computing kernels and formalize terminologies for measuring the quality of approximate computing kernels. In~\secref{sec:q-dot}, we apply this framework to the dot product kernel to develop \qdot. In particular, we provide motivation and pseudocode for \qdot \ together with theoretical analysis capturing the relative approximation error introduced by \qdot. In~\secref{sec:results}, we test our implementation of \qdot \ to empirically verify the error bound as well as demonstrate scientific computing applications with dot products that are amenable to extreme levels of quantization and perforation. Specifically, we integrate \qdot \ with CG and Power method %, and Subspace Iteration 
implementations to demonstrate the limits of quantization and perforation that can be achieved without affecting the convergence of these applications. We now provide notation used throughout the paper.

\subsection{Notation} \label{sec:notation}
For an integer $n \geq 1$, we write $[n]$ to denote the set $\{ i \in \mathbb{N} : 1 \leq i \leq n \}$. We denote scalar values by italicized letters, $x$, and vectors by boldface letters, such as $\bm{x}$. Similarly, we denote scalar valued functions by italicized letters, $f$, and vector-valued functions by boldfaced letters, $\bm{f}$. We write $\odot$ to denote the Hadamard, or component-wise, product of two vectors so that $( \bm{x} \odot \bm{y} )_i = x_i y_i$. We write $\otimes$ to denote floating point multiplication. If $\mathcal{S}$ is a subset of indices of $\bm{x}$, then we write $\bm{x}_{\mathcal{S}}$ to denote the subvector of $\bm{x}$ with indices in $\mathcal{S}$. We use the function $\fl : \mathbb{R} \to \mathbb{Z}$ defined by
\begin{align}
\fl(x) := \left\{
   \begin{array}{lcl}
      \lfloor \log_2 (|x|) \rfloor &:& |x| \geq 1 \\
      \lceil \log_2 (|x|) \rceil &:& |x| < 1
   \end{array}
   \right. . \label{def:exp2}
\end{align}
When $x$ is representable as a floating-point number, $\fl(x)$ is the unbiased exponent of $x$. For $\bm{x} \in \mathbb{R}^n$ we naturally extend this notation by writing $\fl(\bm{x})$ to denote the vector with $i$th component $\fl(x_i)$. For $\bm{x} \in \mathbb{R}^n$ write $e_{\max}(\bm{x})$ and $e_{\min}(\bm{x})$ to denote the maximum and minimum value of $\fl$ over all components of $\bm{x}$. These can be defined mathematically as
\begin{align}
    e_{\max}(\bm{x}) := \max_{1 \leq i \leq n} \ \fl(x_i) \ \quad \text{and} \quad \
    e_{\min}(\bm{x}) := \min_{1 \leq i \leq n} \ \fl(x_i).
\end{align}

\section{Error Bounded Approximate Kernels}
\label{sec:error-bounded-ac}
Suppose we are given a kernel described as a function $\bm{k}: \mathcal{D} \to \mathcal{R}$, where the domain and range of $\bm{k}$ are normed vector spaces.
Our goal is to develop a parametric approximation of $\bm{k}$ for which suitable parameter choices result in bounded error approximations to the output of $\bm{k}$. Depending on the complexity of the original kernel $\bm{k}$, the level of parameterization required to achieve bounded error in the approximate kernel could easily become non-trivial. Hence, in order to promote usability of any approximate kernel with guaranteed error bounds, we also propose that they should be accompanied by a function for identifying parameter choices that ensure the error will be bounded at a specified tolerance. 
%\purple{We now mathematically formalize this approach. Note that we will write $\mathcal{P}$ to denote the set of all possible parameter configurations for the approximate kernel.} 

Let $\mathcal{P}$ denote the set of all possible parameter configurations for the approximate kernel, then we can define an approximate kernel $\bm{\tilde{k}}: \mathcal{D} \times \mathcal{P} \to \mathcal{R}$ and a parameter selection function $\Theta: \mathcal{D} \times \mathbb{R} \to  \mathcal{P}$ such that for any input $\bm{X} \in \mathcal{D}$ and approximation error tolerance $\varepsilon > 0$ we have a bound on the relative approximation error given by
\begin{align}
\| \bm{k} (\bm{X}) - \bm{\tilde{k}} (\bm{X}; \Theta (\bm{X}, \varepsilon)) \| 
&\leq \varepsilon \| \bm{k} (\bm{X}) \|. \label{eq:bounded-approx}
\end{align}

\noindent %\red{As we are defining the approximate kernel $\bm{\tilde{k}}$ to be a parameterized function, there are many potential choices for parameter settings to use within the approximate kernel. The parameterization should be defined so that $\bm{\tilde{k}}$ can introduce varying amounts of approximation by choosing different parameter settings. Hence, depending on the use case and the level of approximation an application can handle, certain parameter settings can corrupt an application if too much approximation is introduced. Hence, the function $\Theta$ is responsible for ensuring that the error due to approximation is controllable in any setting.}
It is important to formalize the roles that $\bm{\tilde{k}}$ and $\Theta$ play in achieving approximation in practice. To this end, we establish the following definitions that emphasize the design goals for error-bounded approximate computing kernels and simplify our discussion of them.

\begin{definition}
Let $\bm{k}: \mathcal{D} \to \mathcal{R}$ be a kernel, $\bm{\tilde{k}}: \mathcal{D} \times \mathcal{P} \to \mathcal{R}$ be a parameterized approximation of $\bm{k}$, and $\Theta: \mathcal{D} \times \mathbb{R} \to  \mathcal{P}$ a parameter selection function for $\bm{\tilde{k}}$.
Let $\varepsilon > 0$ and $\bm{X} \in \mathcal{D}$ be given. 

The \textbf{approximation potential} of $\bm{\tilde{k}}$ is how close the relative approximation error can be made to the tolerance $\varepsilon$ for any choice of parameter $\theta \in \mathcal{P}$. When $\| \bm{k}(\bm{X}) \| \neq 0$ the approximation potential can be quantified by
\begin{align}
\varepsilon - \frac{\| \bm{k} (\bm{X}) - \bm{\tilde{k}} (\bm{X}; \theta^*) \|}{\| \bm{k} (\bm{X}) \|} \label{eq:approx-potential}
\end{align} 
where $\theta^*$ is a solution to the optimization problem
\begin{align}
\begin{array}{cc}
\displaystyle \min_{\theta \in \mathcal{P}} & \varepsilon \| \bm{k} (\bm{X}) \| - \| \bm{k} (\bm{X}) - \bm{\tilde{k}} (\bm{x}; \theta) \| \\
\text{s.t.} & \| \bm{k} (\bm{X}) - \bm{\tilde{k}} (\bm{X}; \theta) \| \leq \varepsilon \| \bm{k} (\bm{X}) \|
\end{array} \label{prob:potential}
\end{align}

The \textbf{approximation effectiveness} of $\bm{\tilde{k}}$ and $\Theta$ is how close the relative approximation error is to the tolerance $\varepsilon$ when using the parameter $\Theta(\bm{X}, \varepsilon)$ given by the parameter selection function. When $\| \bm{k}(\bm{X}) \| \neq 0$ the approximation effectiveness can be quantified by
\begin{align}
\varepsilon - \frac{\| \bm{k} (\bm{X}) - \bm{\tilde{k}} (\bm{X}; \Theta (\bm{X}, \varepsilon)) \|}{\| \bm{k} (\bm{X}) \|}. \label{eq:approx-effect}
\end{align} 
If the relative approximation error exceeds the tolerance $\varepsilon$ with the selected parameter $\Theta (\bm{X}, \varepsilon)$, we say that the approximation is ineffective.

The \textbf{approximation efficiency} of $\Theta$ is the ratio of the time spent computing the original kernel, $\bm{k}(\bm{X})$, to the time spent selecting parameters $\Theta (\bm{X}, \varepsilon)$ and running the original kernel. That is,
\begin{align}
    \frac{\text{Asymptotic Runtime of } \bm{k}(\bm{X})}{\text{Asymptotic Runtime of } \Theta(\bm{X}, \varepsilon) + \text{Asymptotic Runtime of } \bm{k}(\bm{X})}. \label{eq:approx-efficiency}
\end{align}
The approximation efficiency takes on values in the interval $(0,1)$. Values close to 0 indicate that more time will be required to identify the parameters (inefficient) while values close to 1 indicate that the time spent identifying parameters is fast compared to evaluating the original kernel (efficient).

The \textbf{approximation speedup} of $\bm{\tilde{k}}$ is the ratio of the time spent computing the original kernel, $\bm{k}(\bm{x})$, to the time spent computing the approximate kernel. That is,
\begin{align}
    \frac{\text{Runtime of } \bm{k}(\bm{X})}{\text{Runtime of } \bm{\tilde{k}}(\bm{X}, \Theta(\bm{X}, \varepsilon))}. \label{eq:approx-speedup}
\end{align}
Note that for this definition we do not include the runtime of $\Theta(\bm{X}, \varepsilon)$ as this is to serve as a measure of the maximum possible speedup achievable by the approximate kernel $\bm{\tilde{k}}$.
\end{definition}
%\purple{With these definitions established, we now highlight how $\bm{\tilde{k}}$ and $\Theta$ affect the approximation potential, effectiveness, efficiency, and speedup of the approximate kernel.} 
%\red{We note that the efficiency and speedup are not designed to be numbers for comparing approximate kernels against each other. Instead, they serve to indicate whether the approximate kernel and parameter selection algorithm are usable in an application at runtime, or if it is more appropriate to use the parameter selection algorithm as a preprocessing tool so that the application will only require computation of the approximate kernel at runtime using the parameters determined from preprocessing.}

The design of the parametric approximate kernel $\bm{\tilde{k}}$, in particular the complexity of the parameterization, 
is solely responsible for determining the approximation potential and approximation speedup. The potential is more of a theoretical value, since an optimal parameter, $\theta^*$, will often be difficult to obtain for most $\bm{\tilde{k}}$. We can think of the approximation potential to be the best case approximation achievable by $\bm{\tilde{k}}$, provided we have unlimited resources to identify the best choice of parameters. The approximation speedup, on the other hand, indicates how fast the computation of $\bm{\tilde{k}}$ is compared to the original kernel, $\bm{k}$. It does not include the cost of selecting parameters, and one can consider this as the speedup in executing $\bm{\tilde{k}}$, if the parameter selection function, $\Theta$, is only used as a preprocessing tool. 
%\purple{In contrast to a measure of global speedup, which would measure the time to select parameters and run the approximate kernel, approximation speedup provides a measure of how much speedup the computation of $\bm{\tilde{k}}$ provides without selecting parameters. This corresponds to the best possible speedup $\bm{\tilde{k}}$ could provide, or the cost of executing $\bm{\tilde{k}}$ if the parameter selection function, $\Theta$, is only used as a preprocessing tool. }
%\red{Hence, approximation speedup acts much like an upper bound on the speedup that $\bm{\tilde{k}}$ can achieve for a given parameter choice, thereby informing whether the approximation technique is capable of reducing execution time of the kernel.}
Values greater than 1 indicate that $\bm{\tilde{k}}$ provides a practical speedup over the original kernel.

On the other hand, the design of the parameter selection function $\Theta$ determines the approximation effectiveness and approximation efficiency. We can interpret the approximation effectiveness to be the level of approximation actually realized by replacing $\bm{k}(\bm{X})$ with $\bm{\tilde{k}} (\bm{X}; \Theta (\bm{X}, \varepsilon))$ in an application. Hence, $\Theta$ has the important role of identifying a parameter choice for computing accurately and tightly bounded error in an application containing the kernel $\bm{k}$.
An approximate kernel will have low approximation efficiency when the asymptotic time to select parameters is much longer than the time to run the original kernel. 
%\red{As the importance of approximation effectiveness and efficiency can be dependent from application to application, the design and choice of an approximate kernel is dependent on the needs of the application.} 
In order to have a error-bounded kernel with high approximation effectiveness, the user will likely have to sacrifice some level of efficiency. The approximation efficiency defined in (\ref{eq:approx-efficiency}) provides an indication of the overhead in the parameter selection phase required for the design of the approximate kernel $\bm{\tilde{k}}$, with respect to the cost of executing the original kernel $\bm{k}$. This overhead can be defined as $\displaystyle \frac{\text{Asymptotic Runtime of } \Theta(\bm{X}, \varepsilon)}{\text{Asymptotic Runtime of } \bm{k}(\bm{X})}$ so that (\ref{eq:approx-efficiency}) is equivalent to
%Note that while we could make the case to include the runtime of $\bm{\tilde{k}} (\bm{x}; \Theta (\bm{x}, \varepsilon))$ in (\ref{eq:approx-efficiency}), this value is designed only to measure the cost of selecting parameters required for the approximate kernel to have bounded error. Thus, approximation efficiency serves to capture the cost of bounding the relative approximation error when running $\bm{\tilde{k}}$ and, hence, $\Theta$ is the determining factor in efficiency of the approximation. The definition of approximation efficiency is motivated by parallel efficiency \cite{grama2003introduction} of a parallel algorithm which measures the fraction of time for which a processor is doing useful work. %While we could define the efficiency as $\displaystyle \frac{\text{Asymptotic Runtime of } \Theta(\bm{x}, \varepsilon)}{\text{Asymptotic Runtime of } \bm{k}(\bm{x})}$, 
%Following from this motivation, the value in (\ref{eq:approx-efficiency}) is equivalent to 
%$\displaystyle \frac{1}{1 + \frac{\text{Asymptotic Runtime of } \Theta(\bm{x}, \varepsilon)}{\text{Asymptotic Runtime of } \bm{k}(\bm{x})}}$ 
$\left(1 + \frac{\text{Asymptotic Runtime of } \Theta(\bm{X}, \varepsilon)}{\text{Asymptotic Runtime of } \bm{k}(\bm{X})} \right)^{-1}$, which is bounded in the interval $(0,1)$. This formulation of approximation efficiency is motivated by the definition of parallel efficiency \cite{grama2003introduction} of a parallel algorithm, which measures the fraction of time for which a processor is doing useful work. Approximation efficiency values close to 1 indicate that the approximate kernel will spend most of the time doing useful work. 
%\red{In particular, the efficiency should be greater than $1/2$ and closer to 1 for error-bounded approximation techniques used at runtime.} 
Efficiency values less than $1/2$ indicate that $\Theta$ will not be performant enough for runtime-based error-bounded approximation. 

The design of the approximate kernel and the parameter selection function go hand-in-hand to form an application that will be effective, efficient, and provide speedup in practice. A well-designed approximate kernel with a poor parameter selection function will result in poor approximation in practice, since the effectiveness of the approximation may not reach the potential. On the other hand, a poorly designed approximate kernel will typically result in ineffective approximation regardless of the design of $\Theta$, since the approximation potential may be limited. Further, it is important to note that an over-parameterized approximate kernel may result in difficulty defining efficient parameter selection functions. 
%\red{These observations provide general guidelines for designing approximate kernels that work well in practice.}

%\red{We now briefly discuss choosing a suitable parameter selection function. Then the choice for $\Theta$ that yields approximation effectiveness equal to approximation potential is taking $\Theta (\bm{X}, \varepsilon)$ to be a global minimizer of problem (\ref{prob:potential}). However, for the majority of kernels worth approximating this definition of $\Theta$ will likely be very inefficient or intractable. As there will often be many approaches that can be used to select parameters, the definition of $\Theta$ should be made in a manner most suitable for the application. That is, the parameter selection function should be designed keeping in mind if the approximation benefits most from effectiveness, efficiency, or a balance of both.}

Based on our theoretical formulation of the problem, we summarize our framework for designing error bounded approximate kernels in two phases:
\begin{enumerate}
    \item \textit{Parameter Selection Phase}: Define a function or algorithm, $\Theta$, that takes a kernel input $\bm{X}$ and an error tolerance $\varepsilon$ and returns approximate kernel parameters that ensure that the relative error of the approximation $\bm{\tilde{k}}$ evaluated on the input $\bm{X}$ is below the tolerance $\varepsilon$. %\red{$\Theta$ should be designed to take into account desired levels of effectiveness and efficiency for the application.}
    \item \textit{Computation Phase}: Use parameters returned by $\Theta$ in approximate kernel with input $\bm{X}$ and error tolerance $\varepsilon$. 
\end{enumerate}
It is clear from this general framework that the cost of ensuring bounded error in the approximate kernel is based on the approximation effectiveness and efficiency of $\Theta$. %and data preprocessing phases. 
%However, without these steps the kernel will be unable to ensure the approximation error is below the user's required tolerance. 
%In light of this, the goal of the information collection %and data preprocessing phases 
%phase should be to collect the minimal amount of meaningful information that can be used to bound the approximation error while still providing benefits from approximation in the computation phase. 

%The parameter selection phase is any process for specifying parameters for the approximate kernel. We note that many parameter selection algorithms are available and have been used in existing approximate computing techniques. In our discussion, we are only interested in parameter selection phases that guarantee the error tolerance is satisfied, whenever possible. This results in designing a parameter selection phase that (a) is input and tolerance dependent and (b) yields a parameter selection for which the approximation error is bounded. 

Before applying this general framework to design an approximate kernel for the dot product, we briefly consider two examples of existing approximate computing techniques and demonstrate how they closely match this framework. First, loop perforation \cite{mittal2016survey} is an approximate computing technique where iterates in a loop are skipped during execution of the code. For a loop of length N, the parameters for this approximate kernel are subsets of $[N]$ that indicate which iterates of the loop should be skipped. Popular choices for these parameters are selecting random subsets of $[N]$ or the last $k$ iterates of the loop independent of the input $\bm{X}$. These choices are often made for efficiency as, in this case, the parameter selection function has asymptotic runtime of $O(1)$, while the original kernel has runtime $O(N)$. Hence, the parameter selection function is very efficient. However, for many inputs the approximation error will be unbounded, making the kernel ineffective. %these choices of parameters for loop perforation cannot ensure a guaranteed bound on the approximation error introduced. 
%\red{As a fitting example, consider loop perforation applied to the dot product given by the kernel $\bm{k} (\bm{x}, \bm{y}) = \sum_{i=1}^N x_i y_i$. For simplicity take $N = 2$ and let $\theta = \{ 2 \}$ denote the set of indices to be skipped from the summation used in computing the dot product. Taking $\bm{x} = [ 1, 1e8]$ it follows that $\bm{k} (\bm{x}, \bm{x}) = 1 + 1e16$ and $\bm{\tilde{k}}(\bm{x}, \bm{x}; \theta) = 1$. Hence, the resulting approximation is ineffective as the error is unbounded and thereby has the potential to be unbounded for additional choices for the input. This toy example demonstrates that a parameter selection function that fails to take information about the input data into account carries the risk of being ineffective for certain inputs.} 
For applications that are error-sensitive, which is the target of this paper, these types of parameter selection functions are infeasible. 
Second, approximate computing strategies such as those based on precision tuning and significance-aware methods can be effective, since they are designed to satisfy some error threshold. However, the parameter selection phase of these strategies can be quite expensive, rendering them inefficient for approximate computing at run time. Hence, they are typically run as an offline analysis tool on a representative set of inputs prior to the actual application run. Examples of techniques that make use of such an approach are Precimonious \cite{gonzalez2013precimonious}, FPTuner \cite{chiang2017rigorous}, and ADAPT \cite{menon2018}. %With respect to loop perforation, loop perforation space exploration techniques exist which can provide guidance on selecting appropriate loop perforation parameters \cite{rinard2011managing}, however, they may not be suitable for HPC applications. 

By framing the error-bounded approximate kernels in this two phase approach, it also provides a clear technique for making use of existing unbounded approximate computing methods. In particular, to make unbounded methods feasible in error sensitive applications, we simply need to replace the existing parameter selection function with one that ensures the approximation error is bounded. We now use this general framework to design and test an error bounded approximate kernel for the dot product. 

\section{QDOT: Quantized Dot Product Kernel} \label{sec:q-dot}
%\purple{In this section, we provide motivation and pseudocode for the approximate dot product kernel \textbf{qdot}. Our primary goal in the design of \textbf{qdot} is to develop an approximate kernel that is effective. In this work, we are interested in understanding how much approximation can be introduced into scientific computing applications without compromising the fidelity of the quantities of interest computed by the application. }
The design of the approximate dot product kernel, \textbf{qdot}, is motivated by our interest in understanding how much approximation can be tolerated by scientific computing applications, without compromising the fidelity of the quantities of interest computed by the application. Enabling easy integration of \textbf{qdot} into existing applications and providing tuneable error bounds, allows for easy evaluation of the approximate kernel in various scientific computing applications. To this end, our primary goal in the design of \textbf{qdot} is to develop an approximate kernel that is effective.
%\green{, with approximation efficiency as a secondary goal}.
%\purple{Since the dot product is a very low level kernel in many scientific applications, this approximate kernel provides an investigation and insight into how much approximation can be achieved by approximating a low level operation in scientific applications. By making \textbf{qdot} easy to integrate into existing applications and providing tuneable error bounds will allow for easy testing in various scientific computing applications. As it is unclear a priori how much approximation can be introduced into this low level kernel in HPC application, the design of \textbf{qdot} will emphasize approximation effectiveness. By focusing on approximation effectiveness, we can consider various levels of approximation and ascertain the sensitivity of each application to approximation. Thus, for this work, approximation efficiency is a secondary goal.} 
If the application can handle large amounts of approximation, then \textbf{qdot} can be used as starting point to develop a more efficient approach.
%We focus on approximation effectiveness for two reasons. Firstly, we wish to determine the impact of approximating low level kernels like the dot product, on scientific applications that rely on such kernels. First, as the dot product is a very low level kernel, we want to see how much approximation can be achieved by approximating a low-level operation in scientific computing applications. Secondly, we are unsure how much approximation can be introduced into scientific computing applications without negative effects. In particular, it is not clear if even small levels of approximation introduced by dot product operations will effectively break our error-sensitive applications. As such, making use of an approximate kernel with high approximation effectiveness will allow us to test tighter levels of approximation in our applications and identify their sensitivity. As such, efficiency is secondary goal. If the applications can handle large amounts of approximation then \textbf{qdot} can be used as starting point to develop a more efficient approach. 

\textbf{qdot} can be described as a combination of variable precision computing and loop perforation, applied to a full precision dot product kernel. At a high-level, \textbf{qdot} takes the vectors $\bm{x}, \bm{y} \in \mathbb{R}^n$ and partitions $\bm{x}$ and $\bm{y}$ into sub-vectors of different precisions. Then, for each sub-vector we compute the dot product in its corresponding precision. The dot products from each sub-vector are then accumulated in the same precision as the input, resulting in an approximation of the full-precision dot product kernel. 
Hence, our approximate algorithm requires two sets of parameters:
\begin{enumerate}
    \item Disjoint bins $\{ B_k \}_{k = 1}^p$ specifying indices for $p$ sub-vectors of $\bm{x}$ and $\bm{y}$, where $\displaystyle \bigcup_{k=1}^p B_k = [n]$.
    \item Bin precisions $\{ \pi_k \}_{k = 1}^p$, where $\pi_k$ specifies the precision for the sub-vectors $\bm{x}_{B_k}$ and $\bm{y}_{B_k}$.
\end{enumerate}
For our implementation, the precisions $\pi_k$ can take on IEEE Floating Point Precisions \hpr, \spr, \dpr, \qpr \ as well as \ppr. Note that \ppr \ denotes all indices that are to be excluded from the computation of the quantized dot product. In essence, we can consider \ppr \ to be components that have been quantized to a 0-bit representation.  Hence, a parameter selection algorithm for \textbf{qdot} should identify bins and the associated precisions for each bin, so that the approximate dot product has bounded error. 
%\red{From a theoretical standpoint, the parameter selection algorithm could identify five (possibly empty) sub-vectors each corresponding to one of the five precisions. However, there are practical considerations that arise during an implementation when converting information from high to low precision (e.g., overflow, underflow) which are necessary to account for in the computation phase of \textbf{qdot}. We take this into account in the design of our parameter selection algorithm $\Theta$ so that the bins and precisions identified by $\Theta$ are used directly by the computation phase.}

%\purple{\textbf{qdot} is implemented using C++ and CUDA. The parameter selection and computation phases both run on a GPU to take advantage of parallelism and \hpr \ precision in hardware. Many steps of the parameter selection phase, as well as the entire computation phase, are amenable to parallelization. In what follows, we present a detailed discussion of each of these phases.}
%\purple{Since the parameter selection algorithm is dependent on the parameterization of the approximate kernel, in what follows, we first provide motivation and pseudocode for the computation phase of \textbf{qdot} in~\secref{sec:comp-phase}. We then provide pseudocode for the parameter selection algorithm in~\secref{sec:param-phase} and a discussion on its approximation effectiveness and efficiency. Finally, in~\secref{sec:qdot-full-alg} we provide details on the implementation of \textbf{qdot}.}

\subsection{Computation Phase} \label{sec:comp-phase}
%To reduce time spent in the information collection phase, we introduce a strategy called \emph{exponent based binning}. Before providing an outline for this approach we establish some definitions and notation that will be useful in our discussion and analysis.

Given vectors $\bm{x}, \bm{y} \in \mathbb{R}^n$, the objective of \textbf{qdot} is to approximate the value of $\bm{x}^{\intercal} \bm{y}$. The design of \textbf{qdot} was motivated by loop perforation but refined to take advantage of other levels of approximation that loop perforation does not account for. For example, given $\bm{x}$ and $\bm{y}$ in \dpr \ precision, loop perforation can only select indices to exclude in the approximation of $\bm{x}^{\intercal} \bm{y}$. The use of only perforation as an option to introduce approximation will likely result in a kernel that has limited approximation potential in instances where a small approximation error tolerance $\varepsilon$ is required by the application. However, we can improve the approximation potential in these instances by including intermediate levels of approximation between \dpr \ and \ppr \ by determining component-wise products that can be computed at \spr \ and \hpr \ precisions, together with component-wise products that can be perforated. To see this, consider the process of computing all of the component-wise products in the dot product before accumulating the values. Let $\bm{z} = \bm{x} \odot \bm{y}$, where each component of $\bm{z}$ is computed at the input precision. Then 
all components have the same number of mantissa bits, say $\mu$, but potentially varying exponent values given by $\fl(\bm{z})$. %For simplicity, supposing each $\bm{z} \geq \bm{0}$ 
Now, whenever $\fl(z_j) < e_{\max}(\bm{z})$, the component $z_j$ is a candidate for reduced precision or perforation since the least significant mantissa bits in $z_j$ may contribute information outside of the required tolerance. Thus, perforation is suitable only when the dynamic range of the exponents of the entries in $\bm{z}$ is sufficiently large, or a large tolerance, $\varepsilon$, is acceptable. \spr \ and \hpr \, however, allow the introduction of approximation when the dynamic range is smaller, or a smaller value of $\varepsilon$ is required. 

The computation phase of \qdot \ computes the dot product of disjoint sub-vectors of the input arrays $\bm{x}$ and $\bm{y}$ at some specified level of precision. The bins representing the disjoint sub-vectors, and the precision of each sub-vector, are prescribed by the parameter selection algorithm, and passed as input to the computation phase. 
%\purple{These parameters are provided as inputs to the computation phase and are identified based on the input arrays and a specified tolerance $\varepsilon$.} 
Pseudocode for the computation phase of \textbf{qdot} is provided in Algorithm~\ref{alg:qdot-comp}. In order to simplify the presentation of the error analysis for Algorithm~\ref{alg:qdot-comp} and to motivate suitable choices for parameters that will result in bounded error, we provide the following definition:
\begin{definition} \label{def:exp-bin}
Let $\bm{x} \in \mathbb{R}^n$. An \textbf{$(\ell,u)$-exponent bin of $\bm{x}$}, denoted $B_{\ell, u} (\bm{x})$, is the set of indices %$i$ such that $\fl(x_i) \in (\ell, u]$, for some $\ell, u \in \mathbb{Z}$ with $\ell < u$. 
$B_{\ell, u} (\bm{x}) := \left\{ i \in [n] : \ell < \fl(x_i) \leq u \right\}$. 
%\label{eq:exp-bin}
\end{definition}
We now establish a bound on the approximation error introduced by Algorithm~\ref{alg:qdot-comp}.

\begin{center}
\begin{algorithm}[t]
\begin{algorithmic}[1]
\STATE{\textit{Inputs}: Arrays $\bm{x}, \bm{y} \in \mathbb{R}^n$, parameters $\theta = \{ B_k, \pi_k \}_{k = 1}^p$, where $B_k$ contains indices for sub-vector $k$ and $\pi_k$ denotes the precision of bin $B_k$.}
\STATE{\textit{Output}: Approximate dot product $z$}
\FOR{$k = 1$ to $p$}
\STATE{$(\bm{x}_k, \bm{y}_k) \gets Precision_{\pi_k} (\bm{x}_{B_k}, \bm{y}_{B_k})$ \hfill Initialize sub-vectors with precision $\pi_k$}
\STATE{$\displaystyle z_{k} \gets \bm{x}_k^{\intercal} \bm{y}_k$ \hfill Compute dot product of sub-vector $k$ at precision $\pi_k$ \ }
\ENDFOR
\STATE{$\displaystyle z \gets \sum_{k = 1}^{p} z_k$, where addition is at precision of $\bm{x}$ and $\bm{y}$ \hfill Accumulate across sub-vectors \ }
\end{algorithmic}
\caption{\textbf{qdot} Computation Phase}
\label{alg:qdot-comp}
\end{algorithm}
\end{center}

\iffalse %%%%%%%%%%%%%%%%%%%%%%%%%%%%%%%
In order to allow for the error analysis to generalize across any choice of a binning algorithm, we specify an assumption that a binning algorithm must satisfy in order for our error bound analysis to hold. 
\begin{assumption} \label{assmp:partition}
    Let $\bm{x}, \bm{y} \in \mathbb{R}^n$. A partition of the interval $[e_{\min}, e_{\max}]$ satisfies the following condition:
    \begin{itemize}
        \item[$(a)$] Cover all possible exponent values: $\displaystyle [e_{\min}, e_{\max}] \subseteq \bigcup_{k=1}^{N_{x,y}} (\ell_k, u_k]$.
        %\item[$(b)$] Pairwise disjoint intervals: $\displaystyle (\ell_i, u_i] \cap (\ell_j, u_j] = \emptyset$, for all $i \neq j$.
    \end{itemize}
\end{assumption}
\fi %%%%%%%%%%%%%%%%%%%%%%%%%%%%%%%
%In order to simplify the analysis, we avoid accumulating the error using floating point addition. %We are aware that this adds an additional amount of approximation that can be found in existing literature [\textcolor{red}{ADD REFERENCE}]. \textcolor{red}{Note that it might be worth just using some existing results to bound our error more accurately. Consider doing this for completeness.}

\begin{theorem} \label{thm:bin-dot-kernel}
Suppose $\bm{x}, \bm{y}$ are floating point arrays of size $n$ with machine epsilon $\varepsilon_{machine}$. Let $\{ (\ell_k, u_k] \}_{k \geq 1}$ be a partition of $[e_{\min}(\bm{x} \odot \bm{y}), e_{\max}(\bm{x} \odot \bm{y})]$ %satisfying Assumption~\ref{assmp:partition} 
and let $\theta = \{ B_{\ell_k, u_k} (\bm{x} \odot \bm{y}), \pi_k \}_{k = 1}^p$ denote parameters for Algorithm~\ref{alg:qdot-comp}. If $\bm{\tilde{k}}(\bm{x}, \bm{y}; \theta)$ is the output of Algorithm~\ref{alg:qdot-comp} then 
\begin{align}
\left| \bm{x}^{\intercal} \bm{y} - \bm{\tilde{k}}(\bm{x}, \bm{y}; \theta) \right|
&\leq \sum_{k = 1}^{p} M_k 2^{u_k+1} \varepsilon (k), \label{result:bin-dot-kernel.1}
\end{align}
where $\varepsilon(k)$ denotes the machine precision for precision $\pi_k$ and $\bsize$ denotes the cardinality of bin $B_{\ell_k,u_k} (\bm{x} \odot \bm{y})$, for $k \in [p]$.
\end{theorem}

\begin{proof}
Let $\bm{z} = \bm{x} \odot \bm{y}$. As $\{ (\ell_k, u_k] \}_{k = 1}^p$ is a partition of $[e_{\min}(\bm{z}), e_{\max}(\bm{z})]$, we have that
\begin{align}
\bm{x}^{\intercal} \bm{y}
&= \sum_{i=1}^n x_i y_i
= \sum_{k=1}^{p} \sum_{j \in B_{\ell_k,u_k} (\bm{z})} x_j y_j. \label{eq:bin-dot-kernel.1}
\end{align}
We momentarily consider each inner summation in (\ref{eq:bin-dot-kernel.1}). Let $\varepsilon (k)$ denote machine epsilon that will be used for the components at the indices in bin $B_{\ell_k,u_k} (\bm{z})$, for $k \in [p]$, where we require each $\varepsilon (k) \geq \varepsilon_{machine}$. By the fundamental axiom of floating point arithmetic \cite{trefethen97} and our choice of $\varepsilon (k) \geq \varepsilon_{machine}$, we have that
$x_j y_j - x_j \otimes y_j = - x_j y_j \varepsilon$, for $j \in B_{\ell_k,u_k} (\bm{z})$ and some $|\varepsilon| \leq \varepsilon (k)$. 
%\label{eq:bin-dot-kernel.2}
By definition of Algorithm~\ref{alg:qdot-comp}, it now follows that
\begin{align}
\bm{x}^{\intercal} \bm{y} - \bm{\tilde{k}}( \bm{x}, \bm{y}; \theta )
&= - \sum_{k=1}^{p} \sum_{j \in B_{\ell_k,u_k} (\bm{z})} x_j y_j \varepsilon_j, \label{eq:bin-dot-kernel.3}
\end{align}
where $| \varepsilon_j | \leq \varepsilon(k)$ for $j \in B_{\ell_k,u_k} (\bm{z})$. By definition of $B_{\ell_k,u_k} (\bm{z})$, we have that
$\left| x_j y_j \right| < 2^{u_k+1}$, for $j \in B_{\ell_k,u_k} (\bm{z})$, %\label{eq:bin-dot-kernel.4} 
which combined with (\ref{eq:bin-dot-kernel.1}) and (\ref{eq:bin-dot-kernel.3}) yields
\begin{align}
\left| \bm{x}^{\intercal} \bm{y} - \bm{\tilde{k}}( \bm{x}, \bm{y}; \theta ) \right|
&\leq \sum_{k=1}^{p} \sum_{j \in B_{\ell_k,u_k} (\bm{z})} | x_j y_j | | \varepsilon_j | 
\leq \sum_{k=1}^{p} \sum_{j \in B_{\ell_k,u_k} (\bm{z})} 2^{u_k+1} \varepsilon (k) 
\leq \sum_{k=1}^{p} \bsize 2^{u_k+1} \varepsilon (k), \label{eq:bin-dot-kernel.5}
\end{align}
which is the desired result, (\ref{result:bin-dot-kernel.1}).
\end{proof}

Theorem~\ref{thm:bin-dot-kernel} establishes that the use of $(\ell,u)$-exponent bins for the bin parameters can be used to formulate a closed-form bound on the absolute approximation error. Further, the error is dependent on the size, maximum exponent, and precision of each $(\ell,u)$-exponent bin. This allows us to limit the bin parameter search space for the parameter selection algorithm to only $(\ell,u)$-exponent bins. By including one additional assumption, we can use this result to establish a relative error bound for Algorithm~\ref{alg:qdot-comp}.

\begin{theorem} \label{thm:rel-bin-dot-kernel}
Let $\bm{x}$, $\bm{y}$, and $\theta$ be as in Theorem~\ref{thm:bin-dot-kernel}. Suppose that $e_{\max} (\bm{x} \odot \bm{y}) \leq \fl \left( \bm{x}^{\intercal} \bm{y} \right)$ and that $\{ (\ell_k, u_k] \}_{k = 1}^p$ is a partition of $[e_{\min}(\bm{x} \odot \bm{y}), e_{\max}(\bm{x} \odot \bm{y})]$. If $\bm{\tilde{k}}(\bm{x}, \bm{y}; \theta)$ is the output of Algorithm~\ref{alg:qdot-comp} then we have the relative error bound
\begin{align}
\left| \bm{x}^{\intercal} \bm{y} - \bm{\tilde{k}}(\bm{x}, \bm{y}; \theta) \right|
&\leq \left| \bm{x}^{\intercal} \bm{y}  \right| \sum_{k=1}^{p} \bsize 2^{u_k-e_{\max}+1} \varepsilon (k),   \label{result:bin-dot-kernel.2}
\end{align}
where $\varepsilon(k)$ denotes the machine precision for bin precision $\pi_k$ and $\bsize$ denotes the cardinality of bin $B_{\ell_k,u_k} (\bm{x} \odot \bm{y})$, for $k \in [p]$.
\end{theorem}

\begin{proof}
Let $E = \fl \left( \bm{x}^{\intercal} \bm{y} \right)$. By our hypothesis that $e_{\max} := e_{\max} (\bm{x} \odot \bm{y}) \leq E$, it follows that $2^{-E} \leq 2^{-e_{\max}}$ which together with the property that $| \bm{x}^{\intercal} \bm{y} | \geq 2^{E}$ yields
\begin{align}
\sum_{k=1}^{p} \bsize 2^{u_k+1} \varepsilon (k)
&= 2^E \sum_{k=1}^{p} \bsize 2^{u_k-E+1} \varepsilon (k)
%&\leq \left| \langle \bm{x}, \bm{y} \rangle \right| \sum_{\substack{k \leq M \\ k \equiv M mod \ell}} \left| B_{k,\ell} (\bm{x}, \bm{y}) \right| 2^{k-E+1} \varepsilon (k). \nonumber \\
\leq \left| \bm{x}^{\intercal} \bm{y} \right| \sum_{k=1}^{p} \bsize 2^{u_k-e_{\max}+1} \varepsilon (k). \label{eq:bin-dot-kernel.6}
\end{align}
Combining (\ref{eq:bin-dot-kernel.6}) with (\ref{eq:bin-dot-kernel.5}) yields the desired result (\ref{result:bin-dot-kernel.2}).
\end{proof}

Theorem~\ref{thm:rel-bin-dot-kernel} provides a way forward for designing the parameter selection phase for \textbf{qdot}. In particular, the parameter selection phase will need to determine the appropriate precisions for each exponent bin that results in the summation in (\ref{result:bin-dot-kernel.2}) being less than the required tolerance. The only practical issue with this result is the required hypothesis $e_{\max} (\bm{x} \odot \bm{y}) \leq \fl \left( \bm{x}^{\intercal} \bm{y} \right)$. Of course, this will hold when computing the norm $\| \bm{x} \|_2^2 = \bm{x}^{\intercal} \bm{x}$ as all of the componentwise products will be nonnegative. 
%\red{One could establish necessary conditions under which this hypothesis holds but any such conditions will likely require additional checks on the sign and exponent of the componentwise products. However, verifying these necessary conditions in an implementation would result in a less efficient parameter selection phase for \qdot.}
For completeness, we provide a relative error bound for when the condition $e_{\max} (\bm{x} \odot \bm{y}) > \fl \left( \bm{x}^{\intercal} \bm{y} \right)$ is satisfied. However, unsurprisingly, this result requires knowing the value of $\fl \left( \bm{x}^{\intercal} \bm{y} \right)$ to compute the error bound exactly. 
%\red{Note that this result follows from the same approach used in Theorem~\ref{thm:rel-bin-dot-kernel}.}
\begin{theorem} \label{thm:rel-bin-dot-kernel-E}
Let $\bm{x}$, $\bm{y}$, and $\theta$ be as in Theorem~\ref{thm:bin-dot-kernel} and $E = \fl \left( \bm{x}^{\intercal} \bm{y} \right)$. If $\{ (\ell_k, u_k] \}_{k = 1}^p$ is a partition of $[e_{\min}(\bm{x} \odot \bm{y}), e_{\max}(\bm{x} \odot \bm{y})]$ and $e_{\max} (\bm{x} \odot \bm{y}) > E$. Then we have the relative error bound
\begin{align}
\left| \bm{x}^{\intercal} \bm{y} - \bm{\tilde{k}}( \bm{x}, \bm{y}; \theta ) \right|
&\leq \left| \bm{x}^{\intercal} \bm{y}  \right| \sum_{k=1}^{p} \bsize 2^{u_k-E+1} \varepsilon (k),   \label{result:bin-dot-kernel-E}
\end{align}
where $\varepsilon(k)$ denotes the machine precision for precision $\pi_k$ and $\bsize$ denotes the cardinality of bin $B_{\ell_k,u_k} (\bm{x} \odot \bm{y})$, for $k \in [p]$.
\end{theorem}

%\green{A slightly modified, yet efficient, strategy can still be used to perform the error bounded approximation in cases where the hypothesis of Theorem~\ref{thm:rel-bin-dot-kernel-E} might be satisfied. In particular, by computing the largest exponent bins first then checking the resulting exponent, the approximate kernel will have a reliable estimate for $E$, say $\tilde{E}$, without computing the full kernel. If $\tilde{E} \geq e_{\max}$ then the computation can proceed with the pre-identified precisions for each bin. If $\tilde{E} < e_{\max}$ then the precisions identified for the remaining bins will need to be recalculated based on the estimate $\tilde{E}$ for $E$. We note that in our experiments using \textbf{qdot} within applications considered in here, we never encountered this issue. }

\subsection{Parameter Selection Phase} \label{sec:param-phase}
Since we are interested in approximation effectiveness in this work, the goal of the parameter selection agorithm is to identify the lowest precision that each component of the input arrays can be quantized to, while ensuring the approximation error is below the specified tolerance $\varepsilon$. 
%\red{Once we have determined how to accomplish this theoretically, we will design an algorithm for selecting the required parameters as efficiently as possible. Since we will be working with IEEE floating point representations in our implementation of \textbf{qdot}, we have limited the quantization levels in our discussion to \hpr, \spr, \dpr, and \qpr \ precision. Additionally, as mentioned earlier, we include \ppr \ as a level of quantization to indicate components that can be excluded from the computation of \textbf{qdot} while still satisfying the error tolerance. In essence, we can consider the perforated entries to be components that have been quantized to a 0-bit representation.} 
Based on Algorithm~\ref{alg:qdot-comp} and the error analysis in Section~\ref{sec:comp-phase}, the key to identifying the appropriate level of precision for each component can be achieved by partitioning the interval $[e_{\min}(\bm{x} \odot \bm{y}), e_{\max}(\bm{x} \odot \bm{y})]$ into $p$ disjoint intervals $\left\{ (\ell_k, u_k] \right\}_{k=1}^p$ and binning the indices in $(\ell_k,u_k)$-exponent bins. By using the size and maximum exponent of each exponent bin $B_{\ell_k, u_k}$, we can determine the highest level of quantization that can be applied to the indices in each bin without violating the error tolerance requirement. 

Suppose that an application requiring the dot product can effectively utilize an approximate dot product algorithm satisfying a relative error tolerance of $\varepsilon$. Then using (\ref{result:bin-dot-kernel.2}) in Theorem~\ref{thm:rel-bin-dot-kernel}, our goal would be to choose $\varepsilon(k)$, for $k \in [p]$, such that
\begin{align}
\sum_{k=1}^{p} M_k 2^{u_k-e_{\max}+1} \varepsilon (k) \leq \varepsilon, \label{eq:bin-dot-kernel-sec.1}
\end{align}
where $\bsize$ denotes the cardinality of bin $B_{\ell_k,u_k} (\bm{x} \odot \bm{y})$, for $k \in [p]$. 
%\green{Additionally, as outlined in Theorem~\ref{thm:bin-dot-kernel}, recall that $\varepsilon(k) \geq \varepsilon_{machine}$, for $k \in [p]$, where $\varepsilon_{machine}$ is the machine precision associated with the floating point representation of the input arrays $\bm{x}$ and $\bm{y}$.} 
This choice of $\varepsilon(k)$ corresponds to choosing the precision at which to compute the products for the elements in each bin. 
%\red{From a theoretical standpoint, we could think of the variables $\varepsilon(k)$ as continuous and fine tune them to choose values that optimize the effectiveness of (\ref{eq:bin-dot-kernel-sec.1}). However, from a practical standpoint we have limited our implementation to use \ppr \ and IEEE \hpr, \spr, \dpr, and \qpr \ precisions. This limitation somewhat simplifies the process for selecting the precision for each bin.} 
Note that for bins with $\pi_k = \ppr$, the value of $\varepsilon(k)$ is equal to 1.

If we follow the outline above, we can determine that the precision for each $(\ell,u)$-exponent bin is given by the function
\begin{align}
\text{Precision} (B_{\ell_k,u_k}(\bm{x} \odot \bm{y}) ) = \left\{
   \begin{array}{lcll}
      %\texttt{Octuple} &:& \beta_k \in [q, o),\ p_{x,y} \geq o \\
      %\texttt{Quadruple} &:& \sigma_k \in [m_{double}, m_{quad}), &\mu \geq m_{quad} \\
      \texttt{Double} &:& \sigma_k \in [m_{single}, m_{double}), &\mu \geq m_{double} \\
      \texttt{Single} &:& \sigma_k \in [m_{half}, m_{single}), &\mu \geq m_{single} \\
      \texttt{Half} &:& \sigma_k \in [0, m_{half}), &\mu \geq m_{half} \\
      \texttt{Perforate} &:& \sigma_k < 0
   \end{array}
   \right., \label{eq:precision-def}
\end{align}
%where \texttt{Perforate} indicates that it is not necessary to compute dot product for bin $k$,
where the condition involving $\mu$ ensures that $\varepsilon(k) \geq \varepsilon_{machine}$ is satisfied for each bin. Here, we refer to $\sigma_k$ as the \textit{bin score} and define it by
\begin{align}
    %\sigma_k := \lceil \mu - e_{\max} + \log_2 \left( | B_{\ell_k,u_k} (\bm{x} \odot \bm{y}) | \right) + u_k \rceil,
    %\sigma_k := \lceil \log_2 \left( \varepsilon \right) - e_{\max} + \log_2 \left( | B_{\ell_k,u_k} (\bm{x} \odot \bm{y}) | \right) + u_k \rceil,
    \sigma_k := \lceil \log_2(M_k) \rceil + u_k - e_{\max} - \lfloor \log_2(\varepsilon) \rfloor + 1
    \label{eq:bin-score-def}
\end{align}
for nonempty bins and $-\infty$ for empty bins. It is important to note that the value of $\mu$ is dependent on the precision of the input arrays and is defined by
\begin{align}
\mu = \left\{
   \begin{array}{lcl}
      %o &:& \text{$\bm{x}$ and $\bm{y}$ in \texttt{octuple} precision} \\
      %m_{quad} &:& \text{$\bm{x}$ and $\bm{y}$ in \texttt{quadruple} precision} \\
      m_{double} &:& \text{$\bm{x}$ and $\bm{y}$ in \texttt{double} precision} \\
      m_{single} &:& \text{$\bm{x}$ and $\bm{y}$ in \texttt{single} precision} \\
      m_{half} &:& \text{$\bm{x}$ and $\bm{y}$ in \texttt{half} precision}
   \end{array}
   \right. ,
\end{align}
and %$m_{quad}$, 
$m_{double}$, $m_{single}$, and $m_{half}$ denote the number of mantissa bits in %\texttt{quadruple}, 
\texttt{double}, \texttt{single}, and \texttt{half} precision, respectively. While these values are known and fixed, we have chosen to use this notation to represent them so that it might be clear how this could generalize to other, perhaps intermediate, levels of precision or alternative number representations \cite{cojean2020acceleration}.
%\red{representations with different numbers of significant bits.} 
%\red{As our discussion is limited to floating point representations we have only considered these four representations, however, from a theoretical standpoint one could imagine how this could be applied to representations without large gaps in the number of significant bits between each representation. Additionally, this approach is amenable to recently suggested alternative representations that could accelerate sparse linear algebra \cite{cojean2020customized}. While alternative number representations could allow for more levels of approximation in the Precision function and finer levels of approximation in computation, they are not considered in our theory or implementation as they are beyond the scope of this paper.}

%\purple{We now provide theoretical motivation for the precision function in (\ref{eq:precision-def}) and the bin score $\sigma_k$ in (\ref{eq:bin-score-def}). Specifically, we provide }
The following theorem establishes the relative error bound achieved by Algorithm~\ref{alg:qdot-comp} when (\ref{eq:precision-def}) is used to determine the bin precision parameters $\{ \pi_k \}_{k=1}^p$. %\red{We note that the formula in (\ref{eq:precision-def}) indicates that the most significant features of a bin contributing to the tightness of the relative error bound are the bin size and the maximum exponent of the bin, indicated by $u_k$.}

\begin{theorem} \label{thm:bin-precision}
Let $\bm{x}, \bm{y}$ be floating point arrays of size $n$ with machine precision $\varepsilon_{machine}$. Suppose that $e_{\max} := e_{\max} (\bm{x} \odot \bm{y}) \leq \fl \left( \bm{x}^{\intercal} \bm{y} \right)$ and $\theta = \{ B_{(\ell_k, u_k]} (\bm{x} \odot \bm{y}), \pi_k \}_{k = 1}^p$, where (\ref{eq:precision-def}) is used with tolerance $\varepsilon$ to determine the precisions $\{ \pi_k \}_{k = 1}^p$ at which the dot product for each bin is computed. If $\bm{\tilde{k}}(\bm{x}, \bm{y}; \theta)$ is the output of Algorithm~\ref{alg:qdot-comp} then there exists $\tau = O (\varepsilon)$ such that
\begin{align}
\left| \bm{x}^{\intercal} \bm{y} - \bm{\tilde{k}} (\bm{x}, \bm{y}; \theta) \right|
&\leq \tau \left| \bm{x}^{\intercal} \bm{y}  \right|. \label{result:bin-precision}
\end{align}
%where $\tau = O (\varepsilon)$ and $\varepsilon$ is the desired approximation error tolerance.
%$\tau = O (\varepsilon_{x,y})$ and $\varepsilon_{x,y}$ denotes the machine precision at which $\bm{x}$ and $\bm{y}$ are provided.
\end{theorem}

\begin{proof}
Let $b_k = \log_2 \left( \bsize \right)$, for all $k \in [p]$. By our hypotheses that $e_{\max} \leq \fl \left( \bm{x}^{\intercal} \bm{y} \right)$, the result in (\ref{result:bin-dot-kernel.2}) from Theorem~\ref{thm:rel-bin-dot-kernel} holds. 
\iffalse %%%%%%%%%%%%%%%%%%%%%%%%%%%%%%
In particular, the relative error bound
\begin{align}
\left| \bm{x}^{\intercal} \bm{y} - \bm{\tilde{k}} (\bm{x}, \bm{y}; \theta) \right|
&\leq \left| \bm{x}^{\intercal} \bm{y} \right| \sum_{k=1}^p \bsize 2^{u_k - e_{\max} + 1} \varepsilon (k) \label{eq:bin-precision.1}
\end{align}
holds, where $\bsize$ denotes the cardinality of bin $B_{\ell_k,u_k} (\bm{x} \odot \bm{y})$, for $k \in [p]$. 
\fi %%%%%%%%%%%%%%%%%%%%%%%%%%%%%%%%%%
\iffalse %%%%%%%%%%%%%%%%%%%%%%%%%%%%
For $k = p$, the term in the summation in (\ref{eq:bin-precision.1}) is 
\begin{align}
\bsize 2^{e_{\max} - e_{\max} + 1} \varepsilon (p) 
%&\geq 2^{b_N} 2^{1} \varepsilon (N) 
&\geq 2^{b_p + 1} \varepsilon(p) 
\geq 2 \varepsilon (p),
\end{align}
where the final inequality holds using $b_p \geq 0$. 
\fi %%%%%%%%%%%%%%%%%%%%%%%%%%%%%%%%%
Since the goal is to 
bound the relative approximation error of $\bm{x}^{\intercal} \bm{y}$ by $\varepsilon$, we show that the choices for bin precisions given in (\ref{eq:precision-def}) provides the desired result up to a scalar multiple.
\iffalse %%%%%%%%%%%%%%%%%%%%%%%%%%%
compute the highest precision approximation of $\bm{x}^{\intercal} \bm{y}$, the choice of precision for the bin $B_{\ell_k,u_k} (\bm{x}, \bm{y})$ should be the same as the precision of $\bm{x}$ and $\bm{y}$. 
Hence, we take 
\begin{align}
\varepsilon (N) = \varepsilon_{x, y} := 2^{-\mu}. \label{eq:bin-precision.1.1}
\end{align}

We now show that the choices for bin precisions given in (\ref{eq:precision-def}) provides the desired result. 
\fi %%%%%%%%%%%%%%%%%%%%%%%%%%%%%%
Let $k \in [p]$, then we require a choice for $\varepsilon(k)$ such that 
\begin{align}
\bsize 2^{u_k - e_{\max} + 1} \varepsilon (k) 
&\leq \varepsilon.
%&\leq \left| B_{\ell_N,u_N} (\bm{x}, \bm{y}) \right| 2 \varepsilon (N). 
\label{eq:bin-precision.2}
\end{align}
Since $\bsize 2^{u_k - e_{\max} + 1} \varepsilon (k) \leq 2^{ \lceil b_k \rceil + u_k - e_{\max} + 1} \varepsilon (k)$ %\label{eq:bin-precision.3}
and 
$\varepsilon
\geq 2^{\lfloor \log_2(\varepsilon) \rfloor}$, 
%\label{eq:bin-precision.4}
we can ensure (\ref{eq:bin-precision.2}) holds by enforcing
\begin{align}
2^{\lceil b_k \rceil + u_k - e_{\max} + 1} \varepsilon (k) 
&\leq 2^{\lfloor \log_2(\varepsilon) \rfloor}
%2^{\lceil b_k \rceil + u_k - e_{\max} + 1} \varepsilon (k) \leq 2 \varepsilon(N) 
\label{eq:bin-precision.5}
\end{align}
which, when rearranged, yields the bound
$\varepsilon (k) 
\leq 2^{e_{\max} - \lceil b_k \rceil - u_k - 1 + \lfloor \log_2(\varepsilon) \rfloor}$.
%\label{eq:bin-precision.6}
Expressing $\varepsilon(k) = 2^{-\mu_k}$, where $\mu_k \in \mathbb{N}$ is to be determined,  yields the simplified inequality
$-\mu_k \leq e_{\max} + \lfloor \log_2(\varepsilon) \rfloor - \lceil b_k \rceil - u_k - 1$. %\label{eq:bin-precision.7}
Reordering and labeling components 
%in (\ref{eq:bin-precision.7}) 
allows for a meaningful interpretation of the bound on $\mu_k$ as %the range of values between the minimum significant exponent and the maximum possible exponent for the dot product in the bin, that is
\begin{align}
\mu_k 
\geq \aunderbrace[l1r]{\strut \left( \lceil b_k \rceil + u_k \right) \strut}_{\strut \text{Max reachable bit}} - 
\aunderbrace[l1r]{\strut \left( e_{\max} + \lfloor \log_2(\varepsilon) \rfloor \right) \strut}_{\strut \text{Min significant bit}} + 
\aunderbrace[l0r]{\strut 1 \strut}_{\strut \text{Min bit}}. \label{eq:bin-precision.7.1}
\end{align}

As $\mu_k$ denotes the number of mantissa bits required to achieve the desired bound, the inequality in (\ref{eq:bin-precision.7.1}) allows us to determine the minimum level of precision the components in bin $B_{\ell_k,u_k} (\bm{x} \odot \bm{y})$ can be computed at while satisfying the desired bound. 
%\red{As all of the values on the right hand side of (\ref{eq:bin-precision.7.1}) are known we can compute this value.} 
%\red{For implementation purposes, we are limiting the choice of precision to IEEE Floating Point representations. Hence, when}
When the value $\lceil b_k \rceil + u_k - e_{\max} - \lfloor \log_2(\varepsilon) \rfloor + 1$ falls between the number of mantissa bits for two different IEEE Floating Point representations, it is necessary to choose the precision with more mantissa bits to ensure the desired relative error bound is satisfied. If $\lceil b_k \rceil + u_k - e_{\max} - \lfloor \log_2(\varepsilon) \rfloor + 1 \leq 0$ then the components in bin $B_{\ell_k,u_k} (\bm{x} \odot \bm{y})$ are not required to compute the approximate dot product at the desired level of accuracy, which results in assigning a precision of \ppr \ to bin $k$. 
%\red{Additionally, by working under the assumption that no computations should be made at a precision higher than the precision used by the given vectors $\bm{x}$ and $\bm{y}$, the highest precision assigned to any bin will not exceed the input precision.} 
Following these observations yields the definition of the Precision function in (\ref{eq:precision-def}) and the bin score value in (\ref{eq:bin-score-def}).

Now using (\ref{eq:precision-def}), we have that $\bsize 2^{u_k - e_{\max} + 1} \varepsilon (k) \leq \varepsilon$,
%\label{eq:bin-precision.8} 
for all $k \in [p]$, and hence,
\begin{align}
\sum_{k=1}^{p} \bsize 2^{u_k - e_{\max} + 1} \varepsilon (k)
%&\leq \sum_{k=1}^p 2 \varepsilon_{x,y}
\leq N_{bins} \varepsilon, \label{eq:bin-precision.9}
\end{align}
where we write $N_{bins}$ to denote the number of nonempty bins and note that $N_{bins} \leq p$. Finally, substituting (\ref{eq:bin-precision.9}) into 
%(\ref{eq:bin-precision.1}) 
(\ref{result:bin-dot-kernel.2}) yields
$| \bm{x}^{\intercal} \bm{y} - \bm{\tilde{k}} (\bm{x}, \bm{y}; \theta) | \leq N_{bins} \varepsilon \left| \bm{x}^{\intercal} \bm{y}  \right|$, %\label{eq:bin-precision.10}
the desired result.
\end{proof}

%\red{The parameter selection algorithm should return the exponent bins and the precision required for each precision bin.} 
From a purely theoretical standpoint, the entire parameter selection algorithm can be summarized by using Definition~\ref{def:exp-bin} to compute the exponent bins, then using  (\ref{eq:precision-def}) to compute the precision for each bin. %\purple{However, since there are several techniques for partitioning the interval $[e_{\min}, e_{\max}]$ and computing the exponent arrays we provide more specific details of our implementation.}
Pseudocode for the parameter selection algorithm is given in Algorithm~\ref{alg:qdot-param}. 
%To accompany the presentation of this algorithm, 
%\purple{We now provide a discussion of each step together with memory and runtime complexities.}
\begin{center}
\begin{algorithm}[t]
\begin{algorithmic}[1]
\STATE{\textit{Inputs}: Arrays $\bm{x}, \bm{y} \in \mathbb{R}^n$, number of mantissa bits in precision of inputs $\mu$, approximation error tolerance $\varepsilon$, and strategy for partitioning exponent range.}
\STATE{\textit{Outputs}: Bins and corresponding Bin Precisions $\theta \gets \{ B_{\ell_k, u_k}, \pi_k \}_{k = 1}^p$}
\STATE{$\bm{e} \gets \fl (\bm{x}) + \fl (\bm{y})$ \hfill Exponent preprocessing}
\STATE{$e_{\min} \gets \min\{ e_i: 1 \leq i \leq n \}$; $e_{\max} \gets \max\{ e_i: 1 \leq i \leq n \}$ \hfill Identify min and max exponents}
\STATE{$\bm{e} \gets \texttt{sort} \ \bm{e}$ \hfill Sorted bin initialization}
\STATE{$B_{\ell_k, u_k} \gets \{ j \in [n] : \ell_k < e_j \leq u_k \}$ \hfill Identify components for each exponent bin}
\STATE{$\sigma_k \gets \lceil \log_2(M_k) \rceil + u_k - e_{\max} - \lfloor \log_2(\varepsilon) \rfloor + 1$ \hfill Compute bin scores from (\ref{eq:bin-score-def})}
\STATE{$\pi_k \gets \text{Precision} \left( B_{\ell_k,u_k} (\bm{x} \odot \bm{y}) \right)$ \hfill Determine bin precision using (\ref{eq:precision-def}) \ }
\end{algorithmic}
\caption{\textbf{qdot} Parameter Selection Algorithm}
\label{alg:qdot-param}
\end{algorithm}
\end{center}

Line 3 of Algorithm~\ref{alg:qdot-param} computes the array of exponents $\fl (\bm{x}) + \fl (\bm{y})$, which is the key piece of information required to bound the relative approximation error.
%From our analysis, we find that the key piece of information required to bound the relative approximation error is the array of exponents $\fl (\bm{x}) + \fl (\bm{y})$. This is computed in line 3 of Algorithm~\ref{alg:qdot-param}. 
In our implementation of the parameter selection phase, this step is accomplished by using a bitmask and bitshift to extract the unbiased floating point exponent of $x_i$ and $y_i$, and saving their sum in an array of integers. For inputs with precision up to \qpr, these sums can be computed using a \sint \ since the \qpr \ exponent is 15 bits long and an unsigned \sint \ has a capacity of 16 bits. Furthermore, by working in the exponent space, we can avoid introducing additional floating point operations in the parameter selection algorithm. In fact, nearly all of the operations required by our exponent preprocessing algorithm involve only integers, in particular \texttt{short int} values, and not floating points. 
%\red{If the implementation of this step in the parameter selection phase required the computation of each component-wise product $x_i y_i$ as floating point operations then in the original input precision, then this step would be as costly as computing the original dot product, and will result in an approximation that is immediately more costly.} 
This step has a complexity of $O(n)$.
%In order to efficiently identify the components for each exponent bin, 

Algorithm~\ref{alg:qdot-param} requires that the exponent array, $\bm{e}$, is sorted. In our implementation, we make use of the counting sort algorithm based on two observations. Firstly, the counting sort has a worst-case runtime of $O(n + c)$, where $n$ is the number of elements being sorted and $c$ is the number of distinct values that the $n$ elements can have. In Algorithm~\ref{alg:qdot-param}, the number of distinct values being counted corresponds to the exponent range, which is given by $c = e_{\max} - e_{\min} + 1$. Hence, the worst-case cost for using counting sort is $O(n + e_{\max} - e_{\min} + 1)$. Secondly, the size of the exponent range is bounded fairly tightly regardless of the precision used. For example, input arrays with \dpr \ precision have exponents in $[-1023, 1024]$ which yields that $e_{\max} - e_{\min} + 1 \leq 2048 = 2^{11}$. Further, input arrays with \qpr \ precision satisfy $e_{\max} - e_{\min} + 1 \leq 32768 = 2^{15}$. 
%Other popular and efficient sorting algorithms, such as merge sort, have a worst-case cost of $O(n \log_2 n)$. Comparing the complexity of counting sort to merge sort yields that counting sort should be theoretically faster than merge sort for problems of dimension $O(\texttt{1e3})$ and larger, regardless of input precision. 
Compared to sorting algorithms with runtime $O(n \log_2 n)$, counting sort should be theoretically faster for problems of dimension $n = O(\texttt{1e3})$ and larger, regardless of input precision. 
%\red{More details supporting this claim can be found in Appendix~\ref{sec:appendixc}. These observations support the use of counting sort instead of merge sort.} 
Hence, the runtime and memory complexity for this step is $O(n + e_{\max} - e_{\min} + 1)$.

In line 6 of Algorithm~\ref{alg:qdot-param}, we split the original vector into sub-vectors using exponent bins that depend on a partition of the interval $[e_{\min}, e_{\max}]$. To simplify the parameter selection algorithm, instead of explicitly requiring a partition of this range, we specify a strategy that is used to determine partition parameters $\ell_k$ and $u_k$ during the parameter selection algorithm. 
%\red{This simplifies the exponent bin construction during the Algorithm~\ref{alg:qdot-param}. In this paper, we theoretically consider three different binning strategies: \begin{enumerate}    \item {\bfseries Exact binning}: The interval $[e_{\min}, e_{\max}]$ is partitioned into singleton intervals defined by $\{ (u_k - 1, u_k] \}_{k=1}^{e_{\max} - e_{\min} + 1}$ where $u_k = e_{\min} + k - 1$.    \item {\bfseries Ranged binning}: The interval $[e_{\min}, e_{\max}]$ is partitioned into intervals of fixed-width $w$ given by $\{ (u_k - w, u_k] \}_{k \geq 1}$, where $u_k = e_{\min} + kw - 1$. The parameter $w$ is the width of each bin and is specified beforehand by the user. Exact binning corresponds to $w = 1$.    \item {\bfseries Variable-width binning} (bin splitting): The interval $[e_{\min}, e_{\max}]$ is recursively split into subintervals of varying size where the user can specify the maximum number of splits beforehand. Using a larger number of splits has the potential to identify more components that can be approximated during computation.\end{enumerate} Our implementation makes use of exact binning and follows from theoretical and practical comparisons of the three partition techniques considered. The choice of exact binning yields an effective and efficient approximation of the full-precision dot product kernel.}
After considering various strategies, we follow a straightforward approach that emphasizes approximation effectiveness: 
\begin{itemize}
    \item {\bfseries Exact binning}: The interval $[e_{\min}, e_{\max}]$ is partitioned into singleton intervals defined by $\{ (u_k - 1, u_k] \}_{k=1}^{e_{\max} - e_{\min} + 1}$ where $u_k = e_{\min} + k - 1$.
\end{itemize}
In addition to having the lowest asymptotic runtime and memory complexity compared to other binning strategies we considered, exact binning also helps our implementation avoid overflow and underflow issues when converting input data from higher precisions to lower precisions (e.g., \dpr \ to \hpr). To avoid running into these issues, we use the following technique outlined here for bin $B_{\ell_k, u_k}$. First, the data for $B_{\ell_k, u_k}$ is scaled so that the unbiased floating point exponent is 0. The scaled data is then copied into an array with precision $\pi_k$. The dot product is then computed for the components in bin $B_{\ell_k, u_k}$ at precision $\pi_k$. The resulting dot product is then converted back to the original input precision. Finally, by using exact binning, we can undo the scaling by multiplying the resulting dot product for bin $B_{\ell_k, u_k}$ by $2^{u_k}$. 
%\red{Hence, in addition to being the most theoretically effective partitioning technique we considered, exact binning also simplifies the technique to avoid overflow and underflow in \textbf{qdot}. Runtime and memory analysis comparing the three different binning methods is provided in Appendix~\ref{sec:appendixa}.}
The computational and memory complexity of exact binning is $O(n + e_{\max} - e_{\min} + 1)$.

The final steps in lines 7 and 8 of  Algorithm~\ref{alg:qdot-param} are used to identify the precision of each bin. For $p$ bins, where $p \leq n$, this step has a computational and memory cost of $O(p)$. Combining the cost for each step yields the total computational and memory cost of $O(n + e_{\max} - e_{\min} + 1)$. Since $e_{\max} - e_{\min} + 1$ is bounded by $2^{11}$ and $2^{15}$ for \dpr \ and \qpr, respectively, it follows that for even moderately sized problems the computational and memory complexity for Algorithm~\ref{alg:qdot-param} are effectively $O(n)$. It is important to emphasize that lines 3 -- 6 and line 8 make use of \texttt{short int}s. Furthermore, even though line 7 makes use of $\log_2$, in practice, we found the number of bins, $p$, to be relatively small (e.g., between 15 and 40).

%\textcolor{red}{(Add discussion of approximation efficiency for parameter selection phase here)}

Before proceeding to the full \textbf{qdot} algorithm, we consider a practical issue that may arise when using this error-bounded approximate kernel. Specifically, 
%\red{if the user's application requires the approximation error to be close to $\varepsilon_{machine}$ corresponding to the input arrays $\bm{x}$ and $\bm{y}$} 
it is possible that the parameter selection phase will recommend that all components should remain at the same precision as the inputs. In this case, our approximate kernel will be identical to the original kernel but will incur the overhead cost of the full parameter selection phase. 
%\red{To alleviate this potential issue, we note that we} 
We can establish necessary conditions for when the parameter selection algorithm will recommend using the input precision for all components based only on the values of $e_{\min}$ and $e_{\max}$ to perform early termination of the quantization. %\qdot}. 
%\red{This condition can be used to perform early termination of the parameter selection algorithm before the more costly steps are performed, as it will be known in advance that they will not yield any quantization.}

\begin{theorem} \label{thm:early-termination}
Let $\bm{x}$ and $\bm{y}$ be arrays of floating point values and let $\hat{\mu}$ denote the number of mantissa bits for the level of precision directly below the precision of the inputs $\bm{x}$ and $\bm{y}$ and let $\varepsilon$ be the tolerance for \qdot. If $e_{\max} - e_{\min} \leq -\lfloor \log_2(\varepsilon) \rfloor - \hat{\mu}$ then the level of precision specified for all bins by (\ref{eq:precision-def}) will be the same as the precision of $\bm{x}$ and $\bm{y}$.
\end{theorem}
\begin{proof}
By our hypothesis that $e_{\max} - e_{\min} \leq -\lfloor \log_2(\varepsilon) \rfloor - \hat{\mu}$ we have $\hat{\mu} \leq -\lfloor \log_2(\varepsilon) \rfloor - e_{\max} + e_{\min}$. Combining (\ref{eq:bin-score-def}) together with the fact that $\log_2 \left( \bsize \right) \geq 0$ and $u_k \geq e_{\min}$ for all $k \in [p]$, it now follows that
$\hat{\mu} \leq -\lfloor \log_2(\varepsilon) \rfloor - e_{\max} + e_{\min} \leq -\lfloor \log_2(\varepsilon) \rfloor - e_{\max} + \lceil \log_2 \left( \bsize \right) \rceil + u_k \leq \sigma_k$,
%\label{eq:early-termination.2}
for all $k \in [p]$. Since $\hat{\mu}$ is the number of mantissa bits for the level of precision directly below the precision of $\bm{x}$ and $\bm{y}$, the Precision function in (\ref{eq:precision-def}) yields that the precision of every bin will be the same as the input. 
%\red{which results in computing the full dot product at the input precision.}
\end{proof}

\iffalse %%%%%%%%%%%%%%%%%%%%%%%%%%%%%%
The speedup that can be achieved in Algorithm~\ref{alg:q-dot} when compared to a full-precision dot product kernel depends on two key factors:
\begin{enumerate}
    \item[(1)] Time spent in the parameter selection and data processing phase
    \item[(2)] Quality of parameters returned by the parameter selection phase
\end{enumerate}
Factor (1) signifies the overhead required to identify parameters for \textbf{qdot} that satisfy the error tolerance and is an overhead cost when compared to the original dot product and to prepare the mixed precision data for computation. The design of the parameter selection phase in Algorithm~\ref{alg:q-dot} attempts to minimize the time spent identify the best parameter choice while still bounding the approximation error. %The following sections will provide runtime and memory analysis that support the design of our parameter selection algorithm. Factor (2) indicates where benefits due to approximation are observed. In particular, inputs amenable to parameters with heavy amounts of half or perforated indices can benefit from speedups on GPUs and due to perforation.
\fi %%%%%%%%%%%%%%%%%%%%%%%%%%%%%%

\subsection{qdot Pseudocode and Implementation Details} \label{sec:qdot-full-alg}
Putting all the components together, pseudocode for our error-bounded approximate dot product kernel framework, \textbf{qdot}, is presented in Algorithm~\ref{alg:qdot}. \textbf{qdot} is implemented using C++ and CUDA. The parameter selection and computation phases both run on a GPU to take advantage of parallelism and \hpr \ precision in hardware. Many steps of the parameter selection phase, as well as the entire computation phase, are amenable to parallelization. 
%\purple{At this point, we have established all of the necessary components for our error-bounded approximate kernel framework and have theoretically established error bound properties for the computation phase based on the parameters identified by the parameter selection phase. We now combine the individual components to form \textbf{qdot}. From our general framework in Section~\ref{sec:error-bounded-ac}, the design of \textbf{qdot} should be clear but for completeness we include pseudocode for \textbf{qdot} in Algorithm~\ref{alg:qdot}.}
\begin{center}
\begin{algorithm}[t]
\begin{algorithmic}[1]
\STATE{\textit{Inputs}: Arrays $\bm{x}, \bm{y} \in \mathbb{R}^n$, and error tolerance $\varepsilon$}
\STATE{\textit{Output}: Approximate dot product $z$}
\STATE{\textbf{------ Parameter Selection Phase ------}}
\STATE{Use Algorithm~\ref{alg:qdot-param} with input $(\bm{x}, \bm{y})$, tolerance $\varepsilon$, and exact binning to identify parameters $\theta$}
\STATE{\textbf{------ Computation Phase ------}}
\STATE{Use Algorithm~\ref{alg:qdot-comp} with input $(\bm{x}, \bm{y})$ and parameters $\theta$ to compute $z$}
\end{algorithmic}
\caption{\textbf{qdot} -- Quantized Dot Product Kernel}
\label{alg:qdot}
\end{algorithm}
\end{center}

%%%%%%%%%%%%%%%%%%%%%%%%%%%%%%%%%%%%

%First, we consider each step of the parameter selection phase, Algorithm~\ref{alg:qdot-param}. 
%\red{To implement line 2 efficiently, we use a binary mask and bit shifts to manipulate the floating point representation of each component and save the unbiased exponent in a \sint.} 
In the implementation of the parameter selection phase, for components $x_i$ and $y_i$,  the \sint \ exponent values are summed on each thread to determine $e_i$, where each index $i \in [n]$ is operated in parallel from the other indices. The values of $e_{\min}$ and $e_{\max}$ are then computed using a reduction over the array of exponents $\bm{e}$. Here, we use \texttt{thrust::minmax\_element} from the CUDA Thrust C++ template library \cite{thrust2021doc} to perform this reduction. Next, we perform counting sort of $\bm{e}$ so that the exponent values are in non-decreasing order. We implemented our own version of counting sort in CUDA to take advantage of parallelism. In our implementation of the counting sort algorithm, we save the cumulative counts of each value in the range $[e_{\min}, e_{\max}]$, which makes it straightforward to define the bins.
%\green{In counting sort, the sorting is achieved by using the count of each integer in the range of possible values $[e_{\min}, e_{\max}]$ then making a second pass over the array and using the count of each element to reorder the indices. To parallelize the step where the values in $[e_{\min}, e_{\max}]$ are counted, the full array $\bm{e}$ is partitioned into $m$ different parts, the counts are computed over the individual parts in parallel, then a reduction is performed across the counts of all the parts to determine the full count of each value. To parallelize the step where the indices are reordered, we make use of the total count and the count over the individual parts. This allows us to perform the reordering of the indices in parallel over the same parts that were used to parallelize the counting step. }
%\textcolor{red}{(Add figure to demonstrate how counting sort is parallelized because description may be confusing).} 
%The process of identifying the bins is actually completed during the counting sort step as we save the cumulative counts of each value in the range $[e_{\min}, e_{\max}]$. 
%\textcolor{red}{which line in the algorithm does this correspond to?}
%\red{This is an additional benefit of using counting sort over an alternative sorting option.} 
All of the remaining steps of the parameter selection algorithm are applied to each bin independently, and hence, are easily parallelizable on the GPU. 
%Note that the number of nonempty bins in experiments with scientific computing applications was typically small (e.g., between 15 and 40).

While the computation phase is theoretically straightforward, we had to carefully handle computation in lower precisions to avoid overflow and underflow. As noted previously in Section~\ref{sec:param-phase}, this requires scaling the input data prior to performing the quantization, and re-scaling the result back after computing dot product of the individual bins. %\purple{First, before converting the data from the input precision to lower precisions, we create a rescaled copy of the input arrays where all of the floating point exponents are modified by applying a bitmask so that the resulting exponent corresponds to an unbiased exponent with value 0. This allows us to copy the rescaled data to any precision without encountering overflow or underflow. We then make copies of the rescaled data in the lower precisions required for the computation of each precision bin. For \dpr \ precision inputs $\bm{x}$ and $\bm{y}$, this requires a rescaled copy of $\bm{x}$ and $\bm{y}$ then copies of the required regions of these arrays at \spr \ and \hpr \ for the computations. This results in a total of at most 3 copies per array, depending on the precisions identified during the parameter selection algorithm. }
%\red{It is possible that the use of rescaled copies will not always be necessary as we observed in several instances where the computation was correct without rescaling. However, we elected to include this rescaling in all of our experiments since the primary focus of our implementation was on approximation effectiveness and not efficiency.}
%\purple{After the computation of each precision bin is completed, the resulting scalar value for each bin is converted to the input precision then rescaled using the exponent $u_k$ corresponding to the exact bin for which the computation was performed. Specifically, this amounts to multiplying the resulting dot product for each bin $B_{\ell_k, u_k}$ by $2^{u_k}$. }
%\red{This is where the use of exact bins helps our implementation avoid issues with overflow and underflow while also providing a straightforward rescaling and unscaling process.} 
Each bin dot product is computed using \texttt{cublas} functions from the CUDA Basic Linear Algebra Subroutine library \cite{cublas2021doc}. In particular, as the higher precision values, such as \dpr, do not need to be scaled in our experiments, we compute all of the bins identified as \dpr \ using a single call to \texttt{cublasDdot}. The lower precision bins, \hpr \ and \spr, need to be scaled individually so we use strided and batched \texttt{cublas} functions to efficiently compute the dot products for each bin. This way, they can be re-scaled individually before accumulating. For \spr \ bins, we use \texttt{cublasSgemmStridedBatched}, and for \hpr \ bins, we use \texttt{cublasGemmStridedBatchedEx}. To avoid overflow when computing the dot product for large bins in \hpr \ precision, we make use of Tensor Core Units \cite{cuda2021guide, zachariadis2020accelerating} when calling the function \texttt{cublasGemmStridedBatchedEx}. This is a hardware efficient operation for performing the products in \hpr \ precision while performing the accumulation in \spr \, which avoids overflow. When the Tensor Core Units were not used, we observed that overflow occurred in some cases where the input arrays were large in dimension.
\section{Numerical Experiments}
\label{sec:results}
In this section, we perform numerical experiments using \textbf{qdot} to verify the theoretical error bound and to test the level of approximation that can be used in scientific computing applications. The first set of experiments is performed using synthetic data and serves to empirically verify the error bound and compute the approximation effectiveness, efficiency, and speedup of \textbf{qdot}. Then we integrate \textbf{qdot} into two scientific computing applications to investigate the level of approximation that can be introduced into the dot product kernel without compromising the fidelity of each application. 
%\purple{The applications we have chosen to test serve as representatives for the large classes of (a) linear solvers and (b) eigenvalue computation techniques. We choose a fundamental technique to serve as a representative from each class for our experiments. In particular, we use Conjugate Gradient (CG) as the linear solver representative and the Power method as the eigenvalue application representative.}
The applications we consider are the Conjugate Gradient (CG) method for solving linear systems, and the Power Method for computing extreme eigenvalues of linear systems.

\subsection{Tightness of Error Bound}
\label{sec:experiment-errorbound}
In the first set of experiments, we used \textbf{qdot} to approximate a \dpr \ precision dot product for arrays of varying dimension sampled from two different distributions. 
%\red{In these experiments, we will vary the array dimension $n$, the exponent range $[e_{\min}, e_{\max}]$ in the sample data, and the error tolerance $\varepsilon$ for \textbf{qdot}.} 
Since the level of approximation that can be introduced is dependent on the exponent range $[e_{\min}, e_{\max}]$, the floating point exponents for the components of the arrays are sampled from  uniform or normal distributions, parameterized by a single value $t$:
\begin{enumerate}
    \item Distribution $\mathcal{A}$: Uniformly Distributed Floating Point Exponents. Array components of the form $s * 2^p$, where $s \sim \mathcal{U}[0.5,1)$ and $p \sim \mathcal{U}[-0.5*t, 0.5*t]$.
    \item Distribution $\mathcal{B}$: Normally Distributed Floating Point Exponents. Array components of the form $s * 2^p$, where $s \sim \mathcal{U}[0.5,1)$ and $p \sim \mathcal{N}(0,0.5 * t)$.
\end{enumerate}
The theoretical analysis, suggests that larger exponent ranges $[e_{\min}, e_{\max}]$ provides the potential for more components to be quantized to lower precisions, or be perforated even when small \qdot \ tolerance values, $\varepsilon$, are used.

In the first experiment, we vary the problem dimension while keeping the \textbf{qdot} tolerance fixed at $\varepsilon = \texttt{1e-16}$, in order to investigate the tightness of the error bound. For each problem dimension, we sample 1900 array pairs from distribution $\mathcal{A}$ and 1500 array pairs from distribution $\mathcal{B}$. These pairs are sampled with varying values for the distribution parameter $t$ in order to capture a wide range of input arrays with each distribution. For distribution $\mathcal{A}$, $t$ is taken to be values from the set $\{ 5i: 2 \leq i \leq 20 \}$ with 100 pairs sampled for each $t$. For distribution $\mathcal{B}$, $t$ is taken to be values from the set $\{ 2i : 1 \leq i \leq 15 \}$ with 100 pairs sampled for each $t$. %Larger values of $t$ for either distribution resulted in an overwhelming majority of the components being perforated. 
For each array pair, we use \textbf{qdot} to compute an approximate dot product value and the cublas function \texttt{cublasDdot} to compute the \dpr \ precision dot product. The relative approximation error, denoted by \emph{relerr} in the figures that follow, is equal to the quantity $| \texttt{cublasDdot}(\bm{x}, \bm{y}) - \textbf{qdot}(\bm{x}, \bm{y}; \varepsilon) |/| \texttt{cublasDdot}(\bm{x}, \bm{y}) |$. The relative approximation error bound for \textbf{qdot} follows from (\ref{result:bin-precision}). Since the number of non-empty bins, $N_{bins}$, is known prior to determining the bin precision, the bin score formula in (\ref{eq:bin-score-def}) is slightly adjusted to use $\varepsilon / N_{bins}$ in place of $\varepsilon$. This is done to ensure that the relative approximation error bound is satisfied for the desired $\varepsilon$. To capture the tightness of the \textbf{qdot} error bound, we compute the difference between the \textbf{qdot} error bound, denoted $qdot_{bound}$, and the relative approximation error of \textbf{qdot}. For each problem dimension from the set $\{ \texttt{1e2}, \texttt{1e3}, \texttt{1e4}, \texttt{1e5}, \texttt{1e6}, \texttt{1e7} \}$, we generated bar and whisker plots of this difference over all of the samples in Figure~\ref{fig:verify-bound}. We find that data where the exponents are sampled from a normal distribution yields instances where the error is tightly bounded by the error bound even with increasingly large array sizes. On the other hand, the data in which the floating point exponents are uniformly distributed is not captured tightly by the error bound as the problem size increases. %\purple{It is not entirely clear why the bound becomes looser as the problem size increases in instances where the floating point exponents are sampled from a uniform distribution.}

\begin{figure}[!ht]
  \begin{subfigure}[b]{0.48\textwidth}
    \includegraphics[width=\textwidth]{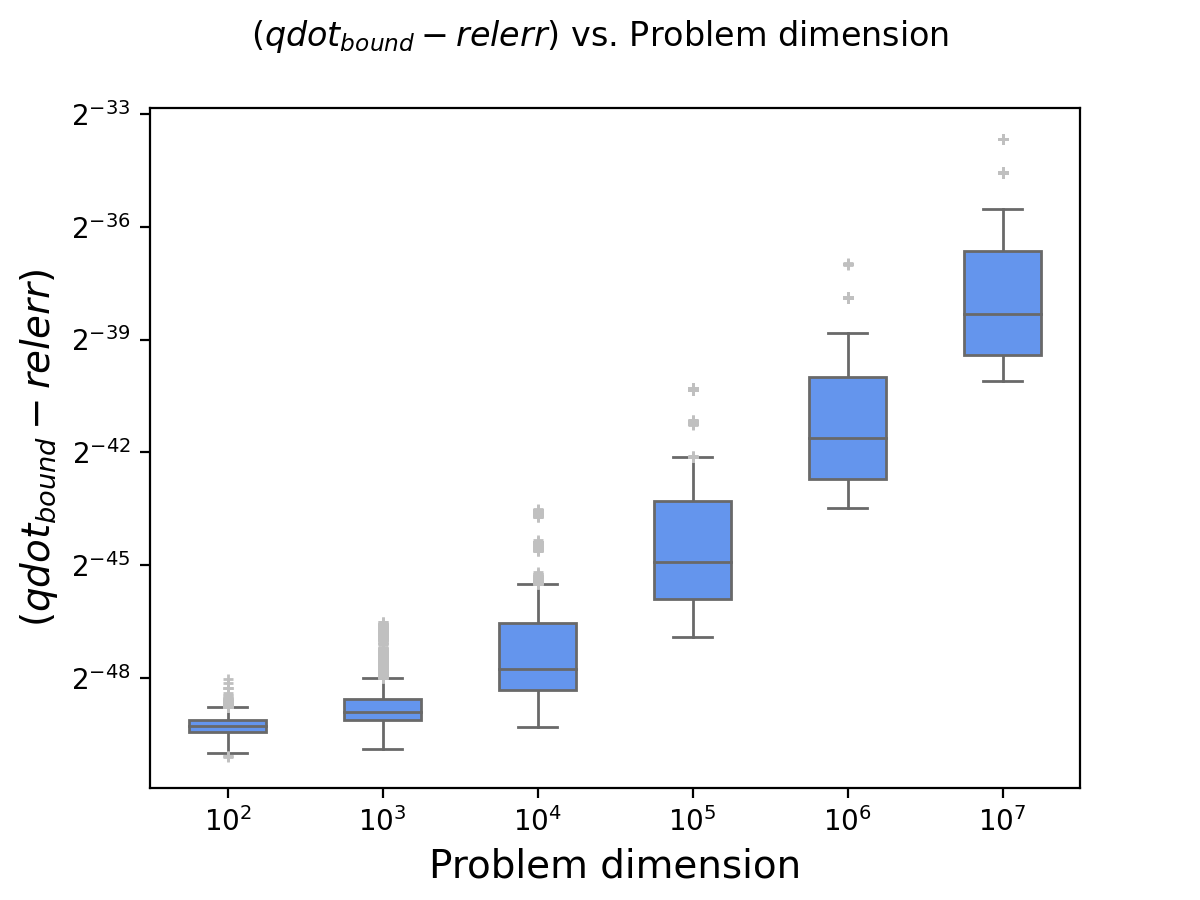}
    \caption{Distribution $\mathcal{A}$. For each problem dimension, bar and whisker plots are generated from using \qdot \ with 1900 sampled array pairs.}
    \label{fig:verify-uniform}
  \end{subfigure}
  \hfill
  \begin{subfigure}[b]{0.48\textwidth}
    \includegraphics[width=\textwidth]{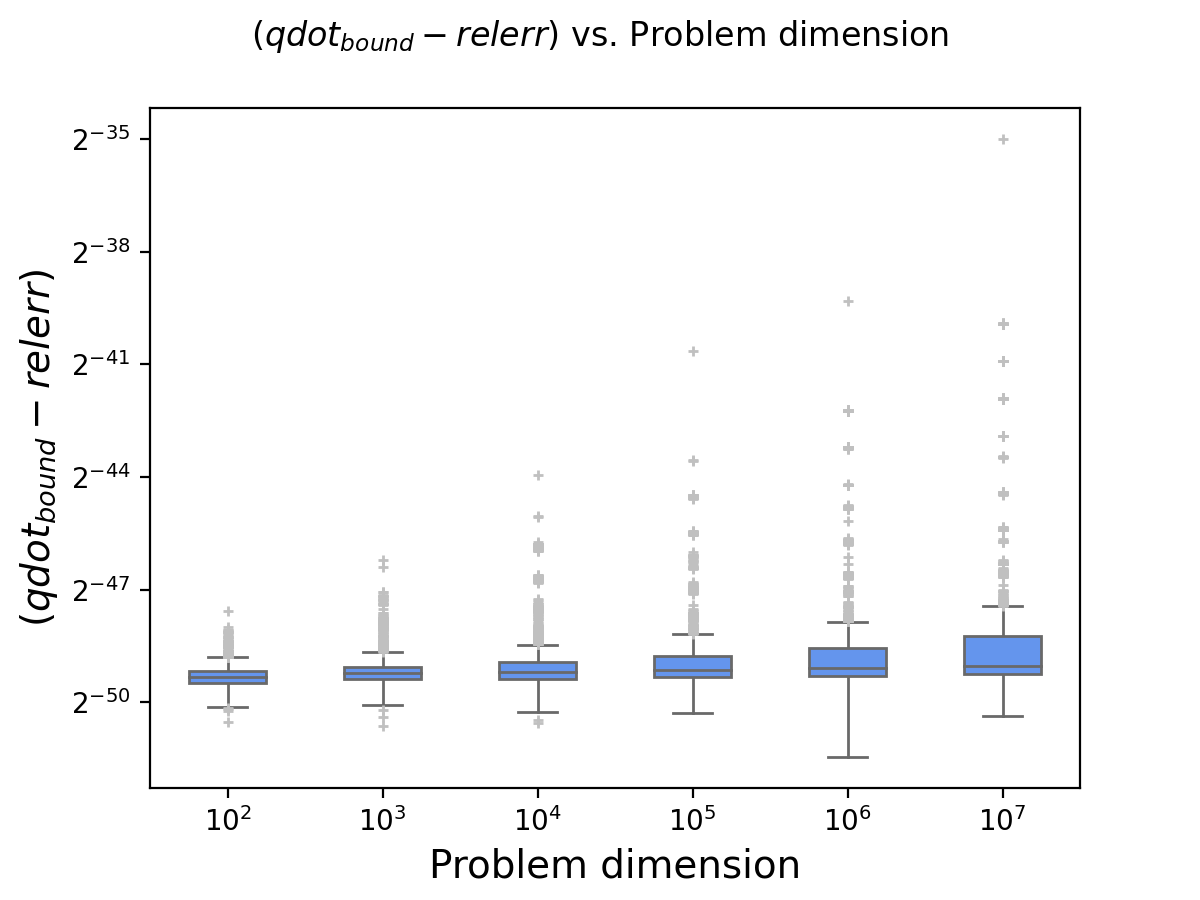}
    \caption{Distribution $\mathcal{B}$. For each problem dimension, bar and whisker plots are generated from using \qdot \ with 1500 sampled array pairs.}
    \label{fig:verify-normal}
  \end{subfigure}
  \caption{\textbf{Empirically Verifying qdot Error Bound}. For each problem dimension and distribution, we provide a box and whisker plot for the value $(qdot_{bound} - relerr)$ after using \qdot \ to approximate the dot product and \texttt{cublasDdot} to compute the \dpr \ precision dot product for each sampled array.}
  \label{fig:verify-bound}
\end{figure}

In the following experiments, we evaluate approximation effectiveness, efficiency and speedup on our sampled data problem. We vary the \textbf{qdot} tolerance $\varepsilon$, from $\texttt{1e-15}$ to $\texttt{1e-1}$ successively by a factor of $10$, for data sampled from distributions $\mathcal{A}$ and $\mathcal{B}$, each using three different parameter values, $t \in \{5, 9, 13\}$. The experiments were carried out for problems of dimension $\{ \texttt{1e5}, \texttt{1e6}, \texttt{1e7}, \texttt{1e8}\}$, and for each \textbf{qdot} tolerance value, we sample 50 pairs of arrays. Here, the results were similar for both distributions so we omit results for $\mathcal{A}$. 
%For each problem dimension we sampled 50 array pairs from the distribution in 2, approximated their dot product with \textbf{qdot}, and computed the relative approximation error, denoted $relerr$. 
%We plot the average value of $relerr$ as well as the minimum and maximum value of $relerr$ observed over all samples when the qdot tolerance value is varied from $[\texttt{1e-16}, \texttt{1e-1}]$. 

Figure~\ref{fig:effectiveness} illustrates the results for approximation effectiveness, obtained by computing the relative error, $relerr$, between \textbf{qdot} and the \dpr \ precision dot product. For the plot of each problem dimension, the solid line represents the mean and the shaded region of the same color represents the min and max corresponding to that \textbf{qdot} tolerance. In each of these plots, we include a dashed line denoting the qdot tolerance. The approximation effectiveness is then the distance, or gap, between the \textbf{qdot} tolerance and these plots. While the approximation effectiveness appears to be somewhat dependent on the array dimension, particularly for small $t$, we find that the distribution parameter, $t$, plays a larger role in the approximation effectiveness. In particular, when sampling from a wide enough exponent range, the approximation effectiveness approaches the optimal value and is essentially the same for all of the array dimensions we tested.

\begin{figure}[!t]
  \begin{subfigure}[b]{0.32\textwidth}
    \includegraphics[width=\textwidth]{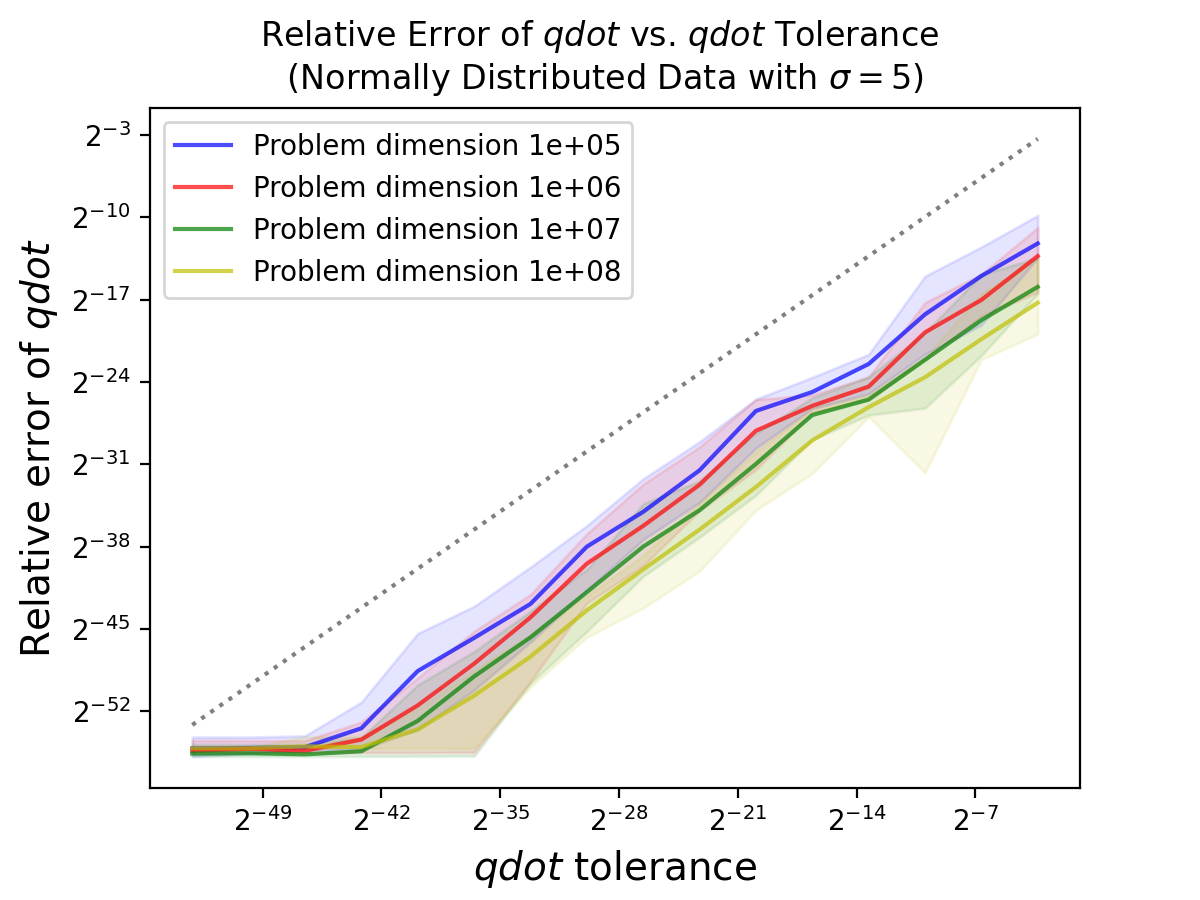}
    \caption{Distribution $\mathcal{B}$ with $t = 5$.}
    \label{fig:effect-normal-a}
  \end{subfigure}
  \hfill
  \begin{subfigure}[b]{0.32\textwidth}
    \includegraphics[width=\textwidth]{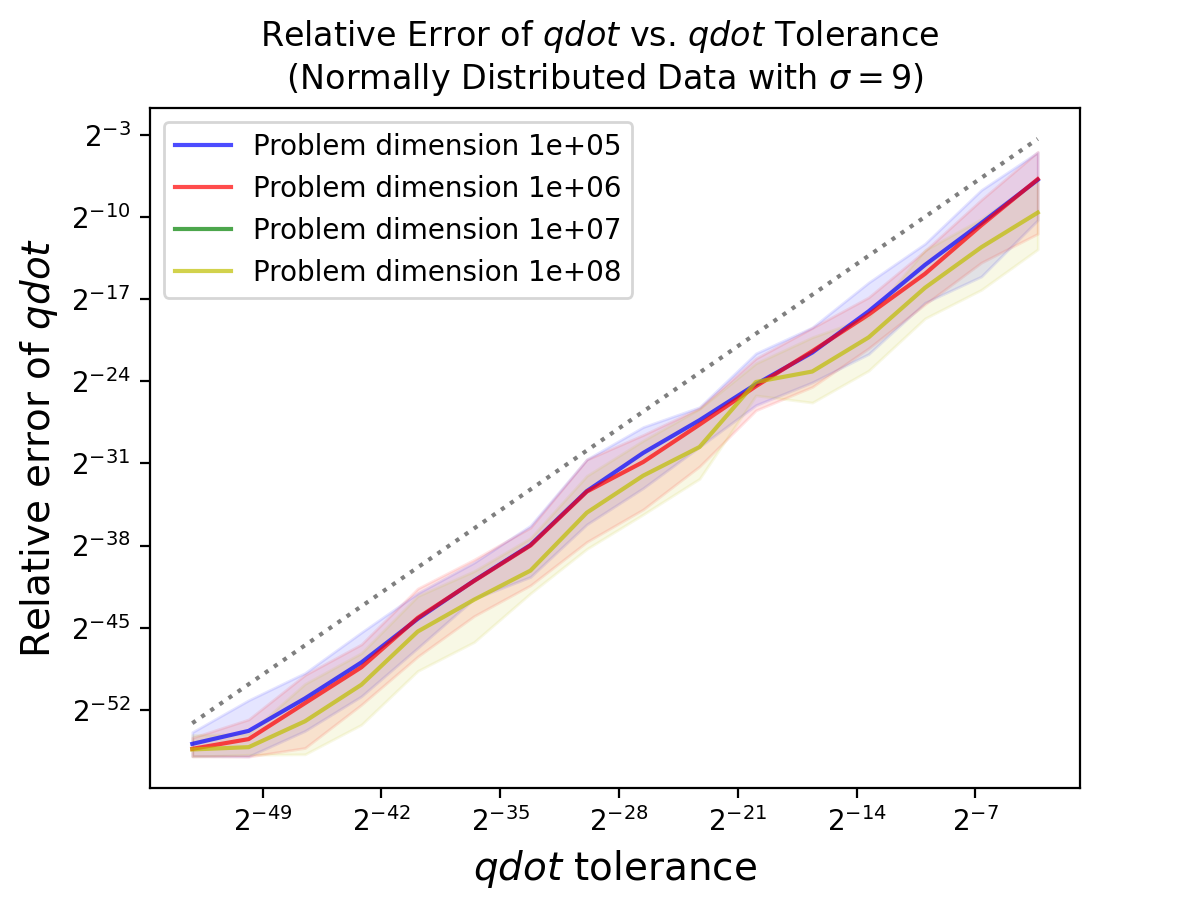}
    \caption{Distribution $\mathcal{B}$ with $t = 9$.}
    \label{fig:effect-normal-b}
  \end{subfigure}
  \hfill
  \begin{subfigure}[b]{0.32\textwidth}
    \includegraphics[width=\textwidth]{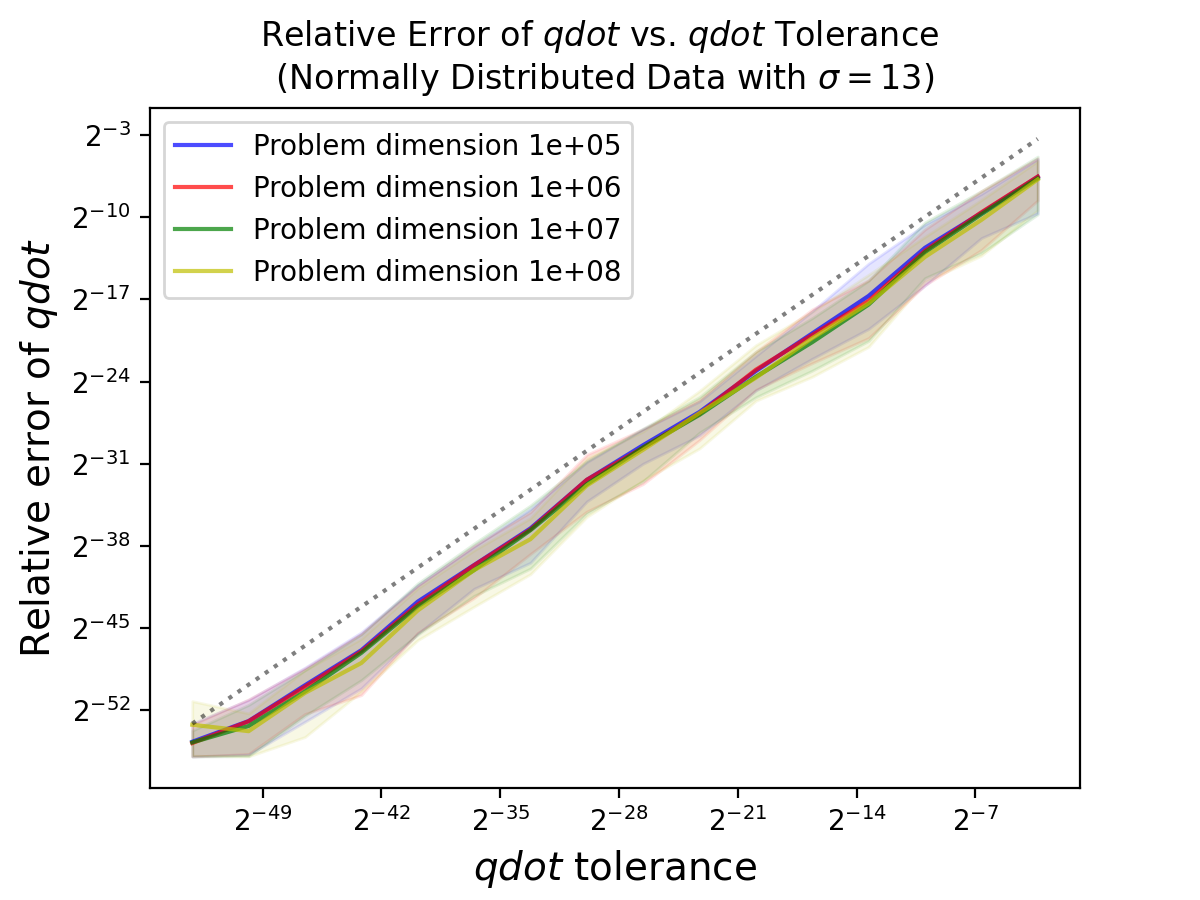}
    \caption{Distribution $\mathcal{B}$ with $t = 13$.}
    \label{fig:effect-normal-c}
  \end{subfigure}
  \caption{\textbf{qdot Approximation Effectiveness Primarily Dependent on Distribution of Floating Point Exponents}. Plots of the relative approximation error as a function of qdot tolerance are provided for varying problem dimensions across different array distributions. The dashed line represents the value of the qdot tolerance. The gap between the dashed line and each solid line corresponds to the approximation effectiveness. As the floating point exponents are sampled from distributions with a wider range, the relative error approaches the qdot tolerance value (dashed line).}
  \label{fig:effectiveness}
\end{figure}

Figure~\ref{fig:speedup} illustrates the results for approximation speedup. From the definition in (\ref{eq:approx-speedup}), this value is given by $\frac{\text{\dpr \ precision dot product run time}}{\text{\textbf{qdot} computation phase run time}}$ in our experiments. 
%\purple{The approximation speedup indicates an upper bound on the speedup achievable by \textbf{qdot} if the parameter selection algorithm did not have any cost.} 
As expected by the design of our approximate kernel, the approximation speedup is dependent on the distribution of the floating point exponents, array size, and \textbf{qdot} tolerance. The results further indicate that the computation phase of \textbf{qdot} is capable of providing speedup over the \dpr \ precision dot product. 

In evaluating the approximation efficiency, we observed that for all the test cases considered in these experiments, the computed approximation efficiency was essentially constant at roughly $\frac{1}{40} = 0.025$. As a result, we omit the figures here. This value of approximation efficiency indicates that the parameter selection algorithm is expensive compared to the \dpr \ precision dot product, and precludes its practical use in applications at runtime.
However, there is potential to improve the efficiency of \textbf{qdot} by designing alternative parameter selection algorithms. The most costly step of the current parameter selection algorithm is the sorting step in line 5 of Algorithm~\ref{alg:qdot-param}. Designing a parameter selection algorithm that either reduces this cost or does not rely on sorting could result in a more efficient approximation technique, albeit at some cost to its approximation effectiveness.
%However, we note that the cost of the parameter selection phase dominates the total time of applying \textbf{qdot}. In particular, we observed that running the full \textbf{qdot} algorithm is between 25 and 40 times slower than the \dpr \ precision dot product depending on the problem dimension and distribution from which the floating point exponents are sampled. This is not surprising as approximation efficiency was not the primary focus in the design of our approach. However, %alternative techniques could be developed to design a more efficient parameter selection algorithm.
%However, as the goal of this parameter selection phase was focused on effectiveness 
%we believe there is potential to improve the efficiency of \textbf{qdot} by designing alternative parameter selection algorithms. The most costly step of the current parameter selection algorithm is the sorting step in line 5 of Algorithm~\ref{alg:qdot-param}. Designing a parameter selection algorithm that either reduces this cost or does not rely on sorting could result in a more efficient approximation technique, albeit at the cost of effectiveness.

\begin{figure}[!t]
\iffalse %%%%%%%%%%%%%%%%%%%%%%%%%%%%%%%
  \begin{subfigure}[b]{0.32\textwidth}
    \includegraphics[width=\textwidth]{figures/comp_ratio_vs_tol_by_dim_uniform_10_plot.png}
    \caption{Distribution $\mathcal{A}$ with $t = 10$.}
    \label{fig:speedup-uniform-a}
  \end{subfigure}
  \hfill
  \begin{subfigure}[b]{0.32\textwidth}
    \includegraphics[width=\textwidth]{figures/comp_ratio_vs_tol_by_dim_uniform_25_plot.png}
    \caption{Distribution $\mathcal{A}$ with $t = 25$.}
    \label{fig:speedup-uniform-b}
  \end{subfigure}
  \hfill
  \begin{subfigure}[b]{0.32\textwidth}
    \includegraphics[width=\textwidth]{zfp_stability_paper copy/figures/comp_ratio_vs_tol_by_dim_uniform_40_plot.png}
    \caption{Distribution $\mathcal{A}$ with $t = 40$.}
    \label{fig:speedup-uniform-c}
  \end{subfigure}
  \hfill
\fi %%%%%%%%%%%%%%%%%%%%%%%%%%%%%%%%%%%
  \begin{subfigure}[b]{0.32\textwidth}
    \includegraphics[width=\textwidth]{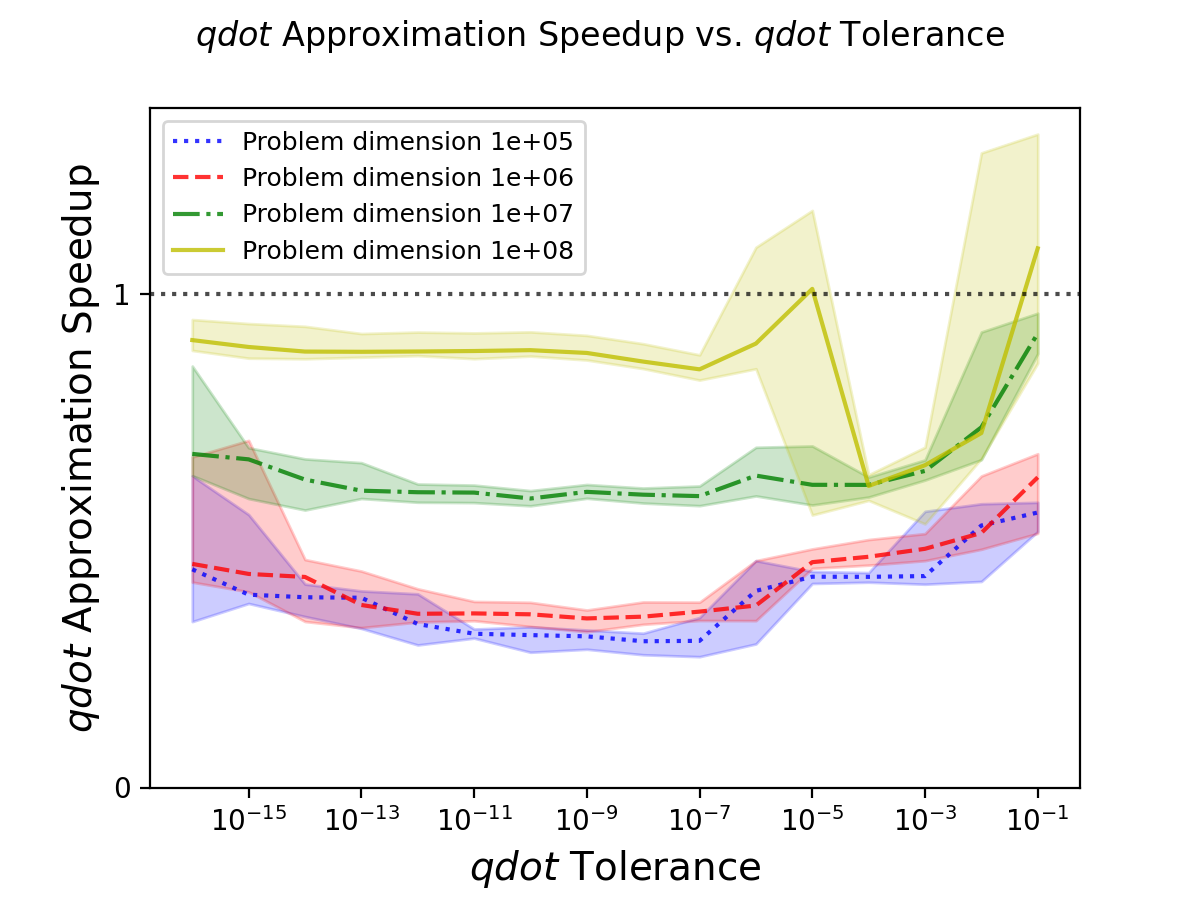}
    \caption{Distribution $\mathcal{B}$ with $t = 5$.}
    \label{fig:speedup-normal-a}
  \end{subfigure}
  \hfill
  \begin{subfigure}[b]{0.32\textwidth}
    \includegraphics[width=\textwidth]{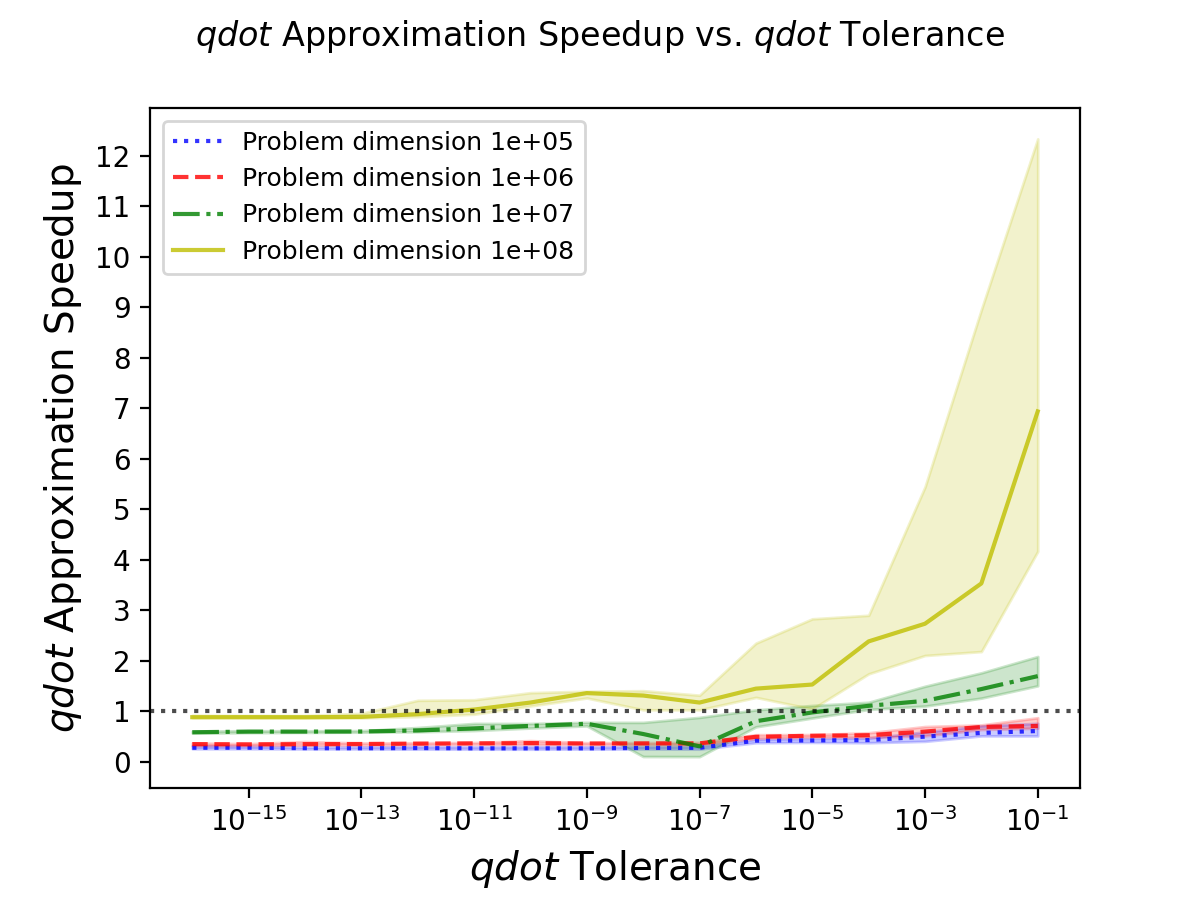}
    \caption{Distribution $\mathcal{B}$ with $t = 9$.}
    \label{fig:speedup-normal-b}
  \end{subfigure}
  \hfill
  \begin{subfigure}[b]{0.32\textwidth}
    \includegraphics[width=\textwidth]{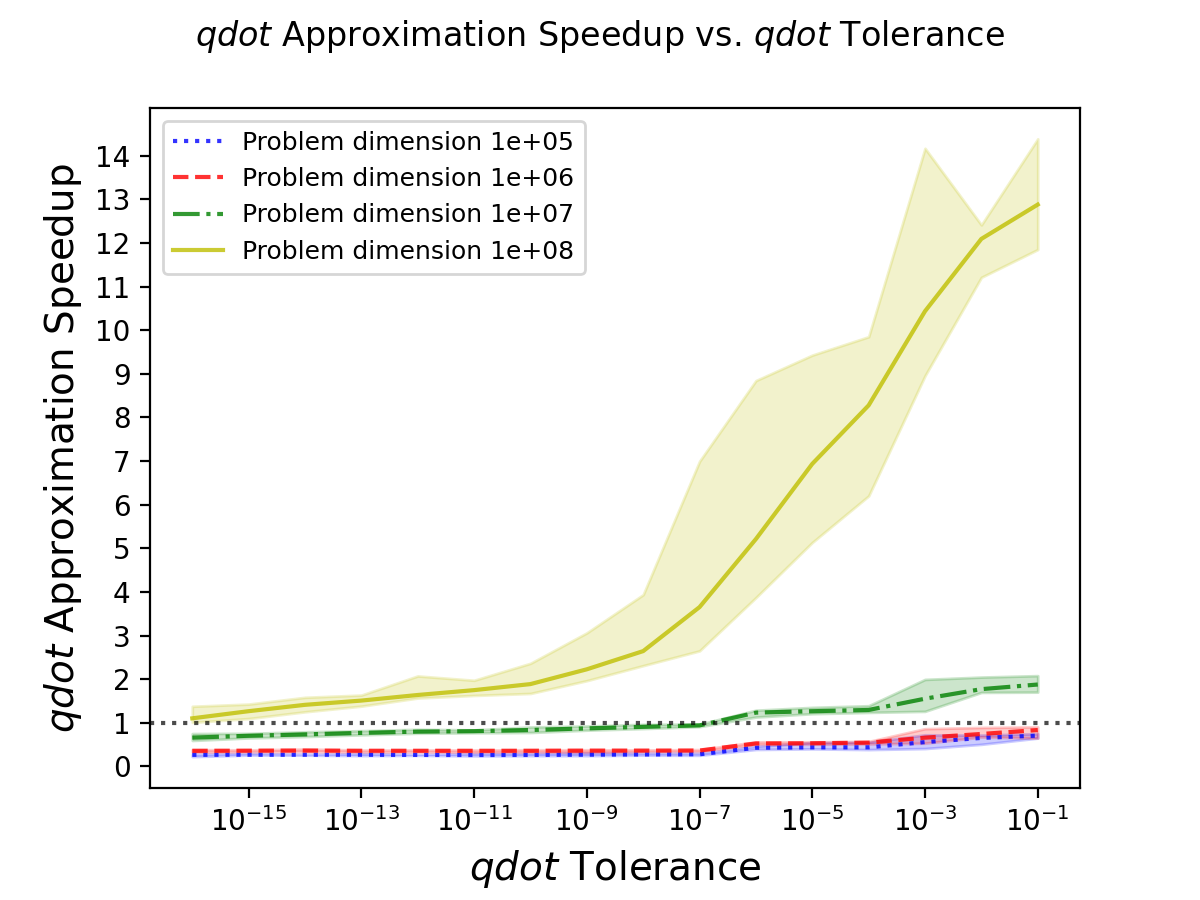}
    \caption{Distribution $\mathcal{B}$ with $t = 13$.}
    \label{fig:speedup-normal-c}
  \end{subfigure}
  \caption{\textbf{qdot Approximation Speedup Dependent on Problem Size and Exponent Distribution}. By varying problem size, \qdot \  tolerance, and distribution parameter, we can visualize how the approximation speedup of \qdot \ is dependent on these three factors. For each problem dimension, the line corresponds to the mean approximation speedup while the shaded region indicates the minimum and maximum approximation speedup at each \qdot \ tolerance value.}
  \label{fig:speedup}
\end{figure}

\subsection{qdot with Conjugate Gradient}
\label{ex:experiment-cg}
Conjugate Gradient (CG) \cite{golub13,Hestenes1952MethodsOC} is a popular iterative technique for solving the linear system $\bm{A} \bm{x} = \bm{b}$, given a positive definite matrix $\bm{A} \in \mathbb{R}^{n \times n}$ and a vector $\bm{b} \in \mathbb{R}^n$. Recent works have focused on introducing approximation into linear solvers, such as CG, without compromising the performance or convergence properties of the solver \cite{menon2018}. 
%, approx. CG/lin solvers}. 
In the following experiments, we evaluate the impact of introducing approximation into CG
%In these experiments, we test an approximate linear solver where approximation is introduced 
by integrating \textbf{qdot} in a \dpr \ precision CG application. Excluding the matrix vector product, each iteration of CG requires the computation of two dot products. By replacing these dot products with \textbf{qdot}, we can tune the amount of approximation incorporated into the resulting approximate linear solver kernel by varying the value of the \textbf{qdot} error tolerance. %We use \textbf{qdot} within the HPCCG benchmark application to approximate the computation of two dot products in each iteration. 
Algorithm~\ref{alg:acg} presents pseudocode for our Approximate Conjugate Gradient (ACG) kernel, specifying the exact operations computed using \textbf{qdot}.
%and can be found in Algorithm~\ref{alg:acg}.
The implementation of Algorithm~\ref{alg:acg} was obtained by using our implementation of \textbf{qdot} in place of the \dpr \ precision dot products in HPCCG, a publicly available CG benchmark for a 3D chimney domain \cite{mantevo}. 

In our experiments with ACG, we solve problems of varying dimension with varying \textbf{qdot} and CG tolerances and visualize how the precision utilized by \textbf{qdot} affects the number of iterations required for convergence. We use the HPCCG application code to generate four different problems: Two 2D problems with dimension $100 \times 100 \times 1$ and $1000 \times 1000 \times 1$ and two 3D problems with dimension $100 \times 100 \times 10$ and $1000 \times 1000 \times 10$. Note that the matrix $\bm{A}$ generated for each problem is of dimension $n \times n$, where $n$ is equal to the product of the problem dimensions. Hence, solving these problems using CG requires computing dot products for arrays of dimension \texttt{1e4} and \texttt{1e6} for the 2D problems and \texttt{1e5} and \texttt{1e7} for the 3D problems. 
%\green{The underlying problem solved by HPCCG is rather simple since the benchmark is designed to evaluate performance on parallel architectures, rather than the complexity of the underlying problem.} 
For the examples evaluated here, a preconditioner was not needed as the matrices were relatively well-conditioned. 
%\purple{The condition number for the matrices in the 2D problems is $O(\texttt{1e0})$ and the condition number for the matrices in the 2D problems is $O(\texttt{1e2})$. }
%We used the \texttt{condest} function in MATLAB to estimate the condition number for each matrix $\bm{A}$ generated by HPCCG. For the two 2D problems the condition number of both matrices is $1.8421$ and for the two 3D problems the condition number of both matrices is $36.1133$. 
In all of these experiments, we used a fixed linear solver convergence tolerance of $\tau = \texttt{1e-8}$. These experiments seek to ascertain how much approximation can be introduced into the CG algorithm, using \textbf{qdot}, while preserving its convergence properties. Thus, we vary the \textbf{qdot} tolerance successively by a factor of $10$ from $\texttt{1e-16}$ to $\texttt{1e3}$, to identify the largest tolerance that resulted in ACG converging in the same number of iterations as the \dpr \ precision CG solver, which we denote simply by CG. 
%\purple{While tolerance values exceeding \texttt{1} seem rather extreme, results presented in  Section~\ref{sec:experiment-errorbound} demonstrated that this is not unusual, and could lead to an approximation with acceptable relative error. This is due to the fact that the error bound captures the worst-case scenario error for all arrays of size $n$. Thus, in some cases, the \textbf{qdot} tolerance could be loose, and yet provide an acceptable amount of approximation. }
%By varying the \textbf{qdot} tolerance parameter up to $\texttt{1e3}$ in these experiments, we ensured that we would satisfy feasible amounts of approximation in ACG.
\begin{center}
\begin{algorithm}[t]
\begin{algorithmic}[1]
\STATE{\textit{Inputs}: Array $\bm{b} \in \mathbb{R}^n$, symmetric positive definite matrix $\bm{A} \in \mathbb{R}^{n \times n}$, initial guess for solution $\bm{x}_0 \in \mathbb{R}^n$, CG error tolerance $\tau$, \textbf{qdot} error tolerance $\varepsilon$.}
\STATE{\textit{Output}: Array $\bm{x} \in \mathbb{R}^n$ satisfying $\| \bm{b} - \bm{A} \bm{x} \| \leq \tau$.}
\STATE{$\bm{r}_0 \gets \bm{b} - \bm{A} \bm{x}_0$, $\bm{p}_0 \gets \bm{r}_0$, $k \gets 0$ \hfill Initialize residual value and index}
\STATE{$c_0 \gets \textbf{qdot}(\bm{r}_0, \bm{r}_0, \varepsilon)$ \hfill Approximate dot product of residual using \textbf{qdot}}
\WHILE{$\| \bm{r}_{k} \| > \tau$}
\STATE{$\bm{q}_k \gets \bm{A} \bm{p}_k$ \hfill Compute matrix vector product}
\STATE{$\alpha_k \gets c_k / \textbf{qdot}(\bm{p}_k, \bm{q}_k, \varepsilon)$ \hfill Compute $\alpha_k$ using \textbf{qdot}}
\STATE{$\bm{x}_{k+1} \gets \bm{x}_{k} + \alpha_k \bm{p}_k$ \hfill Update guess for solution}
\STATE{$\bm{r}_{k+1} \gets \bm{r}_{k} - \alpha_k \bm{q}_k$ \hfill Update residual}
\STATE{$c_{k+1} \gets \textbf{qdot}(\bm{r}_{k+1}, \bm{r}_{k+1}, \varepsilon)$ \hfill Approximate dot product of residual using \textbf{qdot}}
\STATE{$\beta_k \gets c_{k+1} / c_k$ \hfill Compute $\beta_k$}
\STATE{$\bm{p}_{k+1} \gets \bm{r}_{k+1} + \beta_k \bm{p}_k$ \hfill Update $\bm{p}_k$}
\STATE{$k \gets k+1$ \hfill Increment $k$ \ }
\ENDWHILE
\end{algorithmic}
\caption{\textbf{ACG}: Approximate Conjugate Gradient Kernel with \textbf{qdot}}
\label{alg:acg}
\end{algorithm}
\end{center}
%Using the previous plots, we can now identify the \textbf{qdot} tolerance value for each problem size and \textbf{qdot} tolerance that results in the same performance as the original CG application. 

By choosing the largest \textbf{qdot} tolerance, $\varepsilon$, yielding the same number of iterations as CG, the resulting ACG solver utilizes the maximum amount of approximation that  \textbf{qdot} can introduce, without impacting the convergence of the original CG solver. For the 2D problems of dimension $100 \times 100 \times 1$ and $1000 \times 1000 \times 1$, this corresponded to $\varepsilon = \texttt{1e0}$ and $\varepsilon = \texttt{1e3}$, respectively. For both 3D problems, of dimension $100 \times 100 \times 10$ and $1000 \times 1000 \times 10$, this corresponded $\varepsilon$ $\texttt{1e2}$. We note that \textbf{qdot} tolerance values exceeding \texttt{1.0} are not unusual since the error bound for \textbf{qdot} captures the worst-case scenario error for all arrays of size $n$. Thus, in some cases, the \textbf{qdot} tolerance could be loose, and yet provide an acceptable amount of approximation.

To illustrate the approximation introduced by \textbf{qdot}, we plot a breakdown of the percentage of each precision used by \textbf{qdot} at each iteration of Algorithm~\ref{alg:acg}. Figure~\ref{fig:hpccg-precision}(a) - (d) presents results for computing $\bm{r^Tr}$ and $\bm{p^TAp}$ for the 2D problems and Figure~\ref{fig:hpccg-precision}(e) - (h) presents corresponding results for the 3D problems. For the 2D problems, the results indicate that a significant fraction of the dot product computation can be performed in \hpr \ precision for the smaller dimension problem. For the larger dimension 2D problem, the results suggest a larger fraction could be \ppr. For the 3D problems, we observe that a majority of the computation can be performed in \hpr \  precision or \ppr. As the problem size increased, we observe the need to compute more components of the dot product in \spr \ precision. We note that while these results are not a rigorous analysis of the effect of approximating operations in CG, they provide insights into how approximating the dot product operation may impact CG. In particular, they underscore the importance of quantization in preserving the convergence properties of CG. For example, it is easy to see from the results that for larger 3D problems, one cannot guarantee that the convergence rate of CG would be preserved by using solely \hpr \ precision dot products.

%Additionally, for the problem of dimension $100 \times 100 \times 1$, we plot the precision corresponding used by each component at each iteration when transformed back to the problem mesh of dimension $100 \times 100 \times 1$. This plot provides insight into which regions of the mesh are amenable to higher levels of quantization.

\begin{figure}[!tbp]
  \begin{subfigure}[b]{0.47\textwidth}
    \includegraphics[width=\textwidth]{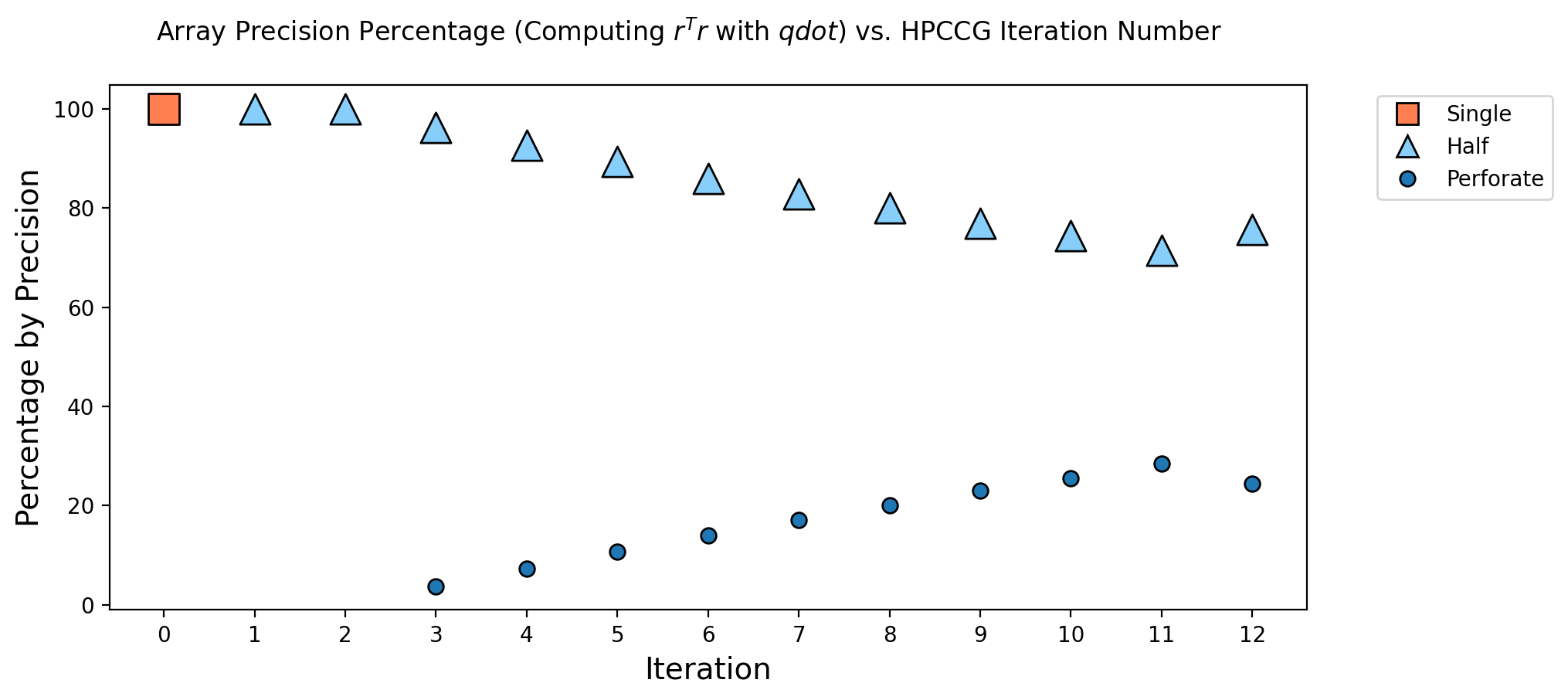}
    \caption{Computation of $\bm{r}^{\intercal} \bm{r}$ in ACG for problem\\ of Dimension $100 \times 100 \times 1$.}
    \label{fig:rtr-10000}
  \end{subfigure}
  \hfill
  \begin{subfigure}[b]{0.47\textwidth}
    \includegraphics[width=\textwidth]{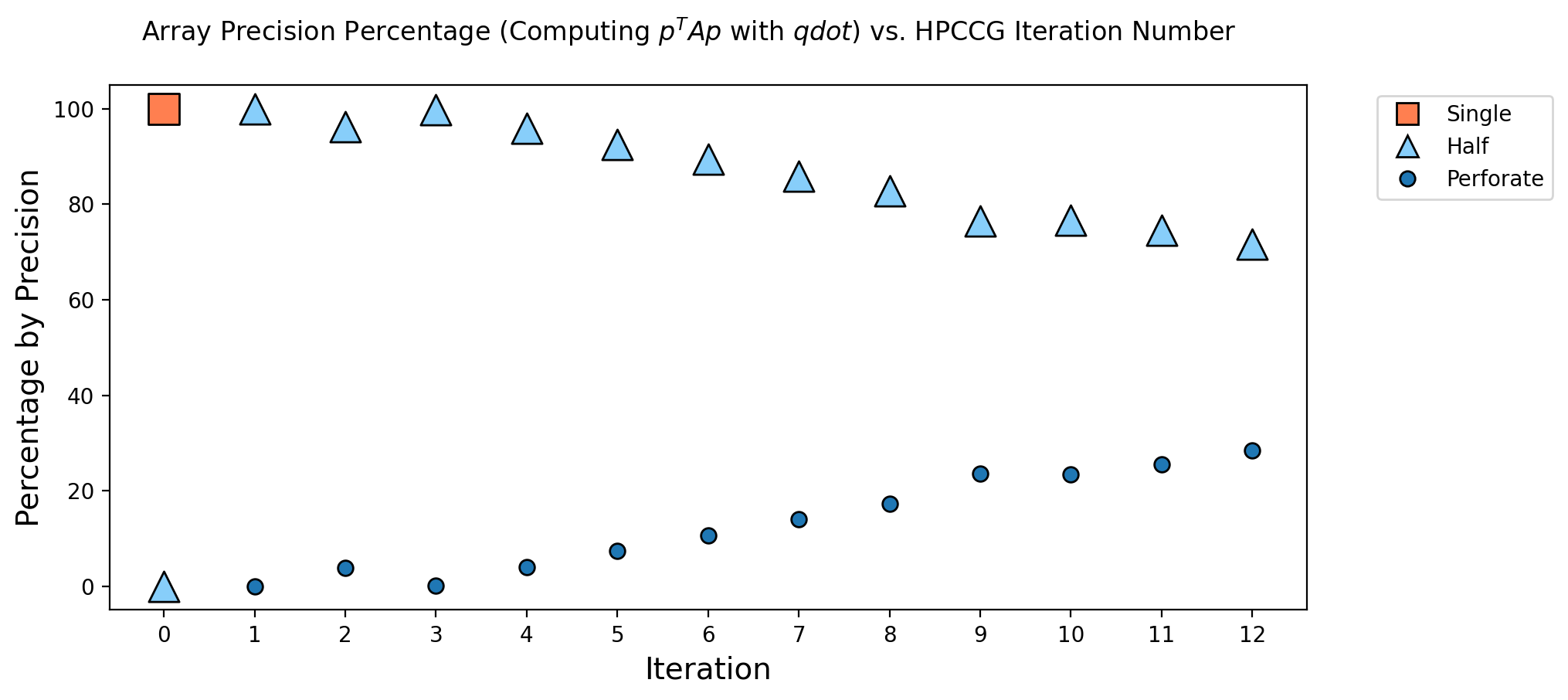}
    \caption{Computation of $\bm{p}^{\intercal} \bm{A} \bm{p}$ in ACG for problem\\ of Dimension $100 \times 100 \times 1$.}
    \label{fig:pAp-10000}
  \end{subfigure}
  \hfill
  \begin{subfigure}[b]{0.47\textwidth}
  \includegraphics[width=\textwidth]{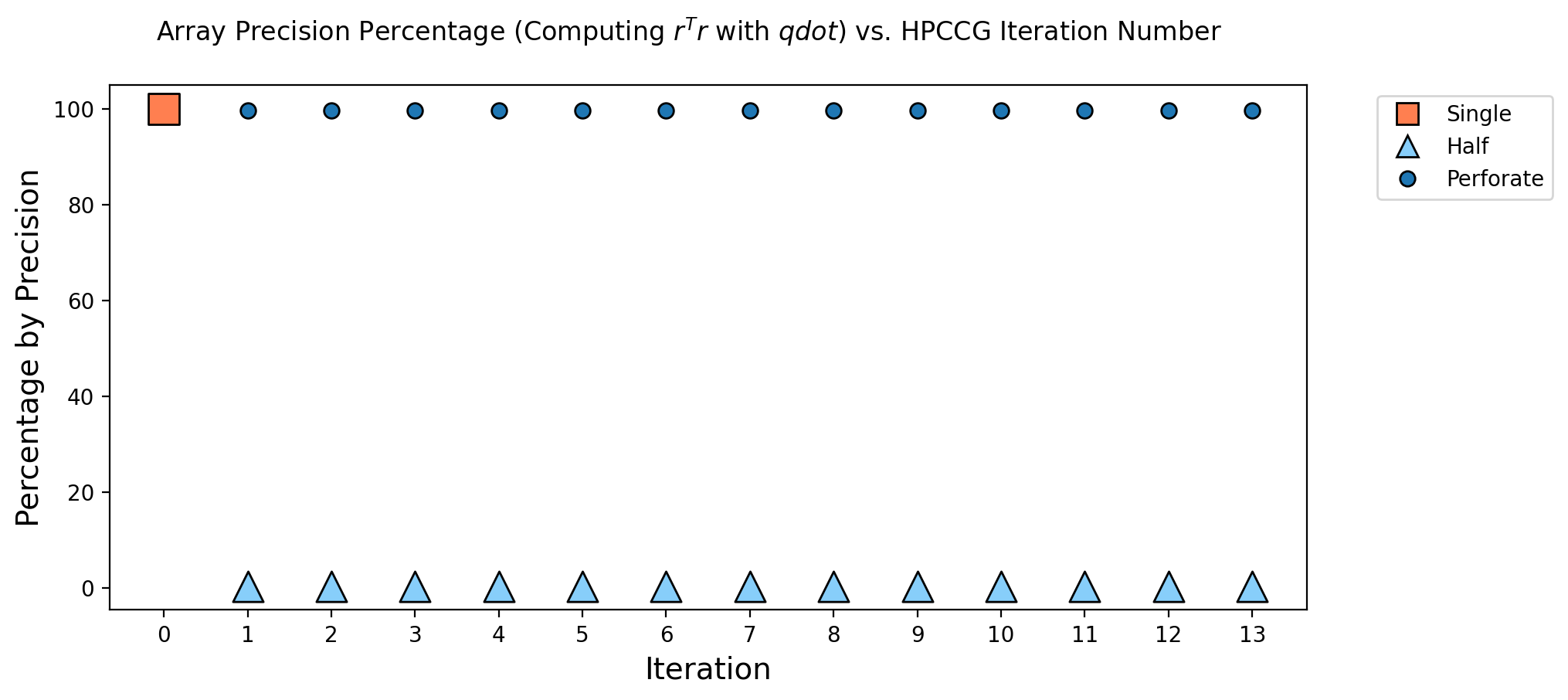}
    \caption{Computation of $\bm{r}^{\intercal} \bm{r}$ in ACG for problem\\ of Dimension $1000 \times 1000 \times 1$.}
  \end{subfigure}
  \hfill
  \begin{subfigure}[b]{0.47\textwidth}
    \includegraphics[width=\textwidth]{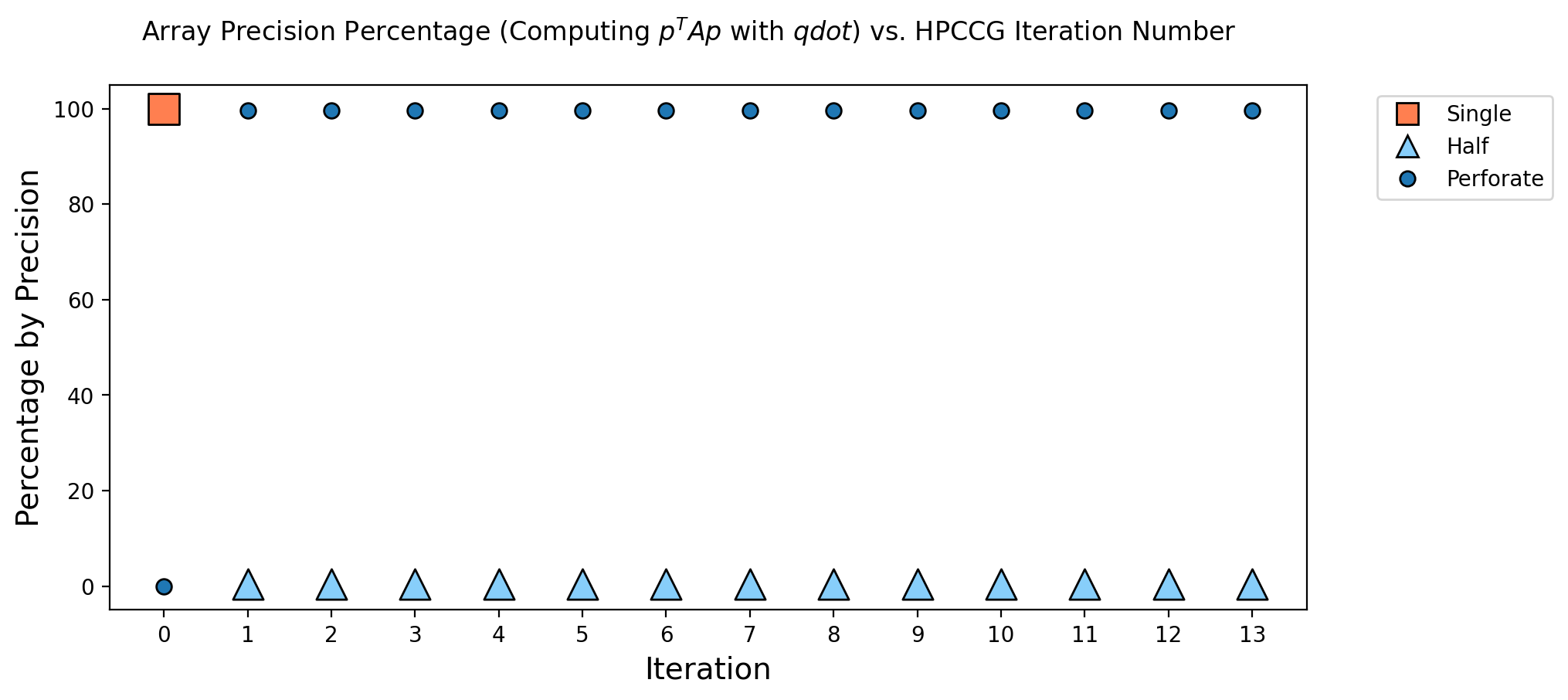}
    \caption{Computation of $\bm{p}^{\intercal} \bm{A} \bm{p}$ in ACG for problem\\ of Dimension $1000 \times 1000 \times 1$.}
  \end{subfigure}
  \hfill
  \begin{subfigure}[b]{0.47\textwidth}
  \includegraphics[width=\textwidth]{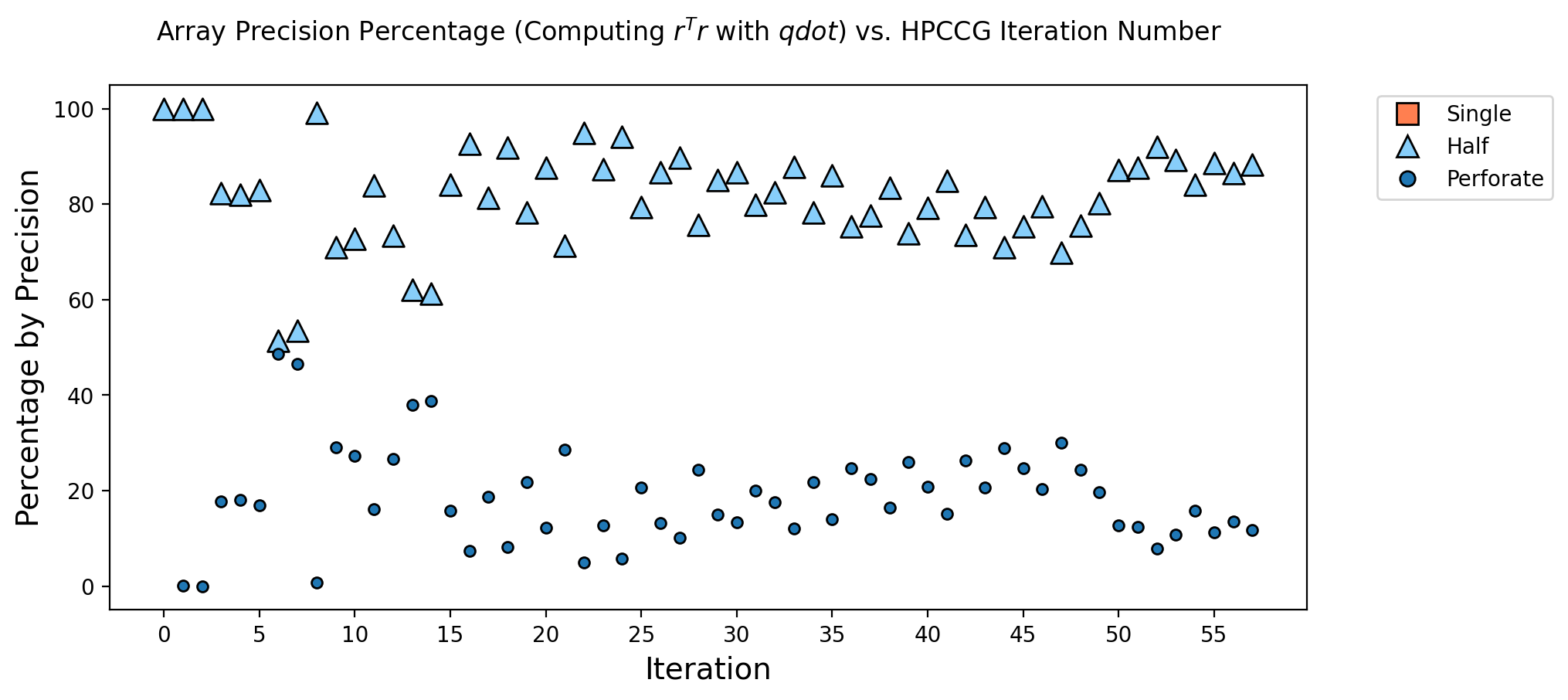}
    \caption{Computation of $\bm{r}^{\intercal} \bm{r}$ in ACG for problem\\ of Dimension $100 \times 100 \times 10$.}
  \end{subfigure}
  \hfill
  \begin{subfigure}[b]{0.47\textwidth}
    \includegraphics[width=\textwidth]{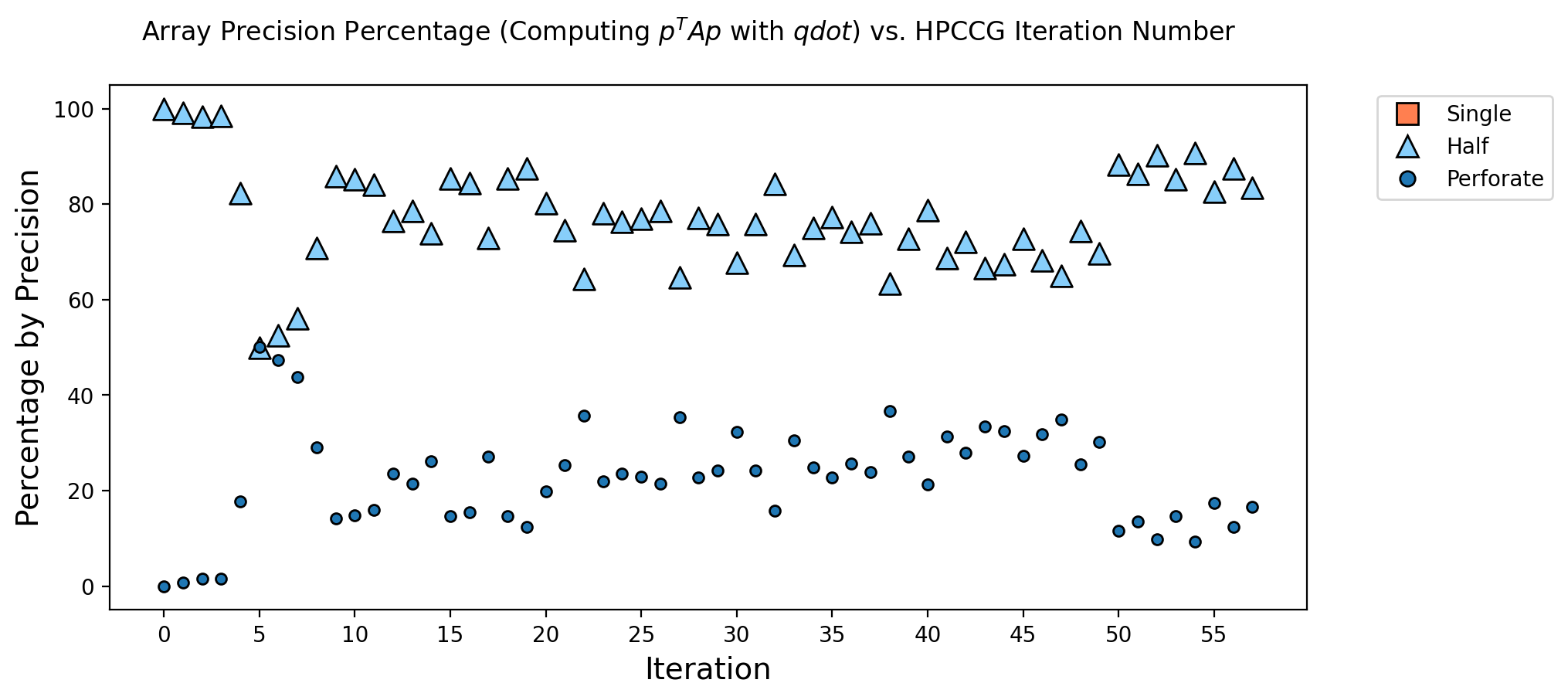}
    \caption{Computation of $\bm{p}^{\intercal} \bm{A} \bm{p}$ in ACG for problem\\ of Dimension $100 \times 100 \times 10$.}
  \end{subfigure}
  \hfill
  \begin{subfigure}[b]{0.47\textwidth}
  \includegraphics[width=\textwidth]{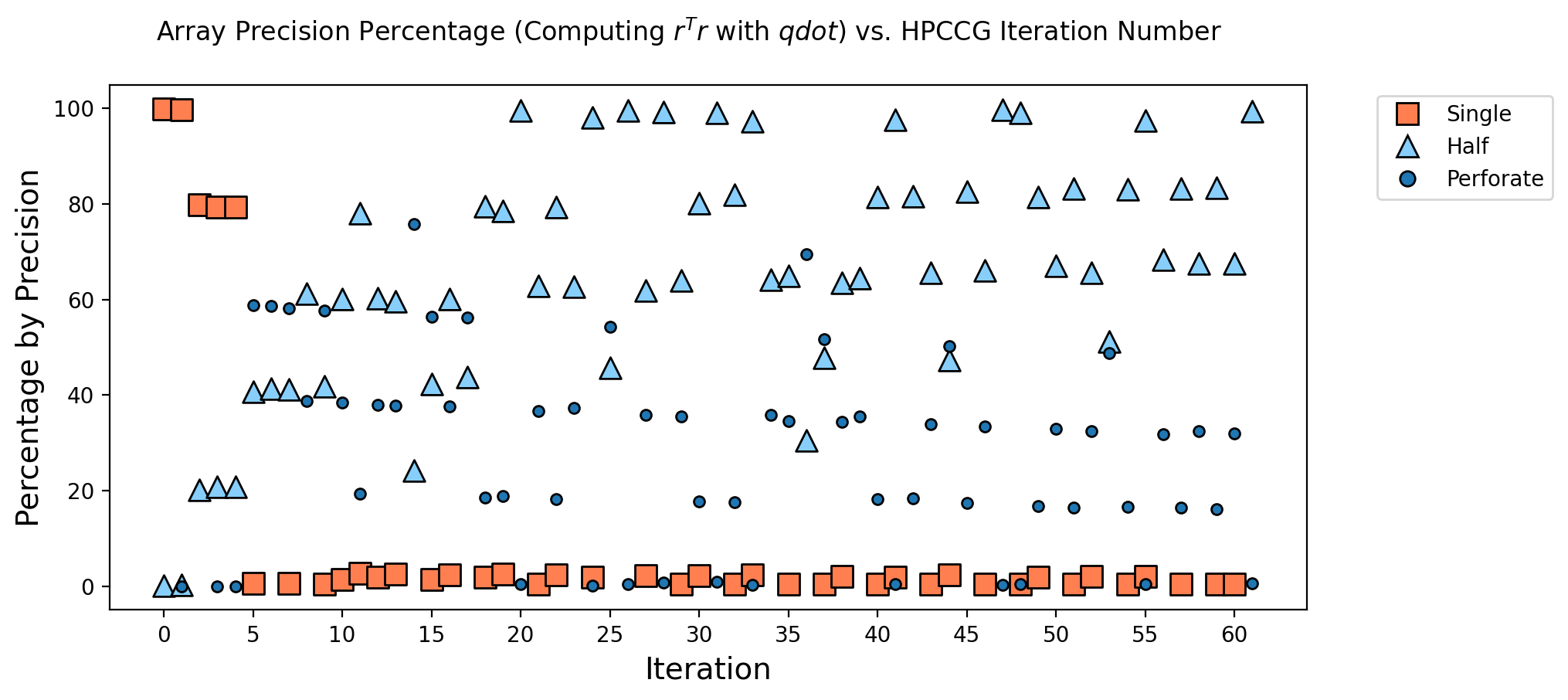}
    \caption{Computation of $\bm{r}^{\intercal} \bm{r}$ in ACG for problem\\ of Dimension $1000 \times 1000 \times 10$.}
  \end{subfigure}
  \hfill
  \begin{subfigure}[b]{0.47\textwidth}
    \includegraphics[width=\textwidth]{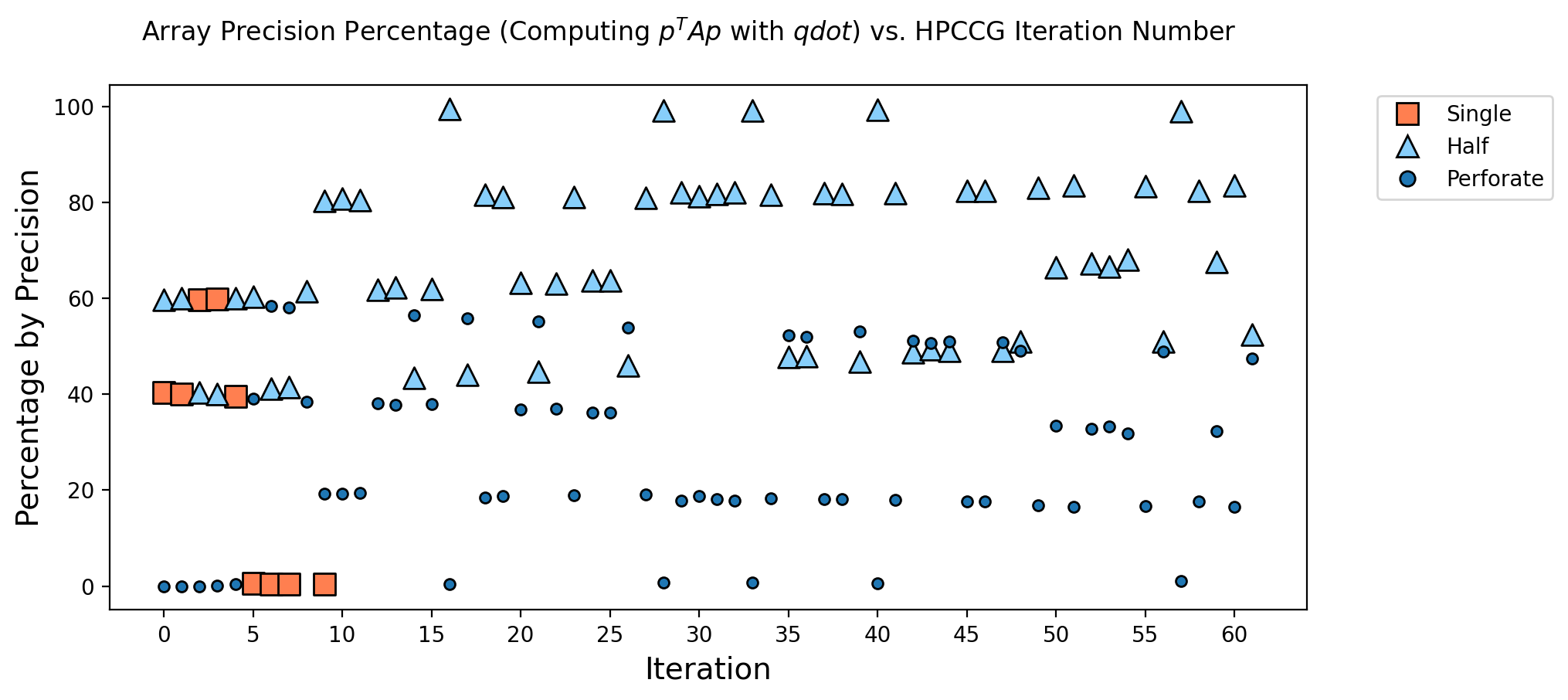}
    \caption{Computation of $\bm{p}^{\intercal} \bm{A} \bm{p}$ in ACG for problem\\ of Dimension $1000 \times 1000 \times 10$.}
  \end{subfigure}
  \caption{\textbf{qdot Precision at Each Iteration of ACG}. 
  We plot the percentage of components at each precision within \textbf{qdot} at each iteration of ACG when using the largest \qdot \ tolerance for which ACG converged in the same number of iterations as CG. Plots in (a) -- (d) correspond to 2D problems while (e) -- (h) correspond to 3D problems. In any iteration where zero components were assigned a precision, that precision is not plotted at that iteration (to differentiate from very small percentages).}
  \label{fig:hpccg-precision}
\end{figure}

\begin{figure}[!t]
  \begin{subfigure}[b]{\textwidth}
    \includegraphics[width=\textwidth,trim={0 14cm 0 0}]{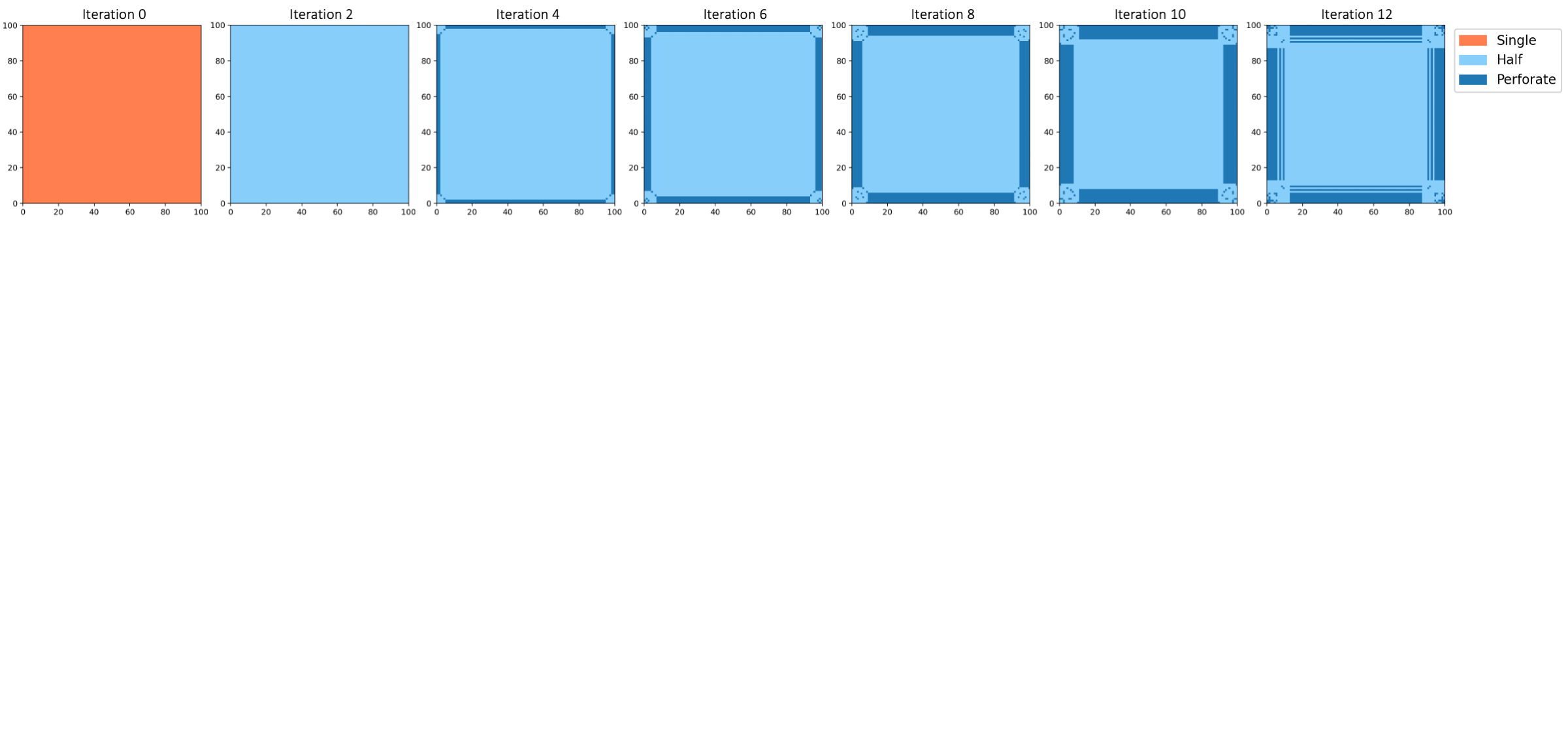}
    \caption{Precision in \qdot \ computation of $\bm{r}^{\intercal} \bm{r}$ in ACG.}
    \label{fig:rtr-qdot-timelapse}
  \end{subfigure}
  \begin{subfigure}[b]{\textwidth}
    \includegraphics[width=\textwidth,trim={0 14cm 0 0}]{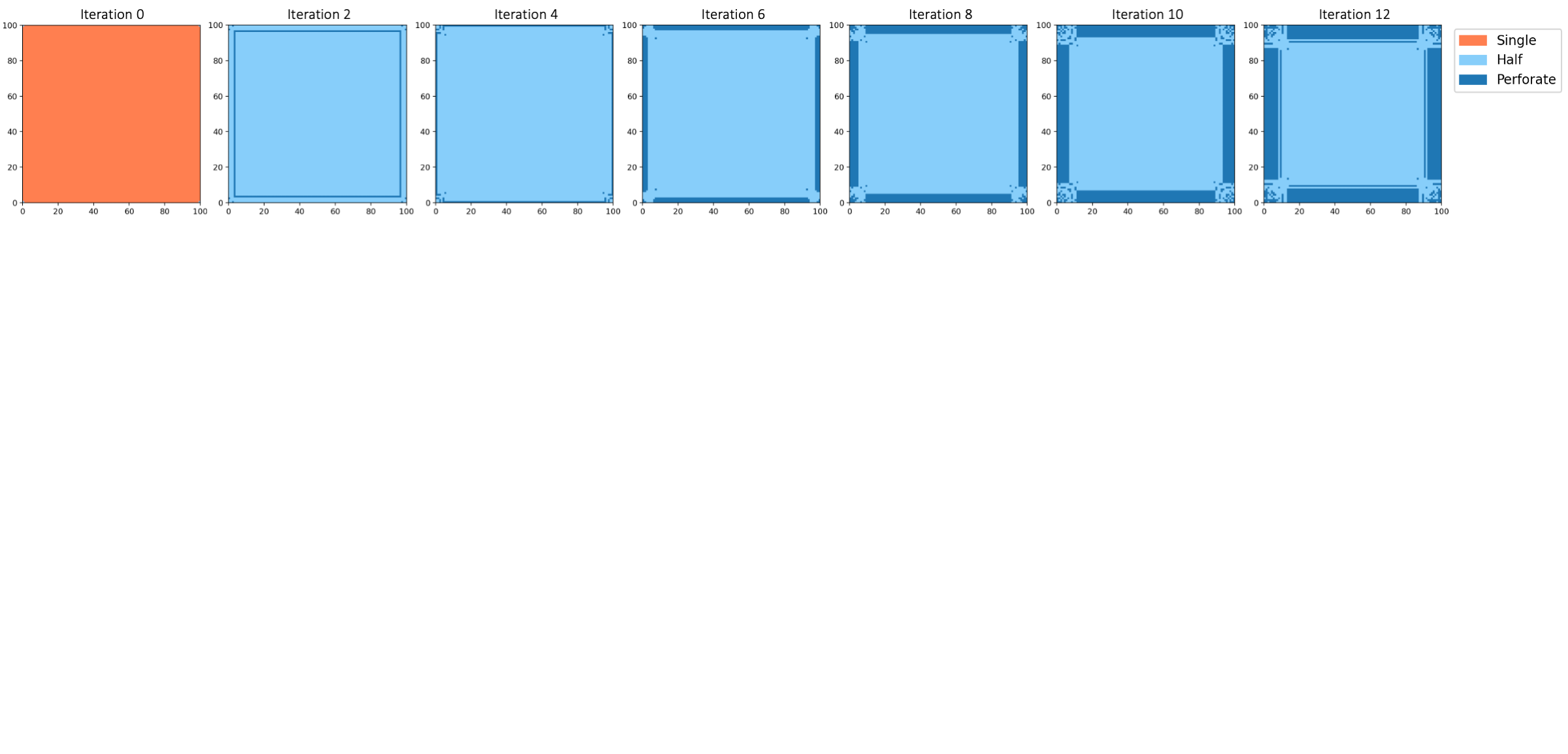}
    \caption{Precision in \qdot \ computation of $\bm{p}^{\intercal} \bm{A} \bm{p}$ in ACG.}
    \label{fig:pAp-qdot-timelapse}
  \end{subfigure}
  \caption{\textbf{Timelapse of Precision Identified by qdot in ACG}. These plots provide a snapshot of the precision used for each component in \textbf{qdot} in the computation of (a) $\bm{r}^{\intercal} \bm{r}$ and (b) $\bm{p}^{\intercal} \bm{A} \bm{p}$ corresponding to it's position on the $100 \times 100$ mesh, at progressive iterations of the CG problem of dimension $100 \times 100 \times 1$. This timelapse corresponds to the precision distribution plot in Figures~\ref{fig:rtr-10000} and \ref{fig:pAp-10000}.
  }
  \label{fig:qdot-timelapse}
\end{figure}

To complete the CG experiments, we provide time lapses of the precision corresponding to each component of the two \textbf{qdot} calls in ACG, on the 2D problem with dimension $100 \times 100 \times 1$.
Figure~\ref{fig:rtr-qdot-timelapse} provides results for computing $\bm{r}_k^{\intercal} \bm{r}_k$ in line 10 of Algorithm~\ref{alg:acg}, and Figure~\ref{fig:pAp-qdot-timelapse} illustrates computing $\bm{p}_k^{\intercal} \bm{q}_k$ line 7 of Algorithm~\ref{alg:acg}.
%The visualization corresponding to precision used by \qdot \ when computing $\bm{r}_k^{\intercal} \bm{r}_k$ in line 10 of Algorithm~\ref{alg:acg} is provided in Figure~\ref{fig:rtr-qdot-timelapse} while the precision used by \qdot \ when computing $\bm{p}_k^{\intercal} \bm{q}_k$ line 7 of Algorithm~\ref{alg:acg} is provided in Figure~\ref{fig:pAp-qdot-timelapse}. 
These figures reveal some structure in the precision distribution of the data used to compute the approximate dot product, as the CG iterations progress. For this particular linear system under consideration, we observe bands of \ppr \ components that increase in size from the boundary of the mesh, as the iteration count increases. As ACG approaches the final iteration some \hpr \ precision begins to mix with the bands perforated components. 
%\textcolor{red}{(James: Can you confirm what we are visualizing in these time-lapse plots? I suspect it is the residual for the $\bm{r}^{\intercal} \bm{r}$ plot, but I am not sure what is for the $\bm{p}^{\intercal} \bm{A} \bm{p}$ plot. Also, did the bands stop growing from the boundary after some iterations, or did they continue to grow (for both figures)? This will help determine if it is a property of CG or a property of the problem (ie. boundary conditions).)}

%These figures indicate that the structure of the problem may provide information on which components are more significant for approximating the dot products required in CG as bands of perforated components increase in size from the boundary of the mesh as the iteration count increases. \textcolor{red}{(Daniel: Maybe mention some of your thoughts about boundary conditions affecting/providing the structure of the different precisions found in the timelapse plots)} We note that similar trends and patterns hold for larger dimensional CG problems, however, due to the larger mesh sizes and increased number of iterations a timelapse is difficult to visualize in a compact format.

%\textcolor{red}{Add timelapse plot for $100 \times 100 \times 1$ problem}

\subsection{qdot with Power Method}
\label{ex:experiment-power-method}
For the final experiment, we integrate \textbf{qdot} into the Power Method \cite{mises1929practice, wilkinson1988algebraic} to compute the eigenvector-eigenvalue pair corresponding to the largest eigenvalue of a randomly generated graph Laplacian matrix. Like the CG method, the Power Method is an iterative technique where two dot products occur at each iteration. Hence, we experiment with an Approximate Power Method (APM) that is designed by replacing the dot products in the Power Method with \textbf{qdot}. For clarity, we provide pseudocode for APM in Algorithm~\ref{alg:apm}. 
\begin{center}
\begin{algorithm}[t]
\begin{algorithmic}[1]
\STATE{\textit{Inputs}: Array $\bm{b} \in \mathbb{R}^n$, matrix $\bm{A} \in \mathbb{R}^{n \times n}$, initial guess for eigenvector $\bm{x}_0 \in \mathbb{R}^n$ with $\| \bm{x}_0 \| = 1$, Power Method error tolerance $\tau$, \textbf{qdot} error tolerance $\varepsilon$.}
\STATE{\textit{Output}: Approximate eigenvalue $\lambda \in \mathbb{R}$ and eigenvector $\bm{x} \in \mathbb{R}^n$ for $\bm{A}$.}
\STATE{$\bm{x}_1 \gets \bm{A} \bm{x}_0 / \| \bm{A} \bm{x}_0 \|$, $\lambda_1 \gets \textbf{qdot}(\bm{x}_0, \bm{x}_{1}, \varepsilon) / \sqrt{\textbf{qdot}(\bm{x}_0, \bm{x}_0, \varepsilon)}$, $k \gets 1$ \hfill Initialize $\lambda$ and index}
\WHILE{$| \lambda_{k} - \lambda_{k-1} | > \tau$}
\STATE{$\bm{x}_{k+1} \gets \bm{A} \bm{x}_k$ \hfill Update guess for eigenvector}
\STATE{$c \gets \textbf{qdot}(\bm{x}_{k+1}, \bm{x}_{k+1}, \varepsilon)$ \hfill Approximate dot product required for normalization}
\STATE{$\bm{x}_{k+1} \gets \bm{x}_{k+1} / \sqrt{c}$ \hfill Normalize guess for eigenvector}
\STATE{$\lambda_{k+1} \gets \textbf{qdot}(\bm{x}_k, \bm{x}_{k+1}, \varepsilon) / c$ \hfill Update guess for eigenvalue}
\STATE{$k \gets k+1$ \hfill Increment $k$ \ }
\ENDWHILE
\end{algorithmic}
\caption{\textbf{APM}: Approximate Power Method Kernel with \textbf{qdot}}
\label{alg:apm}
\end{algorithm}
\end{center}
We fix the convergence tolerance of the power method to $\tau = \texttt{1e-6}$ and perform a maximum of $300$ iterations. 
%\purple{vary the \textbf{qdot} $\varepsilon$ from $\texttt{1e-16}$ to $\texttt{1e0}$ by powers of 10. }
Figures~\ref{fig:power-1e4} and \ref{fig:power-1e5} present various results for problems of dimension \texttt{1e4} and \texttt{1e5}, respectively.

For each problem dimension, we vary the \textbf{qdot} $\varepsilon$, and plot the corresponding number of iterations required for convergence. 
%we plot the number of iterations required for convergence versus the \textbf{qdot} tolerance value to visualize how the tolerance of \textbf{qdot} affects the APM. 
This allows us to identify the highest level of approximation that can be introduced for the same convergence rate as the \dpr \ precision Power Method application, which we will denote simply by PM. Figures~\ref{fig:power-1e4-a} and \ref{fig:power-1e5-a} illustrate the results. In Figures~\ref{fig:power-1e4-b} and \ref{fig:power-1e5-b}, we plot the corresponding absolute error between the eigenvalue estimate computed by APM and the eigenvalue estimate computed by PM. This illustrates the effect of $\varepsilon$ on the quality of the computed eigenvalue using APM. 
%\red{Note that to improve the presentation, the $y$-axis on these plots is split to use a log-scale for the larger absolute error values, and a symmetric log scale for small values.}
%into two parts. This is because log-scale is best for visualising the absolute error of the eigenvalue for larger \qdot \ tolerance values but for smaller \qdot \ tolerance values the absolute approximation error value is 0. Hence, the lower portion of the $y$-axis is in symlog-scale while the upper portion is in log-scale.
The results generally indicate that introducing more approximation by increasing the value of $\varepsilon$ can lead to differences between the computed eigenvalue using APM and PM. While this behavior generally makes sense, it is also understandable since the stopping criterion used in APM and PM is based on the difference between the eigenvalues computed at successive iterates of the respective algorithms (see Algorithm~\ref{alg:apm}).
%\purple{In particular, we found that even for \textbf{qdot} tolerance values for which the number of iterations of APM closely match the number of iterations required by PM, the eigenvalue computed by APM could differ from the eigenvalue computed by PM by more than the power method tolerance $\tau = \texttt{1e-6}$.}

%current and previous estimate for the eigenvalue. This is in contrast to the ACG experiments which did not degrade the quality of the solution when increasing the \textbf{qdot} tolerance as the ACG convergence criterion was satisfied for larger \textbf{qdot} tolerance values thereby allowing for higher level of approximation without compromising the quality of the solution. If \textbf{qdot} was more efficient than the \dpr \ dot product then it would be possible to realize speedup when computing the same quality eigenvalue by using APM instead of PM with a \textbf{qdot} tolerance of \texttt{1e-7}.

Figures~\ref{fig:power-1e4-c} -- \ref{fig:power-1e4-d} and Figures~\ref{fig:power-1e5-c} -- \ref{fig:power-1e5-d} present the distribution of the percentage of components at each precision, corresponding to $\varepsilon = \texttt{1e-7}$. This is the largest \textbf{qdot} tolerance that resulted in an eigenvalue estimate with an absolute error within the accepted tolerance $\tau = \texttt{1e-6}$.
%\purple{Figures~\ref{fig:power-1e4-e} -- \ref{fig:power-1e4-f} and Figures~\ref{fig:power-1e5-e} -- \ref{fig:power-1e5-f} present the distribution corresponding to the \textbf{qdot} tolerance that yielded the fewest number of iterations. This corresponds to $\varepsilon = \texttt{1e-4}$ and $\varepsilon = \texttt{1e-3}$ for problems with graphs of dimension $\texttt{1e4}$ and $\texttt{1e5}$, respectively.}
For these plots, note that we only show the nonzero percentages of the precisions that are represented in the components. Interestingly, the results indicate that a very low percentage of components are needed in \dpr \ precision after the initial iterations. 
%\purple{When the value of $\varepsilon$ is small, we observe a larger percentage of \spr \ precision components compared to when the value is larger. This is to be expected, since the larger value of $\varepsilon$ allows more approximation in the dot product.}

\begin{figure}[!tbp]
  \begin{subfigure}[b]{0.48\textwidth}
    \includegraphics[width=\textwidth]{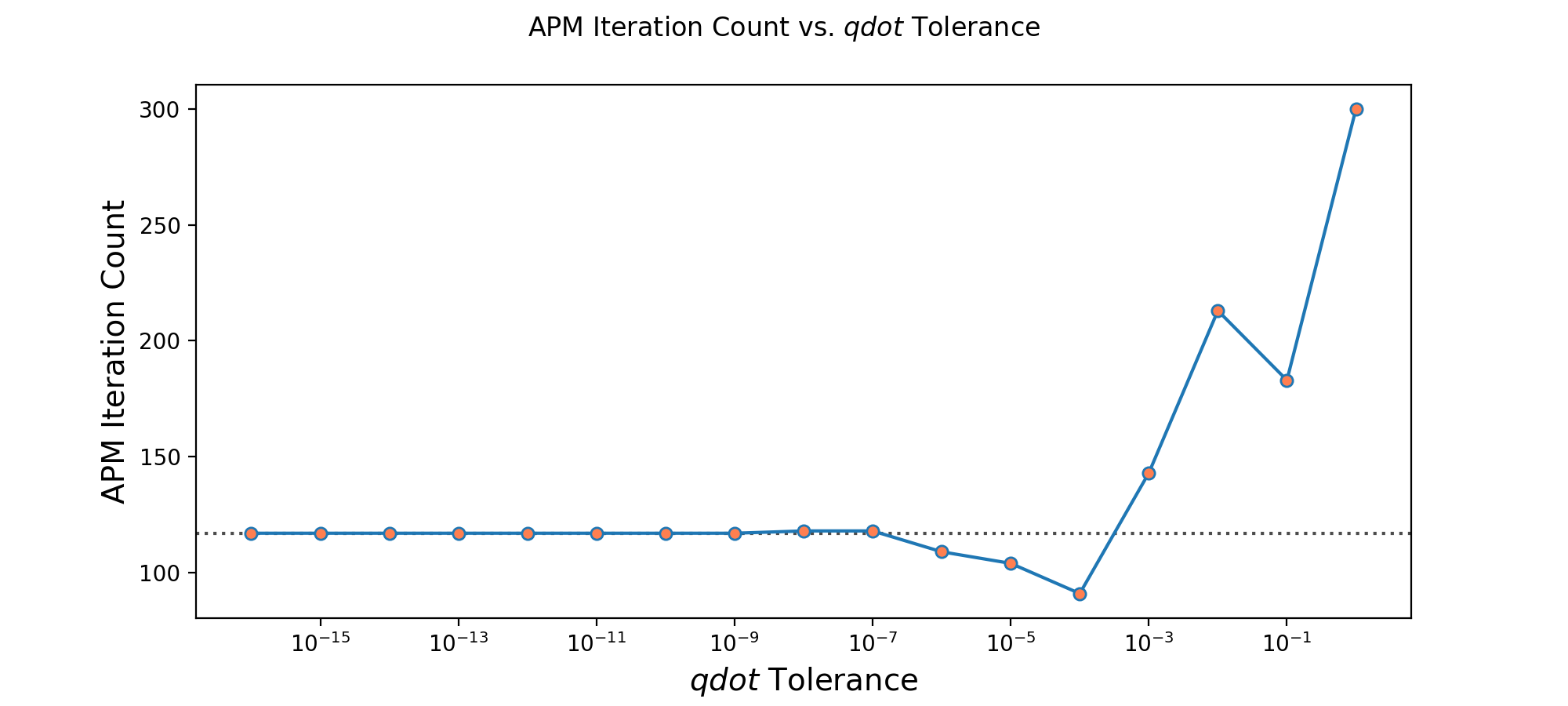}
    \caption{\textit{APM Total Iteration Count}. The dashed line indicates the iteration count of the \dpr \ precision Power Method used with the same graph Laplacian matrix.}
    \label{fig:power-1e4-a}
  \end{subfigure}
  \hfill
  \begin{subfigure}[b]{0.48\textwidth}
    \includegraphics[width=\textwidth]{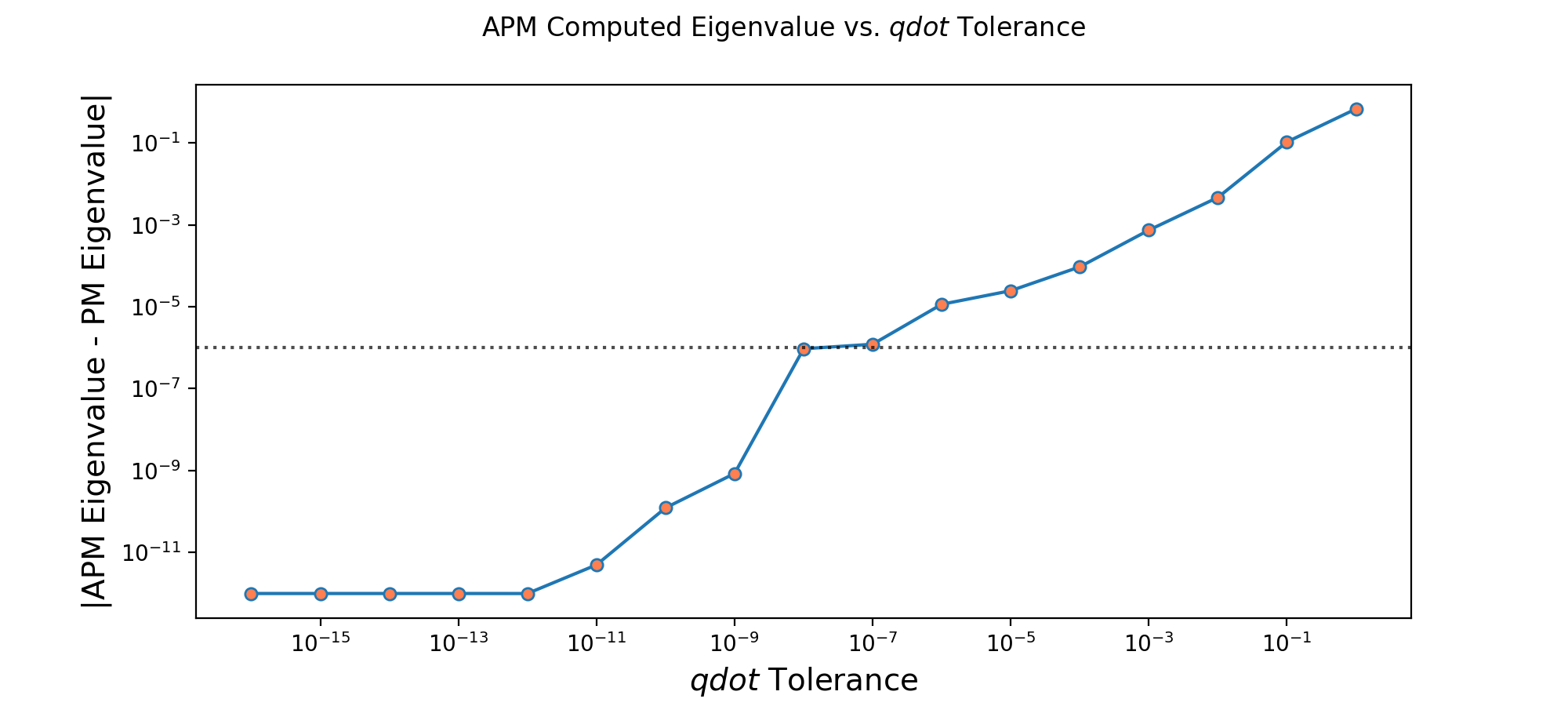}
    \caption{\textit{$|$APM Eigenvalue $-$ PM Eigenvalue$|$ vs. qdot Tolerance}. The dashed line represents the value of the Power Method convergence tolerance $tau$.}
    \label{fig:power-1e4-b}
  \end{subfigure}
  \hfill
  \begin{subfigure}[b]{0.48\textwidth}
  \includegraphics[width=\textwidth]{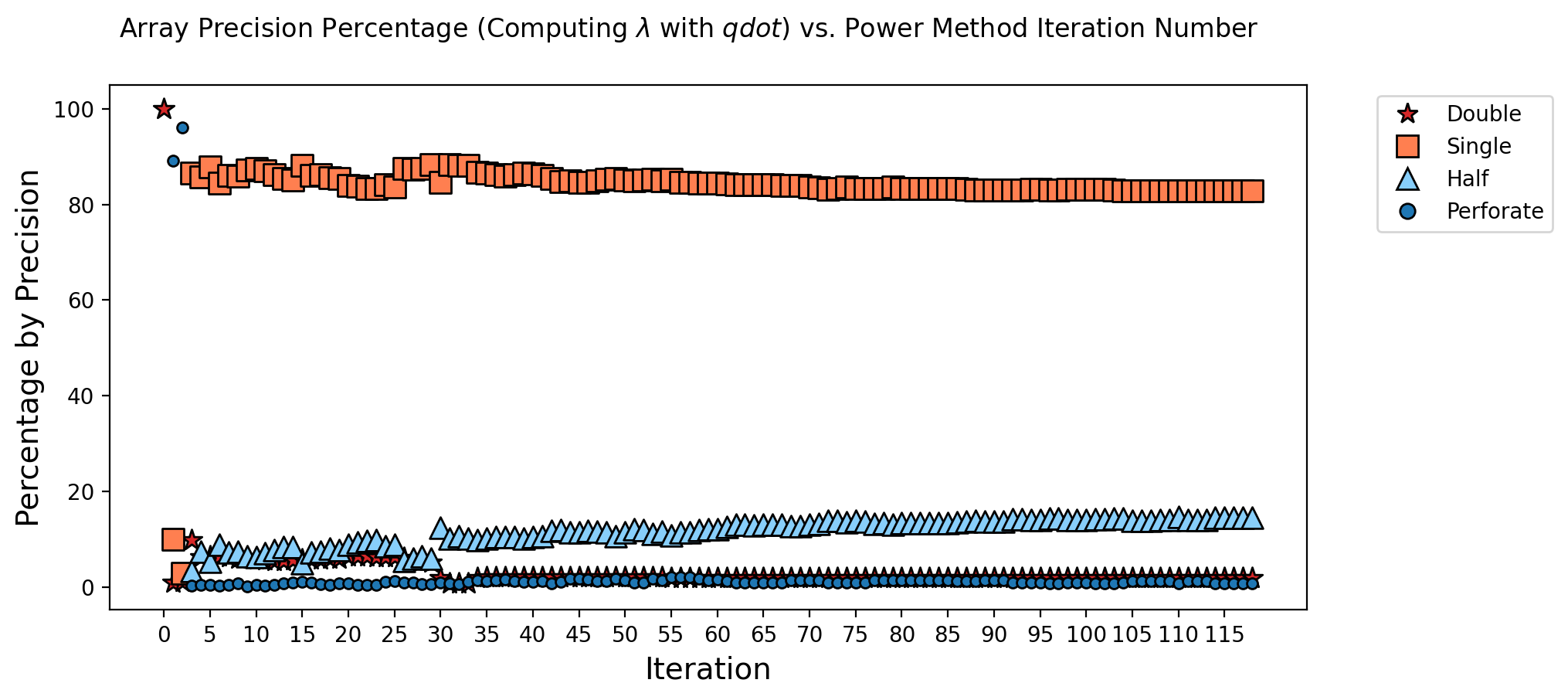}
    \caption{qdot precision percentages during computation of $\lambda$ in APM with qdot tolerance of $\texttt{1e-7}$.}
    \label{fig:power-1e4-c}
  \end{subfigure}
  \hfill
  \begin{subfigure}[b]{0.48\textwidth}
    \includegraphics[width=\textwidth]{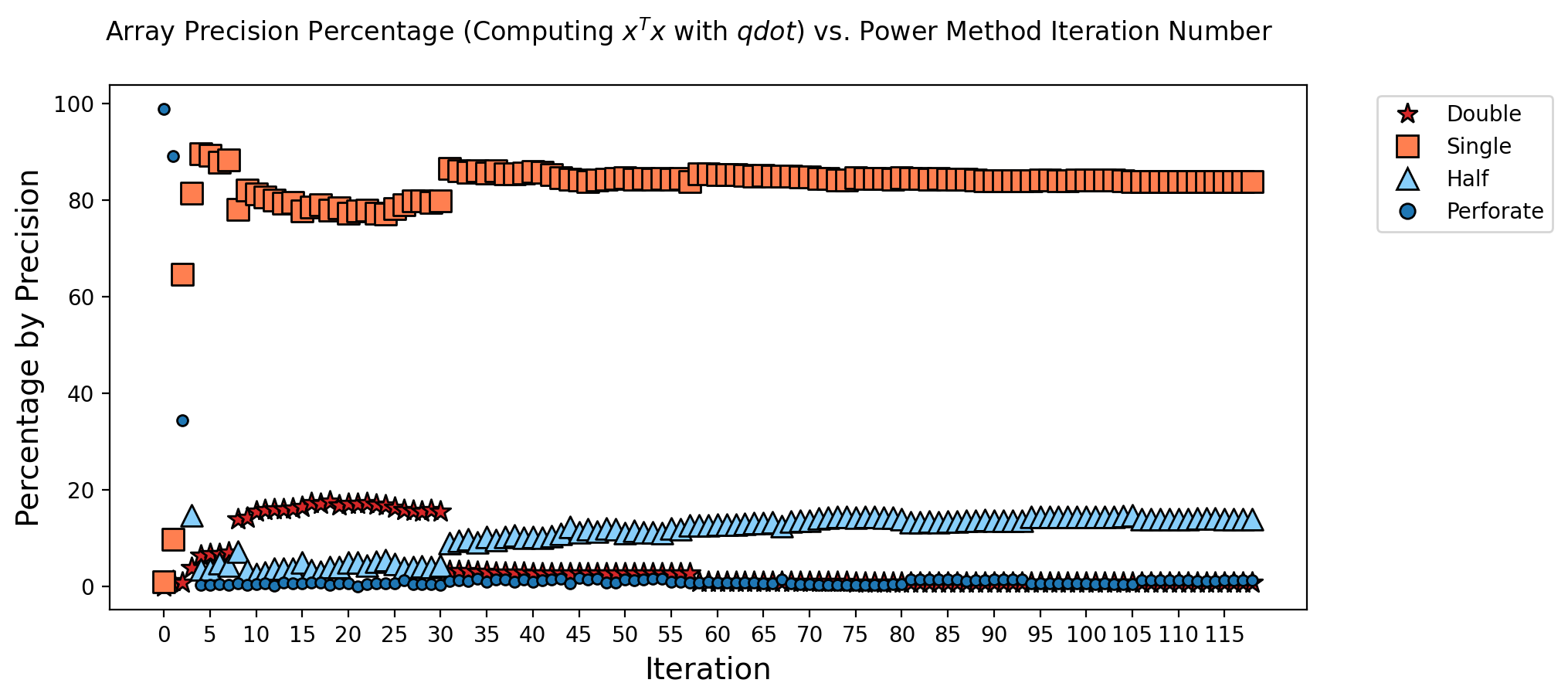}
    \caption{qdot precision percentages during computation of $\bm{x}^{\intercal} \bm{x}$ in APM with qdot tolerance of $\texttt{1e-7}$.}
    \label{fig:power-1e4-d}
  \end{subfigure}
\iffalse %%%%%%%%%%%%%%%%%%%%  
  \hfill
  \begin{subfigure}[b]{0.48\textwidth}
  \includegraphics[width=\textwidth]{figures/powermethod_residual_and_prec_percent_vs_iteration_n1000_qdottol0.0001_powertol1e-06_lambda_plot.png}
    \caption{qdot precision percentages during computation of $\lambda$ in ACG with qdot tolerance of $\texttt{1e-4}$.}
    \label{fig:power-1e4-e}
  \end{subfigure}
  \hfill
  \begin{subfigure}[b]{0.48\textwidth}
    \includegraphics[width=\textwidth]{figures/powermethod_residual_and_prec_percent_vs_iteration_n1000_qdottol0.0001_powertol1e-06_norm_plot.png}
    \caption{qdot precision percentages during computation of $\bm{x}^{\intercal} \bm{x}$ in ACG with qdot tolerance of $\texttt{1e-4}$.}
    \label{fig:power-1e4-f}
  \end{subfigure}
\fi %%%%%%%%%%%%%%%%%%  
  \caption{\textbf{Comparing Approximate Power Method with qdot to a \dpr \ Precision Power Method (PM) for Problem of Dimension $N=10000$}.}
  \label{fig:power-1e4}
\end{figure}

\begin{figure}[!tbp]
  \begin{subfigure}[b]{0.48\textwidth}
    \includegraphics[width=\textwidth]{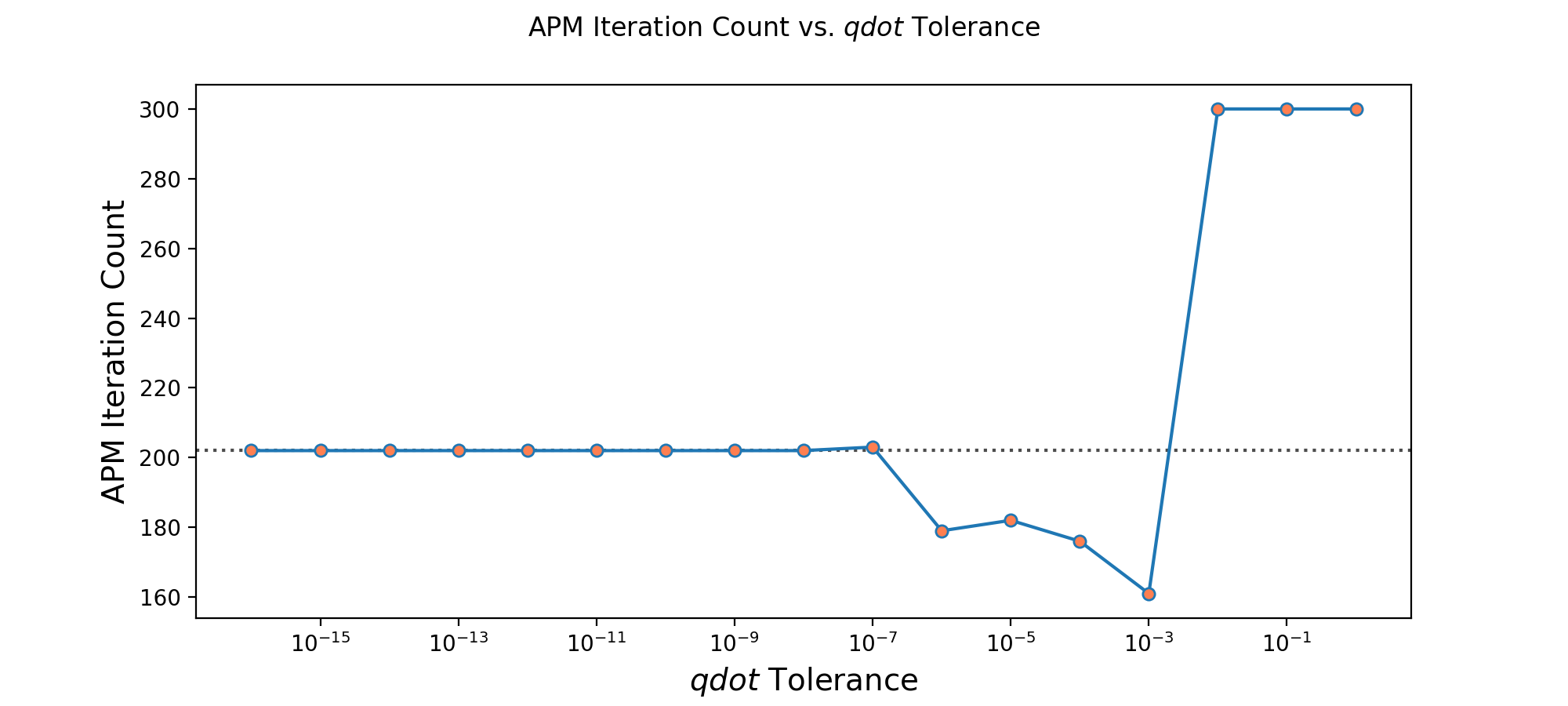}
    \caption{\textit{APM Total Iteration Count}. The dashed line indicates the iteration count of the \dpr \ precision Power Method used with the same graph Laplacian matrix.}
    \label{fig:power-1e5-a}
  \end{subfigure}
  \hfill
  \begin{subfigure}[b]{0.48\textwidth}
    \includegraphics[width=\textwidth]{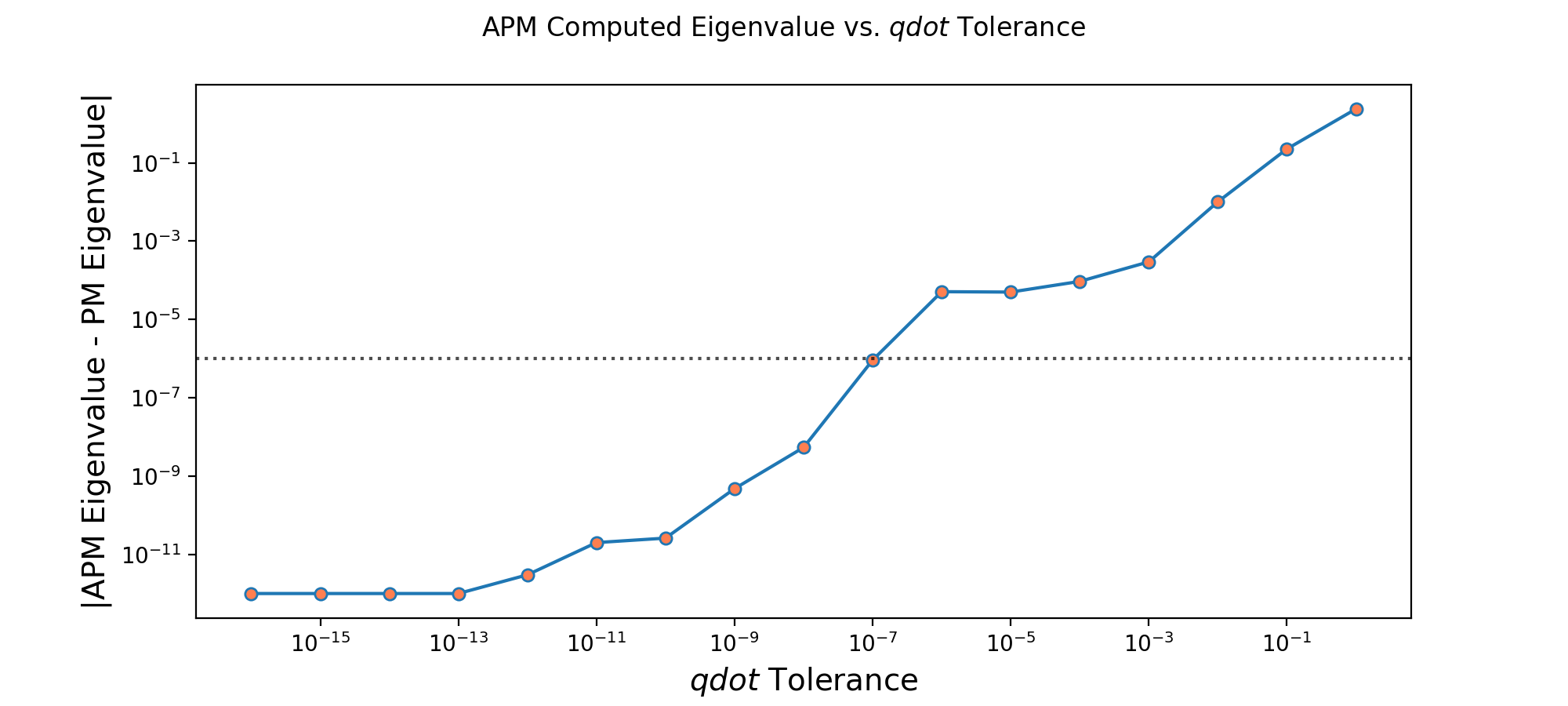}
    \caption{\textit{$|$APM Eigenvalue $-$ PM Eigenvalue$|$ vs. qdot Tolerance}. The dashed line represents the value of the Power Method convergence tolerance $tau$.}
    \label{fig:power-1e5-b}
  \end{subfigure}
  \hfill
  \begin{subfigure}[b]{0.48\textwidth}
  \includegraphics[width=\textwidth]{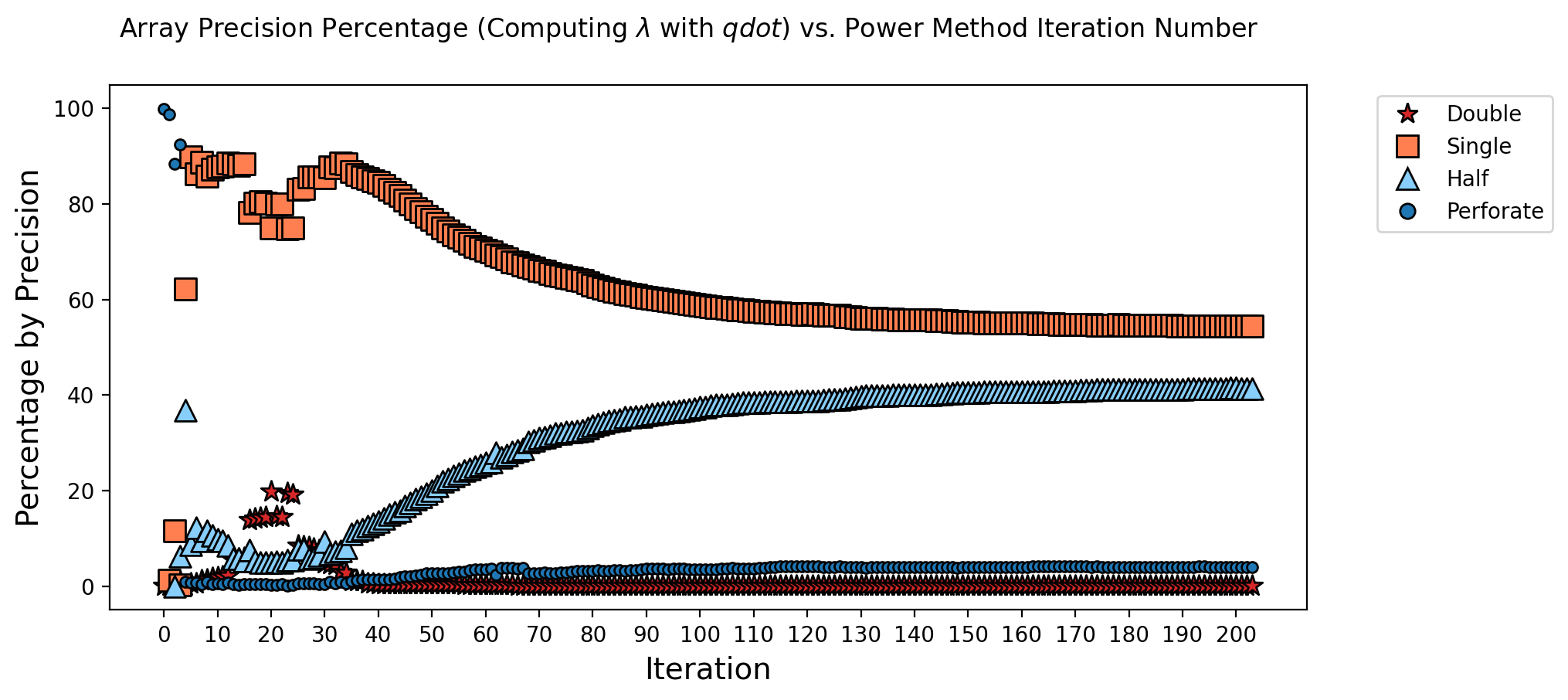}
    \caption{qdot precision percentages during computation of $\lambda$ in APM with qdot tolerance of $\texttt{1e-7}$.}
    \label{fig:power-1e5-c}
  \end{subfigure}
  \hfill
  \begin{subfigure}[b]{0.48\textwidth}
    \includegraphics[width=\textwidth]{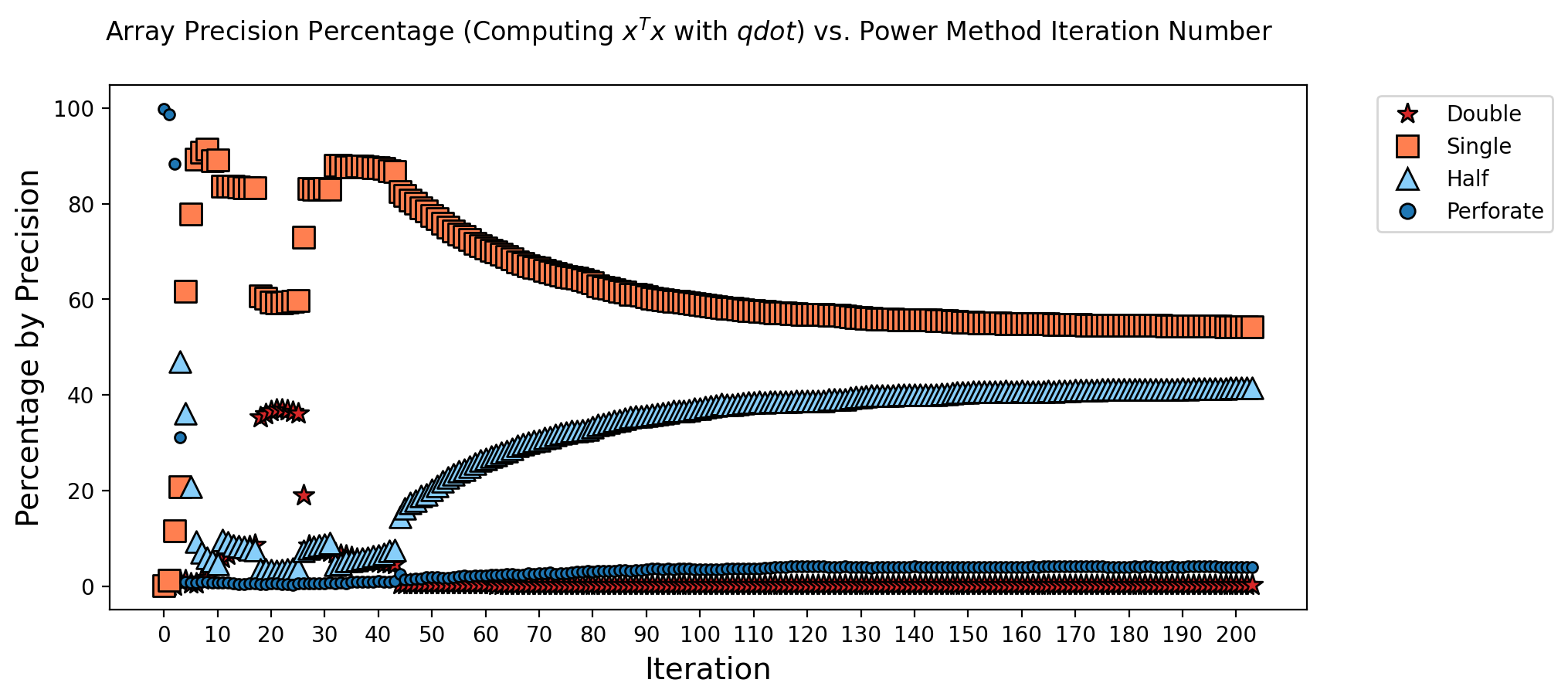}
    \caption{qdot precision percentages during computation of $\bm{x}^{\intercal} \bm{x}$ in APM with qdot tolerance of $\texttt{1e-7}$.}
    \label{fig:power-1e5-d}
  \end{subfigure}
\iffalse %%%%%%%%%%%%%%%%%%%%  
  \hfill
  \begin{subfigure}[b]{0.48\textwidth}
  \includegraphics[width=\textwidth]{figures/powermethod_residual_and_prec_percent_vs_iteration_n10000_qdottol0.001_powertol1e-06_lambda_plot.png}
    \caption{qdot precision percentages during computation of $\lambda$ in ACG with qdot tolerance of $\texttt{1e-3}$.}
    \label{fig:power-1e5-e}
  \end{subfigure}
  \hfill
  \begin{subfigure}[b]{0.48\textwidth}
    \includegraphics[width=\textwidth]{figures/powermethod_residual_and_prec_percent_vs_iteration_n10000_qdottol0.001_powertol1e-06_norm_plot.png}
    \caption{qdot precision percentages during computation of $\bm{x}^{\intercal} \bm{x}$ in ACG with qdot tolerance of $\texttt{1e-3}$.}
    \label{fig:power-1e5-f}
  \end{subfigure}
\fi  %%%%%%%%%%%%%%%%%%%%%
  \caption{\textbf{Comparing Approximate Power Method with qdot to a \dpr \ Precision Power Method (PM)} for Problem of Dimension $N=100000$. }
  \label{fig:power-1e5}
\end{figure}

%%%%%%%%%%%%%%%%%%%%%%%%%%%%%%%%%%%%%%%%%
\section{Conclusion}
In this paper, we formalized a general framework for designing error bounded approximate computing kernels, and applied this framework to the dot product kernel to design \textbf{qdot}. We theoretically proved and empirically demonstrated that \textbf{qdot} bounds the relative error introduced by the approximation. Numerical experiments on the Conjugate Gradient and Power Method algorithms demonstrate that high levels of approximation can be introduced into these algorithms without degrading their performance.
%\purple{using a linear solver and eigenvalue computation application demonstrate that high levels of approximation can be introduced into these applications without degrading the performance of each application. }

While \textbf{qdot} was demonstrated to have good approximation effectiveness, achieving this effectiveness comes at the cost of a parameter selection phase that limits the approximation efficiency in practice. However, for applications that can tolerate high levels of approximation in the dot product, such as those presented in this paper, one can design a more efficient strategy at some cost to the effectiveness. One strategy is to consider a blocked approach that targets the least efficient step in our implementation -- sorting the array of exponent values -- to improve efficiency. 
%\green{This strategy improves the complexity of the sorting and binning steps to $O(n/b + e_{\max} - e_{\min} + 1)$, where $b$ denotes the size of the block.}
Future work could explore finding a better balance between approximation effectiveness and efficiency so that \textbf{qdot} could be practically used at runtime.
%\purple{without incurring an overhead cost to the original kernel. }
If successful, this could lead to an approximate strategy for the matrix-vector multiplication kernel that balances effectiveness and efficiency. Furthermore, the formalized general framework presented in this paper could be used to design other error-bounded approximate computing strategies for kernels other than the dot product.

This work was performed under the auspices of the U.S. Department of Energy by Lawrence Livermore National Laboratory under Contract DE-AC52-07NA27344. Work at LLNL was funded by the Laboratory Directed Research and Development Program under project tracking code 20-ERD-043. 

\bibliographystyle{siamplain}
\bibliography{qdot_paper}

\end{document}